\documentclass[a4paper,12pt]{amsart}

\usepackage{url}
\usepackage{enumerate}
\usepackage{color}
\usepackage{multirow}
\usepackage{hhline}
\usepackage{amsmath}
\usepackage{amsthm}
\usepackage{amsfonts}
\usepackage{amssymb}
\usepackage{times}
\usepackage{mathptm}
\usepackage{array}
\usepackage{graphics}
\usepackage{epsfig}
\usepackage{mathrsfs}

\newcommand{\subg}[2]{\langle #1\rangle_{#2}}
\newcommand{\subgg}[2]{\langle\langle #1\rangle\rangle_{#2}}

\newcommand{\Gcenter}[2]{
	\dimen0=\ht\strutbox%
	\advance\dimen0\dp\strutbox%
	\multiply\dimen0 by#1%
	\divide\dimen0 by2%
	\advance\dimen0 by-.5\normalbaselineskip
	\raisebox{-\dimen0}[0pt][0pt]{#2}}%

\def\h{\mathscr{H}}

\def\x{\mathscr{X}}

\def\f{\mathscr{F}}

\def\X{\tilde{X}}
\def\Y{\tilde{Y}}

\newtheorem{thm}{Theorem}[section]
\newtheorem{cl}{Claim}
\newtheorem{lm}[thm]{Lemma}
\newtheorem{prop}[thm]{Proposition}

\theoremstyle{definition}
\newtheorem{df}[thm]{Definition}

\newenvironment{proofclaim*}{\paragraph{\it Proof of Claim.\ }}
		{\hfill$\Box$\\}

\title[On Graphs with the Smallest Eigenvalue at Least $-1-\sqrt{2}$,
	part II]
	{On Graphs with the Smallest Eigenvalue at Least $-1-\sqrt{2}$,
	part II}

\author{Tetsuji Taniguchi}

\address[T.~Taniguchi]{Matsue~College~of~Technology,
   Nishiikuma-cho~ 14-4,
   Matsue,
   Shimane~690-8518,
   Japan}
\email{tetsuzit@matsue-ct.ac.jp}
\keywords{Generalized line graph, Spectrum}
\subjclass[2000]{05C50}

\date{\today}

\begin{document}

\begin{abstract}
This is a continuation of the article with the same title.
In this paper,
	the family $\h$ is the same as in the previous paper \cite{paperI}.
The main result is that
	a minimal graph which is not an $\h$-line graph,
	is just isomorphic to one of the $38$ graphs found by computer.
\end{abstract}

\maketitle

\section{INTRODUCTION}
In the previous paper \cite{paperI},
	we proved the uniqueness of strict $\{[H_2],[H_3],[H_5]\}$-cover graphs.
This result plays a crucial role in obtaining an upper bound
	on the number of vertices in a minimal forbidden subgraph.

In this paper,
	we completely determine minimal forbidden subgraphs
	for the class of slim $\{[H_2],[H_3],[H_5]\}$-line graphs.
By computer,
	we obtain such graphs (cf. Figure~\ref{MFS}).
The smallest eigenvalue of the minimal forbidden subgraph $G_{5,2}$
	is less than $-1-\sqrt{2}$,
	and others are greater than or equal to $-1-\sqrt{2}$.
We know
	that the smallest eigenvalues of $\{[H_2],[H_3],[H_5]\}$-line graphs
	are greater than or equal to $-1-\sqrt{2}$
	(cf. Theorem 3.7 of \cite{hlg}).
These mean that,
	if a graph does not contain subgraphs in Figure~\ref{MFS},
	then it is a slim  $\{[H_2],[H_3],[H_5]\}$-line graph,
	and has the smallest eigenvalue at least $-1-\sqrt{2}$.

We use the same notation as in \cite{paperI}.

\begin{df}\label{df:Hg}
A \emph{Hoffman graph} is a graph $H$ with vertex labeling
$V(H)\to\{s,f\}$, satisfying the following conditions:
\begin{enumerate}
\item every vertex with label $f$ is adjacent to at least
	one vertex with label $s$;
\item vertices with label $f$ are pairwise non-adjacent.
\end{enumerate}
We call a vertex with label $s$ a \emph{slim vertex},
	and a vertex with label $f$ a \emph{fat vertex}.
We denote by $V_s(H)$ ($V_f(H)$) the set of slim (fat) vertices of $H$.
An ordinary graph without labeling can be regarded as a Hoffman graph
	without fat vertex.
Such a graph is called a \emph{slim graph}.
The subgraph of a Hoffman graph $H$ induced on $V_s(H)$
	is called the \emph{slim subgraph} of $H$.
We draw Hoffman graphs by depicting vertices as large (small) black dots
	if they are fat (slim).
\end{df}

\begin{figure}[h]
\caption{}
\begin{tabular}{ccccc}
\includegraphics[scale=0.11]
{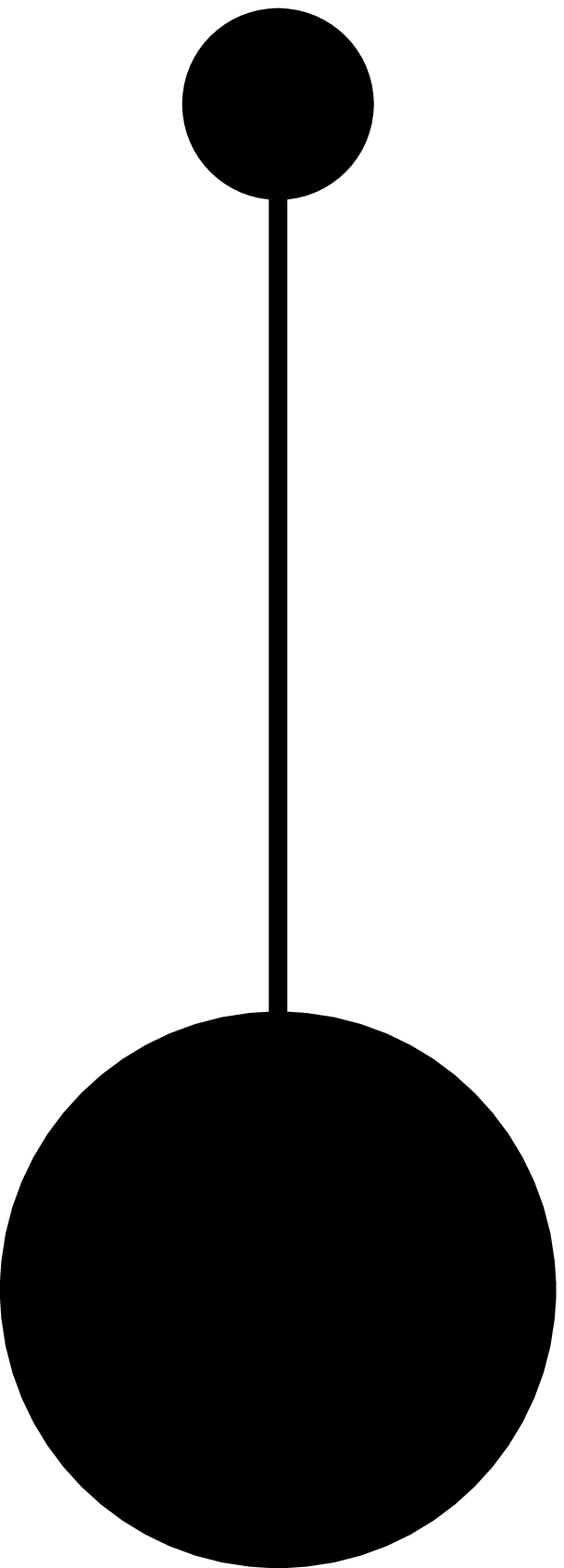} &
\includegraphics[scale=0.11]
{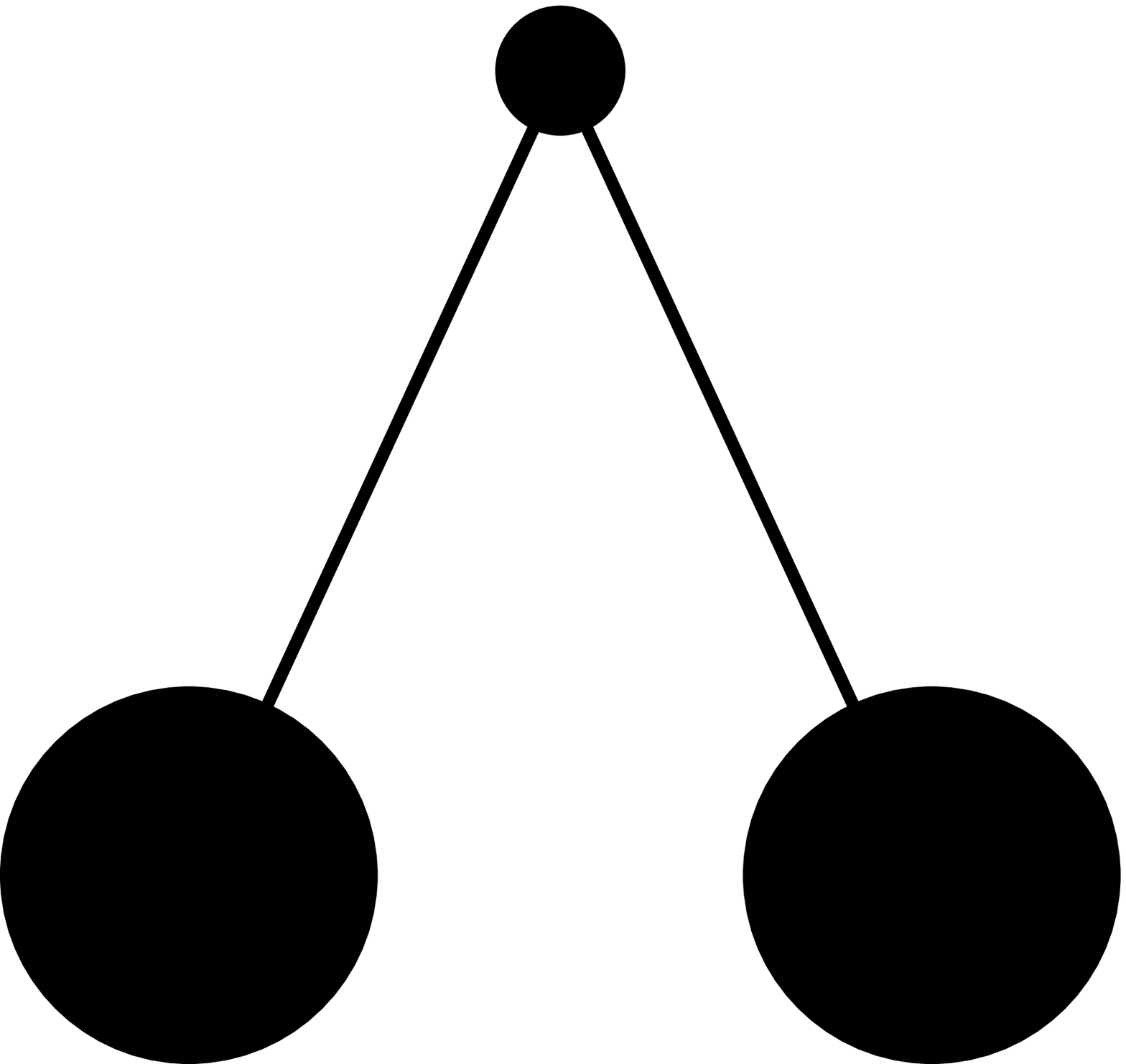} &
\includegraphics[scale=0.11]
{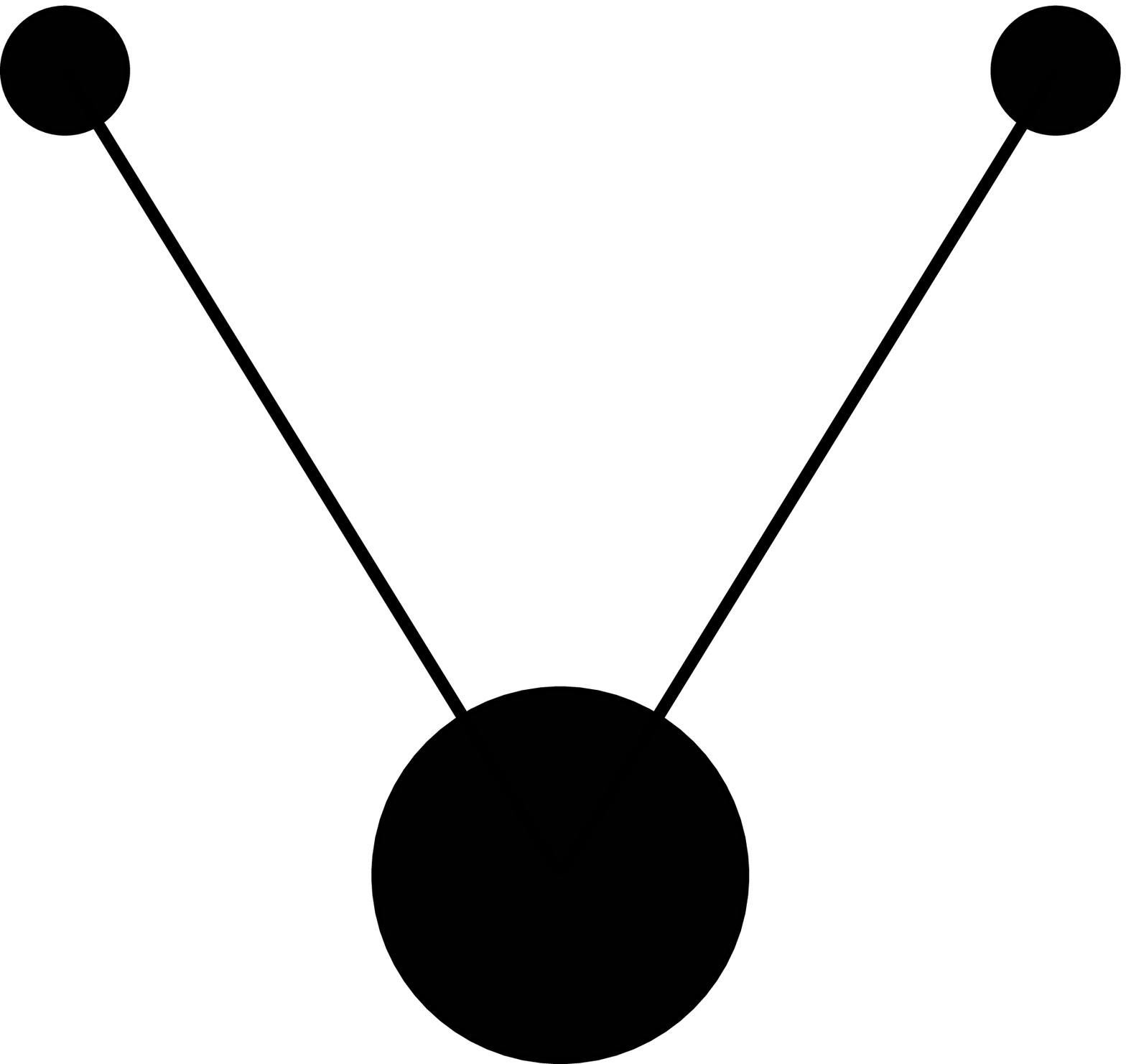} &
\includegraphics[scale=0.11]
{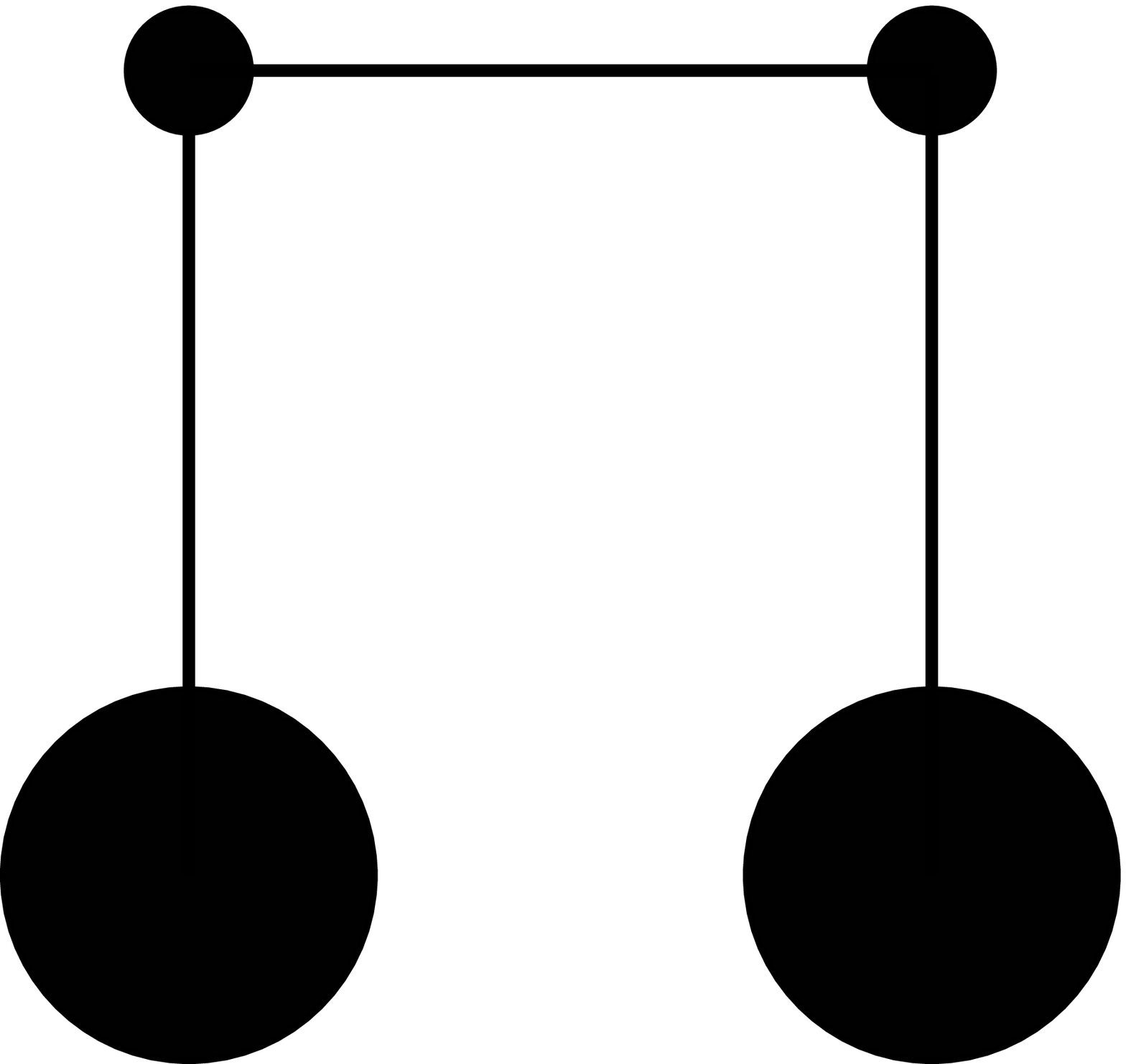} &
\includegraphics[scale=0.11]
{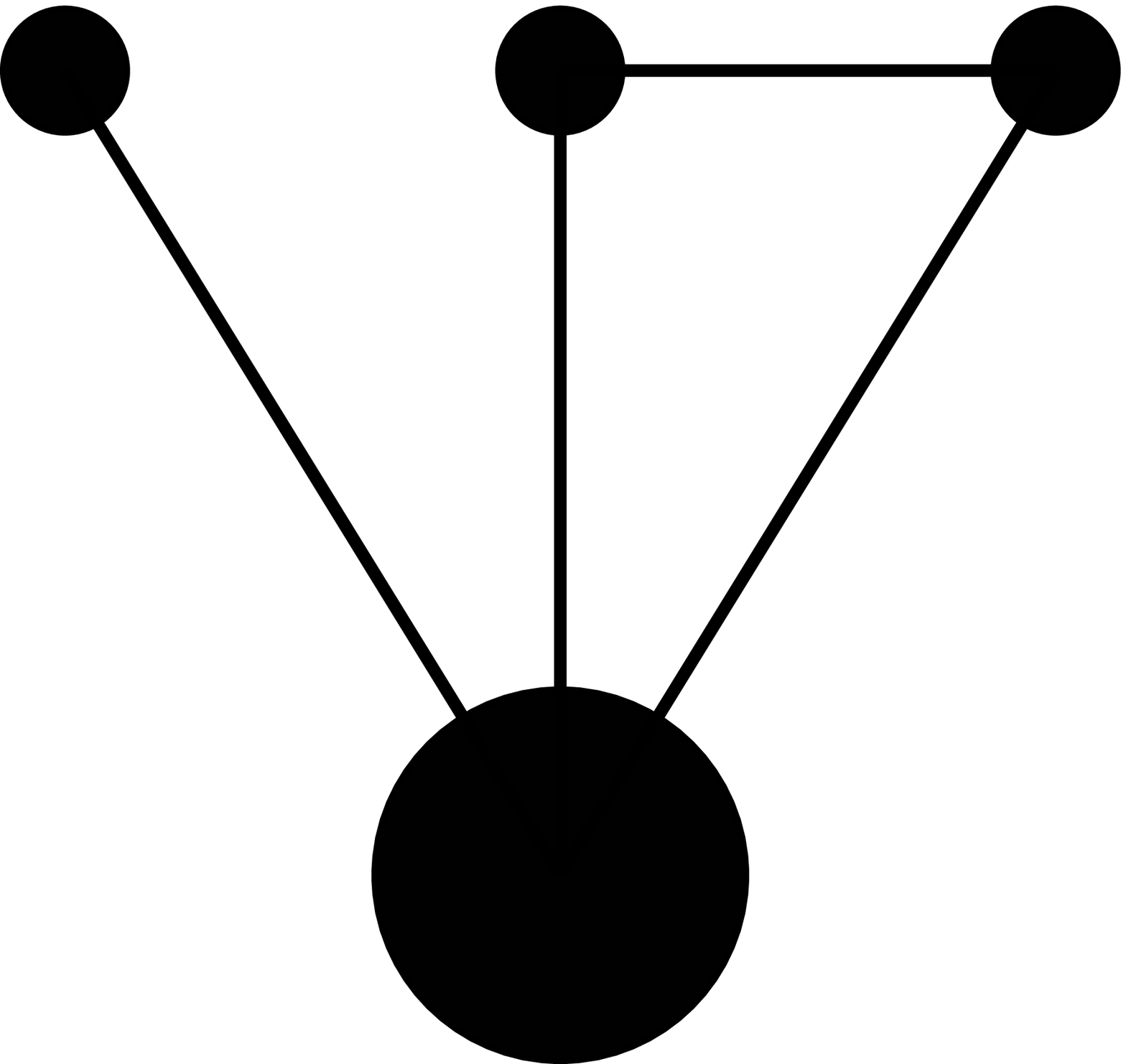}
 \\
$H_1,\lambda_{\min}=\alpha_1$ & $H_2,\lambda_{\min}=\alpha_2$
 & $H_3,\lambda_{\min}=\alpha_2$ & $H_4,\lambda_{\min}=\alpha_2$
 & $H_5,\lambda_{\min}=\alpha_3$ \\
&&&&\\

\includegraphics[scale=0.11]
{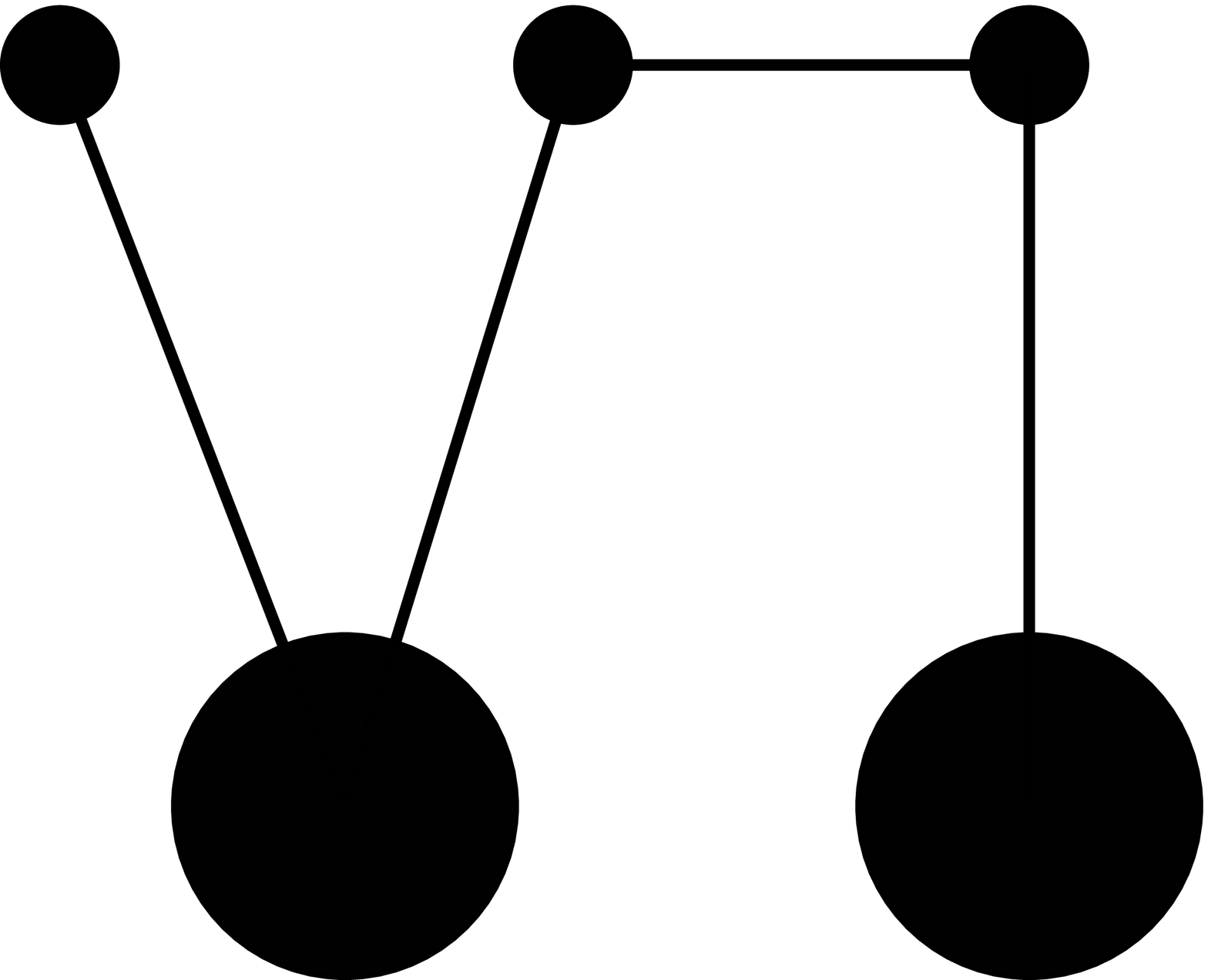} &
\includegraphics[scale=0.11]
{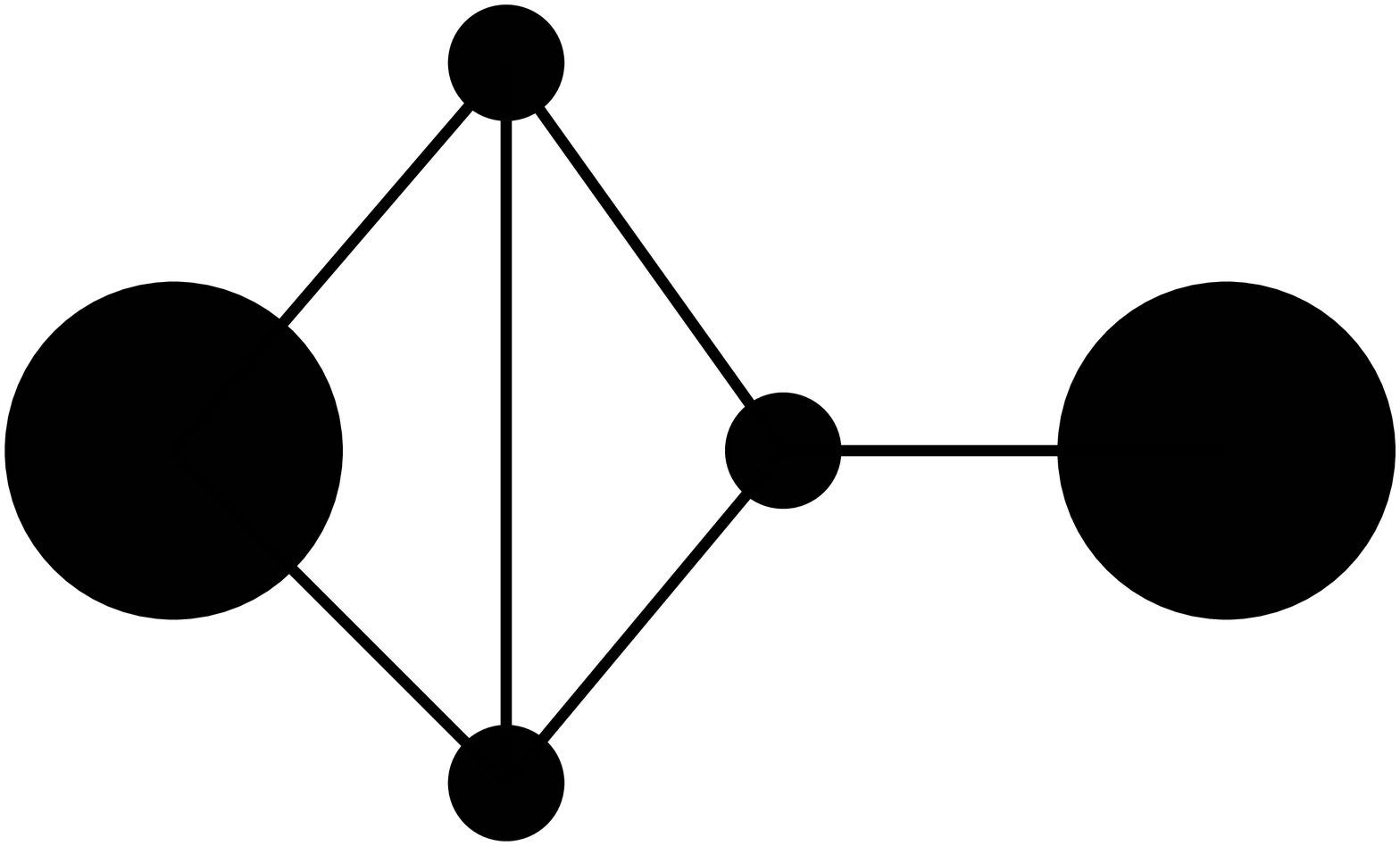} &
\includegraphics[scale=0.11]
{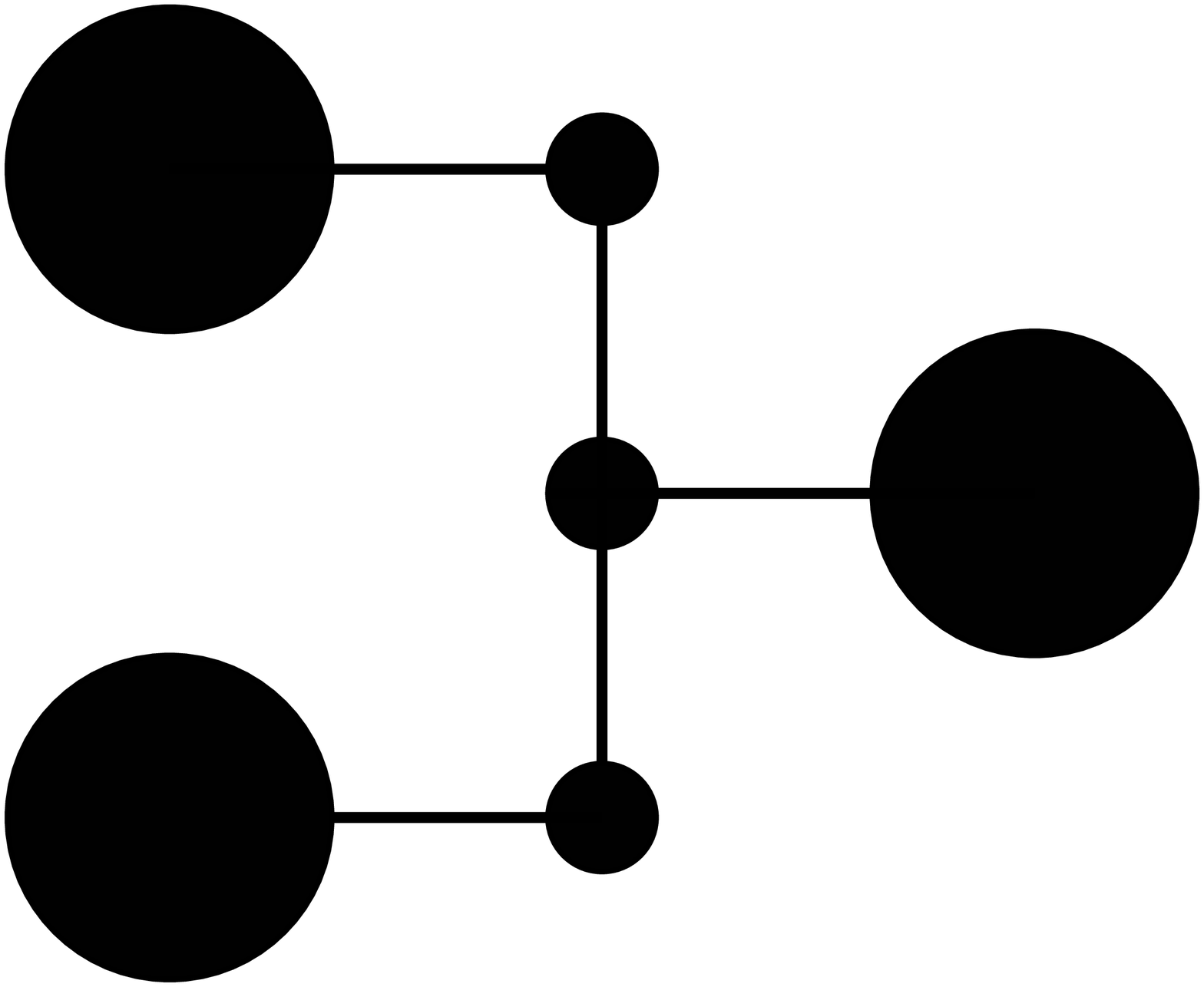} &
\includegraphics[scale=0.11]
{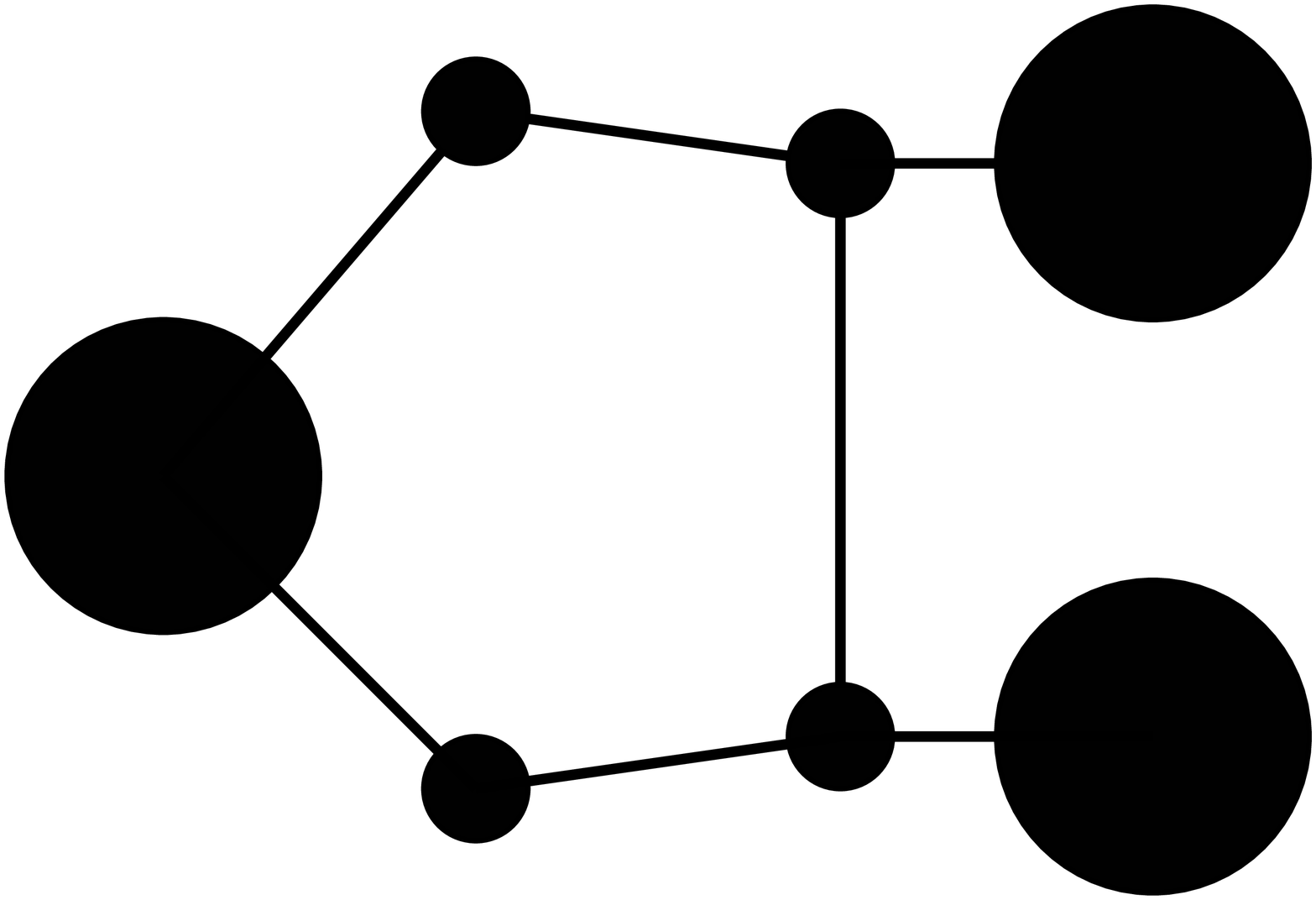}
 \\
$H_6,\lambda_{\min}=\alpha_3$ & $H_7,\lambda_{\min}=\alpha_3$
 & $H_8,\lambda_{\min}=\alpha_3$ & $H_9,\lambda_{\min}=\alpha_3$ & \\

\end{tabular}
\label{Hoffmans}
\end{figure}

We denote by $[H]$ the isomorphism class of Hoffman graphs containing $H$.
In the following,
	all graphs considered are Hoffman graphs
	and all subgraphs considered are induced subgraphs.
For a vertex $v$ of a Hoffman graph $H$,
	we denote by
	$N^s_H(v)$ (resp.\ $N^f_H(v)$)
	the set of all slim (resp.\ fat) neighbours of $v$,
	and by $N_H(v)$ the set of all neighbours of $v$,
	i.e., $N_H(v)=N^s_H(v)\cup N^f_H(v)$.
We write $G\subset H$ if $G$ is an induced subgraph of $H$.
We denote by $\langle S\rangle_H$ the subgraph of $H$ induced
	on a set of vertices $S$.
For a Hoffman graph $H$ and a subset $S\subset V_s(H)$,
	let $\subg{\subg{S}{}}{H}$ denote the subgraph
\[
	\subg{\subg{S}{}}{H}=\subg{S\cup(\bigcup_{z\in S}N_H^f(z))}{H}.
\]
Also,
	define $H-S$,
	$H-x$ by $H-S=\subg{\subg{V_s(H)\setminus S}{}}{H}$,
	$H-x=H-\{x\}$,
	respectively,
	where $x\in V(H)$.
Let $\emptyset$ be an empty set,
	and let $\phi$ be an empty graph.

\begin{df}\label{df:1}
Let $H$ be a Hoffman graph, and let $H^i$ $(i=1,2,\ldots,n)$ be
a family of subgraphs of $H$. The graph $H$ is said to
be the {\em sum}\/ of $H^i$ $(i=1,2,\ldots,n)$, denoted
\begin{equation}\label{eq:lem0a}
H=\biguplus_{i=1}^n H^i,
\end{equation}
if the following conditions are satisfied:
\begin{enumerate}[(i)]
\item $V(H)=\bigcup_{i=1}^n V(H^i)$;
\item $V_s(H^i)\cap V_s(H^j)=\emptyset$ if $i\neq j$;
\item if $x\in V_s(H^i)$ and $y\in V_f(H)$ are adjacent,
	then $y\in V(H^i)$;
\item if $x\in V_s(H^i)$, $y\in V_s(H^j)$ and $i\neq j$,
then $x$ and $y$ have at most one common fat neighbour,
and they have one if and only if they are adjacent.
\end{enumerate}
\end{df}

\begin{df}\label{df:linegraph}
Let $\h$ be a family of isomorphism classes of Hoffman graphs.
An $\h$-line graph $\Gamma$ is a subgraph of a graph $H=\biguplus_{i=1}^n H^i$
	such that $[H^i]\in\h$ for all $i\in\{1,2,\ldots,n\}$.
In this case,
	we call $H$ an $\h$-cover graph of $\Gamma$.
If $V_s(\Gamma)=V_s(H)$,
	then we call $H$ a strict $\h$-cover graph of $\Gamma$.
Two strict $\h$-covers $K$ and $L$ of $\Gamma$ are called equivalent,
	if there exists an isomorphism $\varphi:K\to L$
	such that $\varphi |_\Gamma$
	is the identity automorphism of $\Gamma$.
\end{df}

For the remainder of this section,
	we assume $\h=\{[H_2],\ [H_3],\ [H_5]\}$ (cf. Figure~\ref{Hoffmans}).
In our previous paper \cite{paperI},
	we proved the following theorem:
\begin{thm}
Let $\Gamma$ be a connected slim $\h$-line graph with
	at least 8 vertices.
Then a strict $\h$-cover graph
	of $\Gamma$
	is unique up to equivalence.
\label{cover}
\end{thm}

Every subgraph of an $\h$-line graph
   is an $\h$-line graph.
Thus,
   it is desirable to
   determine all minimal slim non $\h$-line graphs.
If $\Gamma$ is a minimal slim non $\h$-line graph with at least $9$ vertices,
	then we can use Theorem~\ref{cover}
	to derive a contradiction
	(refer to Section \ref{proofofmain} for the details of the proof).
Enumerating all the slim non $\h$-line graphs with
   at most $8$ vertices by comupter,
   we obtain the following theorem
	which is the main result in this paper: 
\begin{thm}
If $\Gamma$ is a minimal slim non $\h$-line graph,
	then $\Gamma$ is isomorphic to one of the graphs
	in Figure~\ref{MFS}.
\label{mainthm}
\end{thm}

\section{FORBIDDEN GRAPHS FOUND BY COMPUTER SEARCH}

In this section,
	we assume $\h=\{[H_2],\ [H_3],\ [H_5]\}$ (cf. Figure~\ref{Hoffmans}).
Proposition \ref{prop1} is the main result in this section.
It is very hard to obtain the propositions without computer search.
In this paper,
        we have computed
                by the software MAGMA \cite{MAGMA}.
In order to prove the propositions,
        we show some lemmas.

Let $\x_n$ be the family of isomorphism classes of
        connected slim graphs with $n$ vertices.
Brendan McKay gives collections of simple graphs
	on his web site (cf. \cite{bdm}).
From the data on this web site,
	we can generate $\x_n$.
Let $S_n$ be the family of isomorphism classes of
        connected slim $\h$-line graphs with $n$ vertices.
By computer, we obtain
\begin{equation}\label{eq:nonHLGs}
	\x_n=S_n~(n=1,2,3,4)
	\text{ and }
	\x_5\setminus S_5=\{[G_{5,1}],[G_{5,2}]\}
	\text{ (cf. Figure~\ref{MFS})}.
\end{equation}
We define $\f_n$ to be the family of isomorphism classes of
	minimal slim non $\h$-line graphs with $n$ vertices.
From (\ref{eq:nonHLGs}),
	$\f_i=\emptyset~(i=1,2,3,4)
	\mbox{ and }
	\f_5=\{[G_{5,1}],[G_{5,2}]\}$.
Removing those graphs which contain $G_{5,1}$ or $G_{5,2}$
	from $\x_6\setminus S_6$,
	we obtain $\f_6=\{[G_{6,i}]|\ i=1,2,\ldots,28\}$.
Similarly we obtain
		$\f_7=\{[G_{7,i}]|\ i=1,2,\ldots,7\}$,
		$\f_8=\{[G_{8,1}]\}$,
		and $\f_9=\emptyset$
	(cf. Figure~\ref{MFS}).
Hence the following proposition holds:
\begin{prop}
Let $\Gamma$ be a minimal slim non $\h$-line graph.
If $|V(\Gamma)|\leq 9$,
	then $[\Gamma]\in\f_5\cup\f_6\cup\f_7\cup\f_8$.
\label{prop1}
\end{prop}

Actually,
	the conclusion of the proposition holds without the assumption
	$|V(\Gamma)|\leq 9$.

\begin{figure}
\caption{}
{\small
\begin{tabular}{ccccc}
&&&&\\
\includegraphics[scale=0.11]{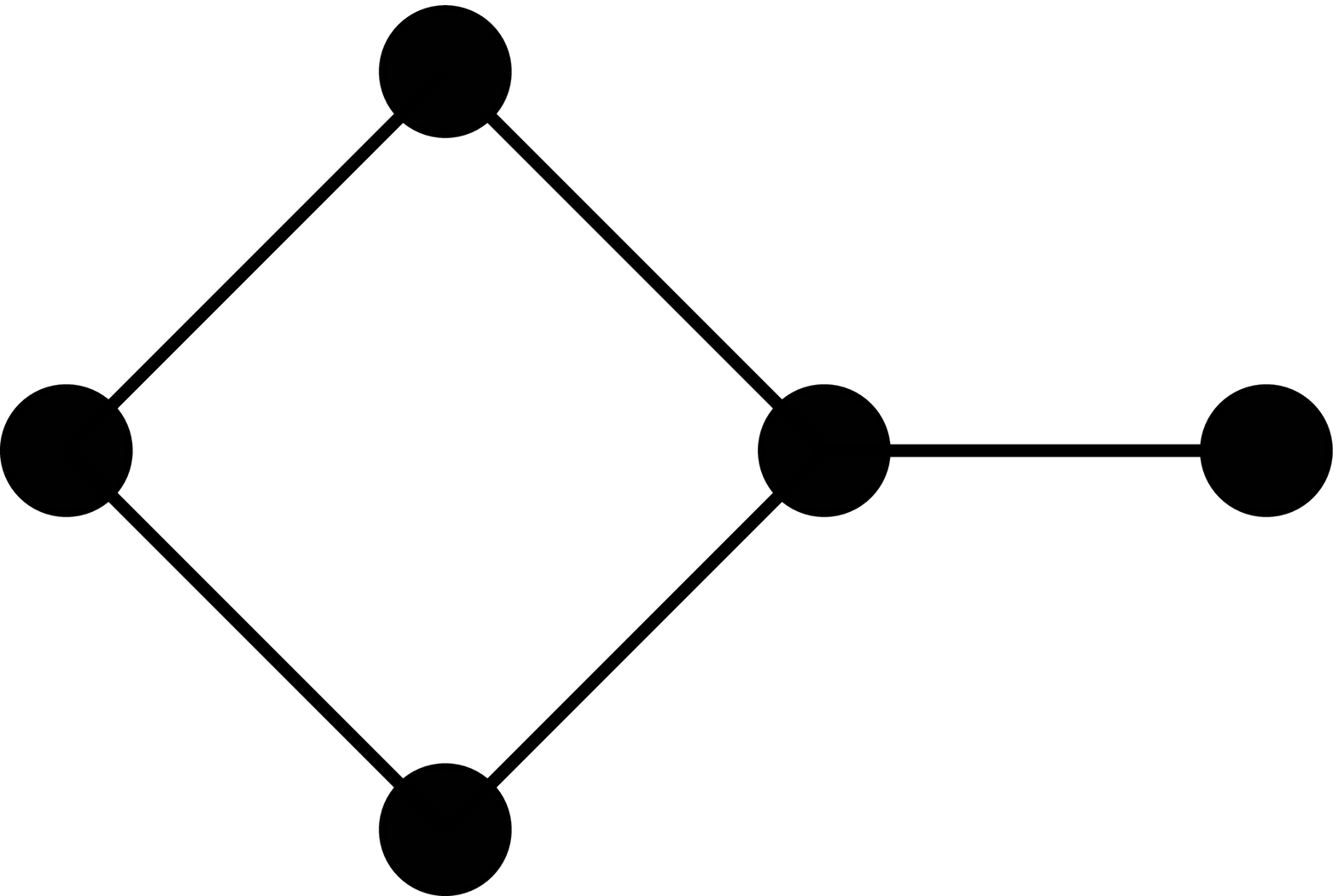} &
\includegraphics[scale=0.11]{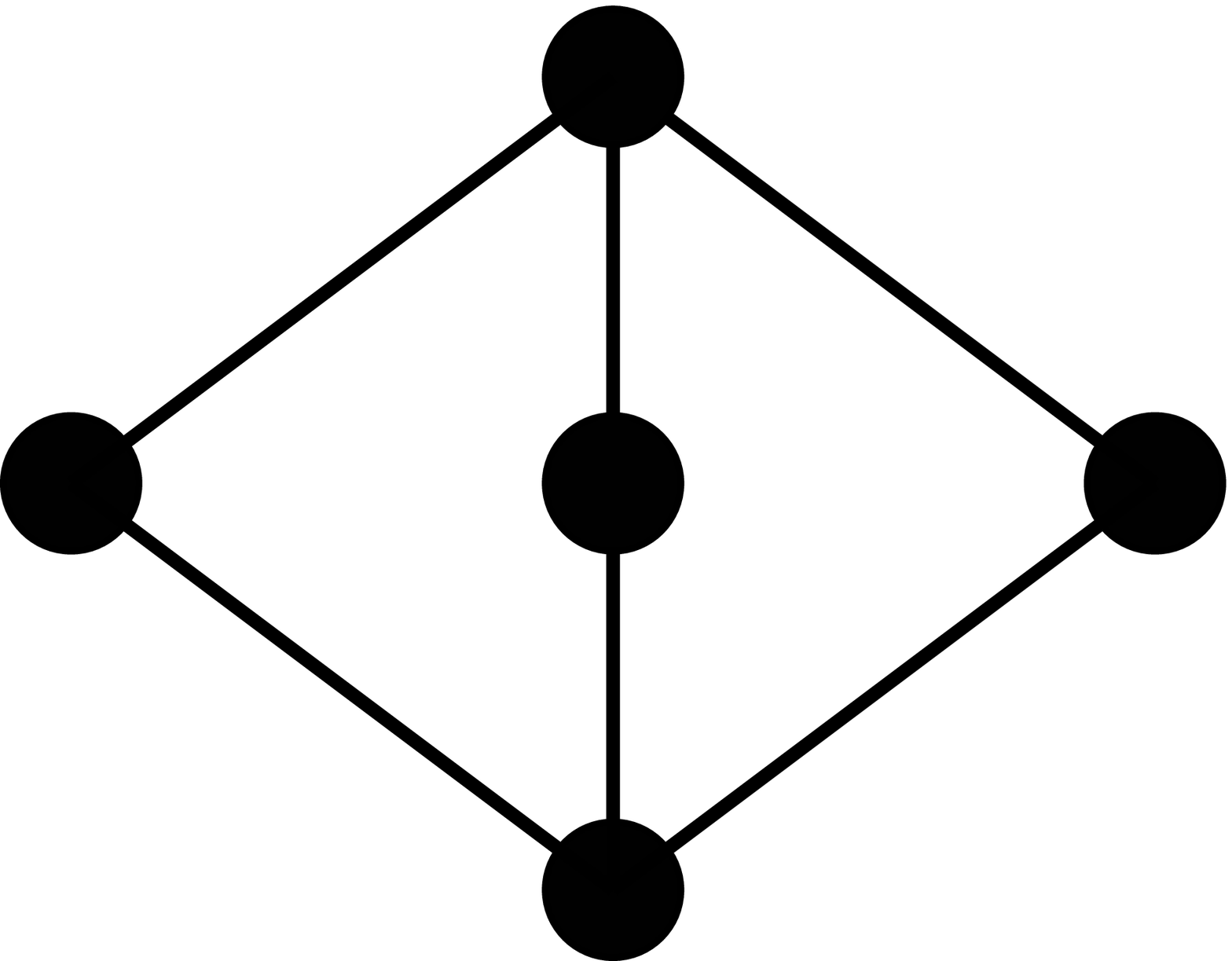} &
\includegraphics[scale=0.11]{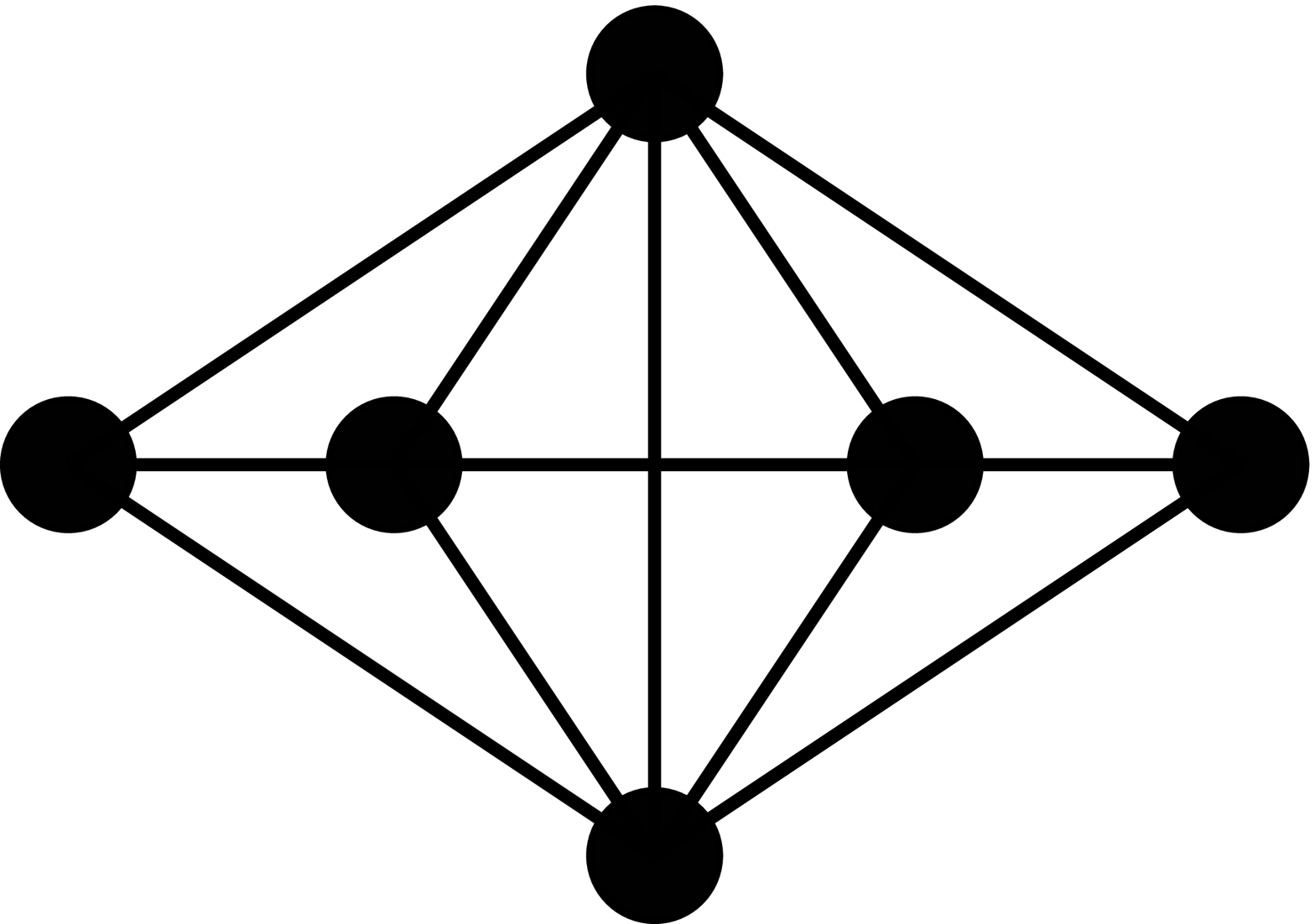} &
\includegraphics[scale=0.11]{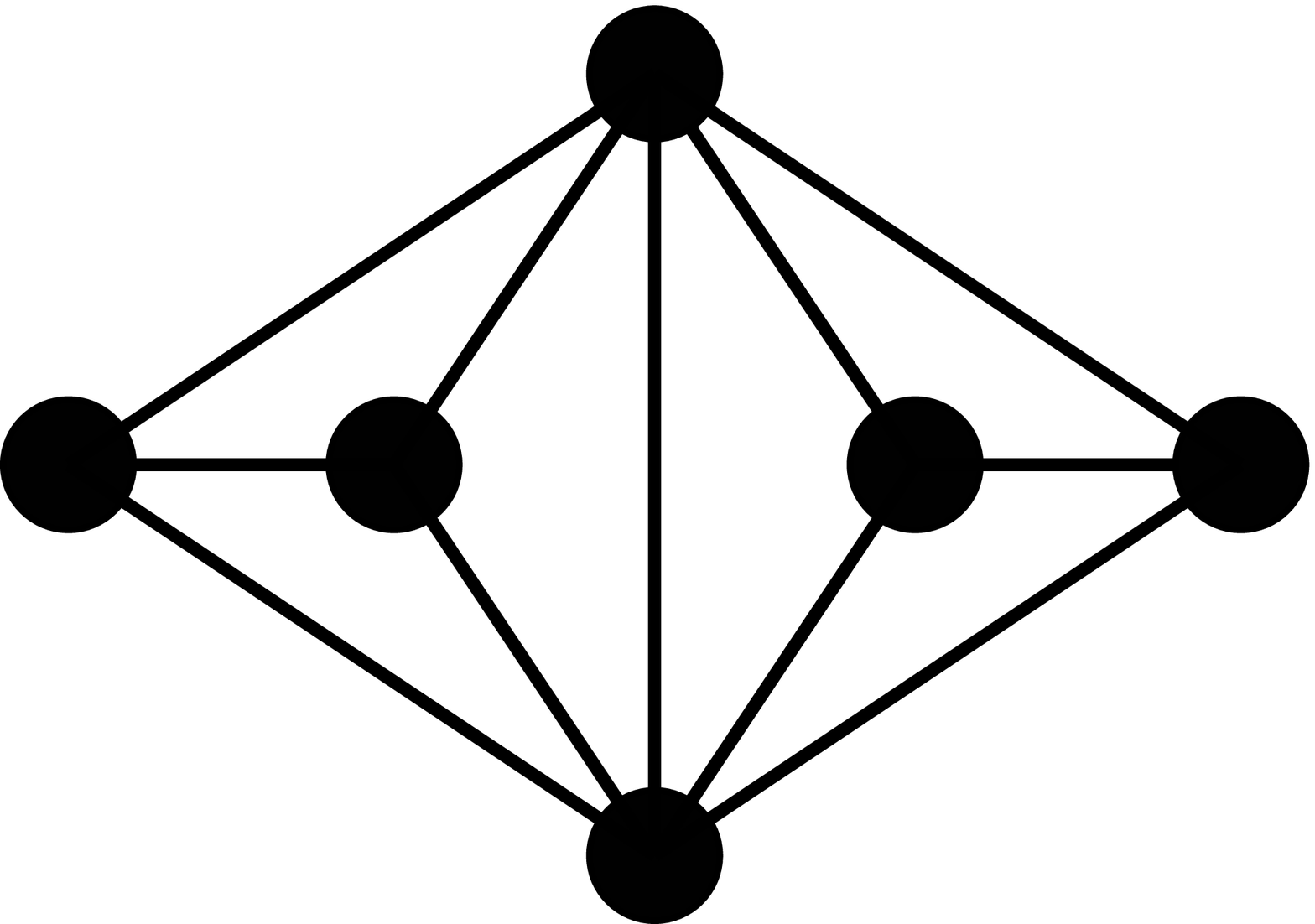} &
\includegraphics[scale=0.11]{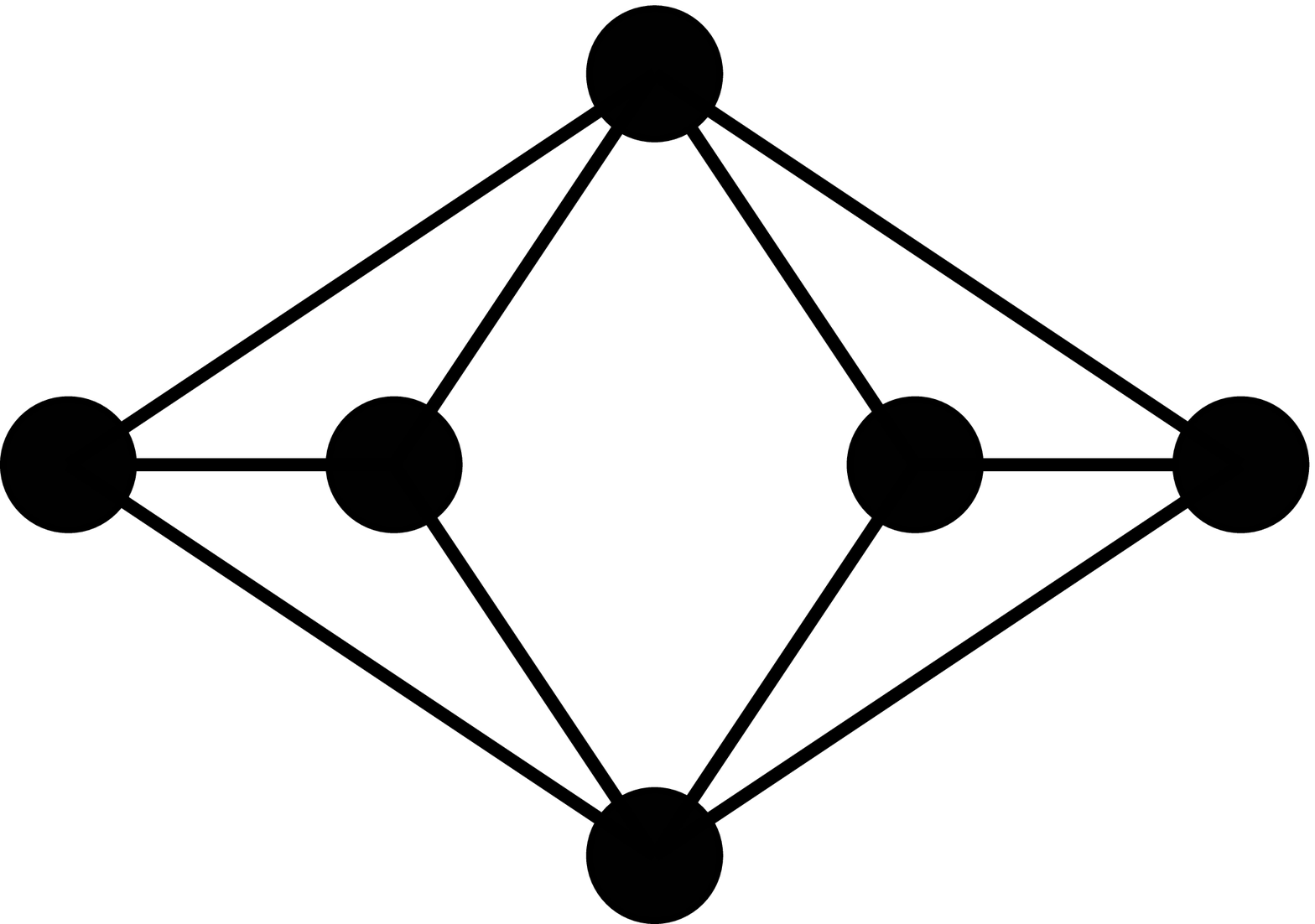}
 \\
$G_{5,1}$ & $G_{5,2}$ & $G_{6,1}$ &
$G_{6,2}$ & $G_{6,3}$ \\
&&&&\\

\includegraphics[scale=0.11]{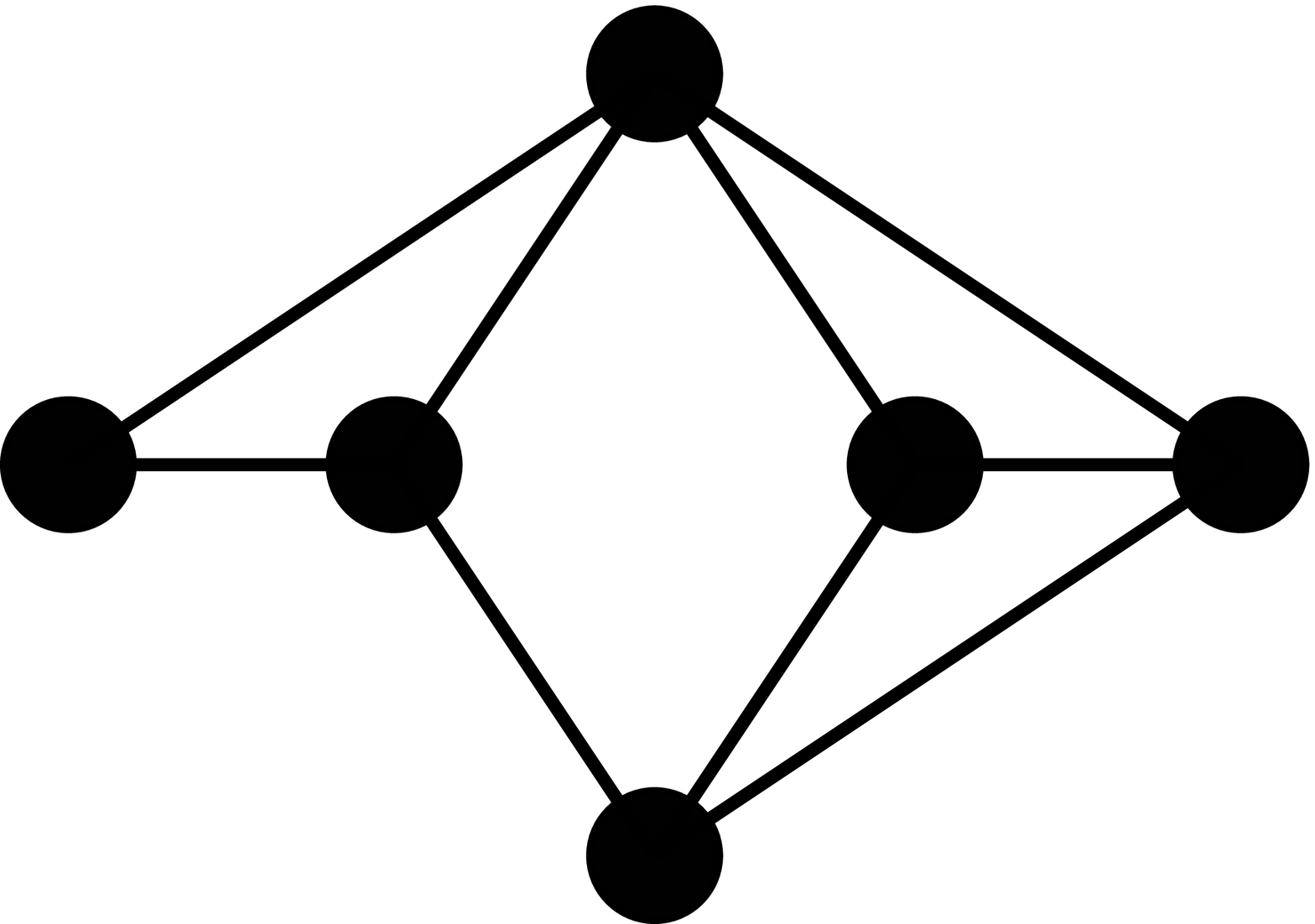} &
\includegraphics[scale=0.11]{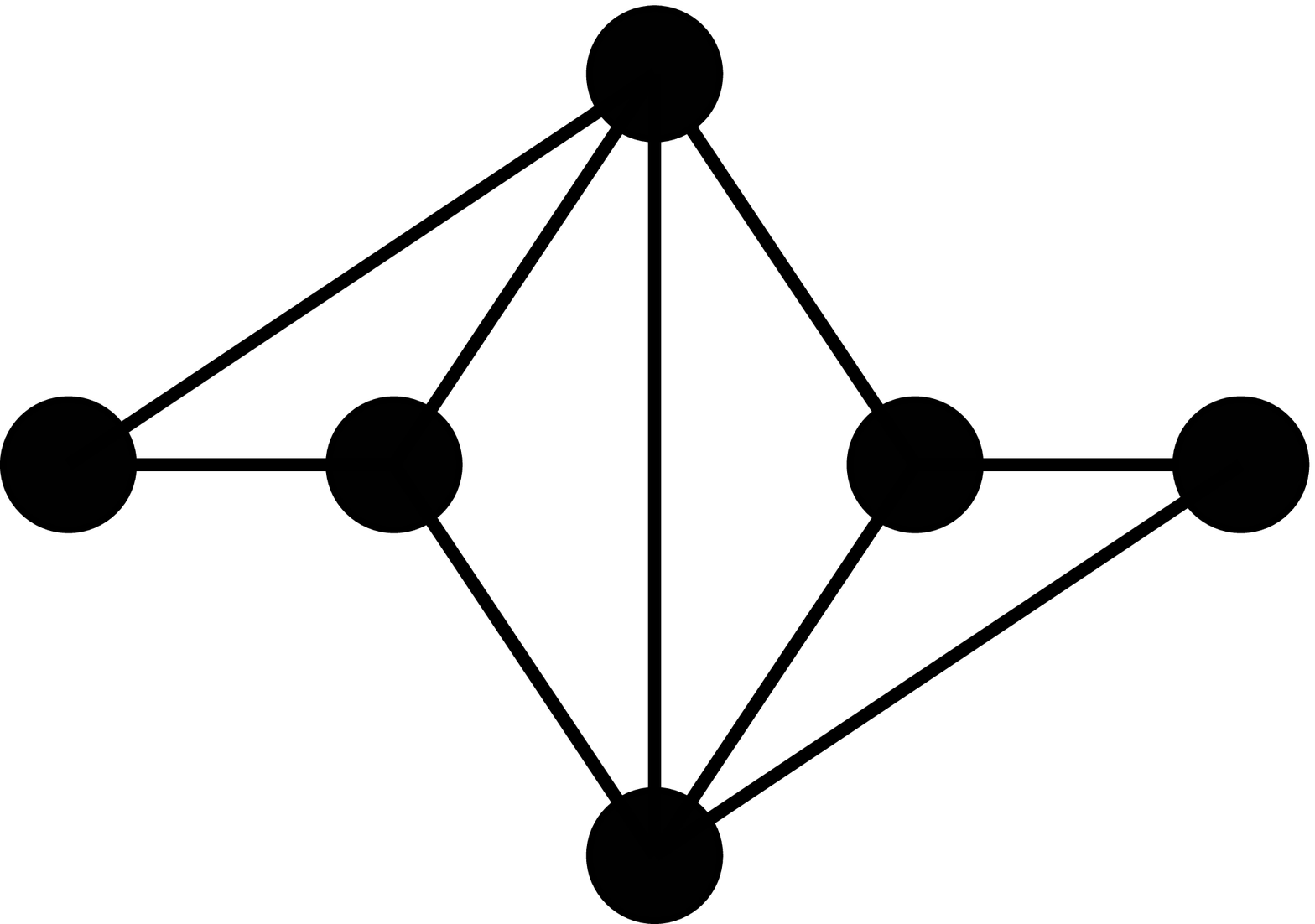} &
\includegraphics[scale=0.11]{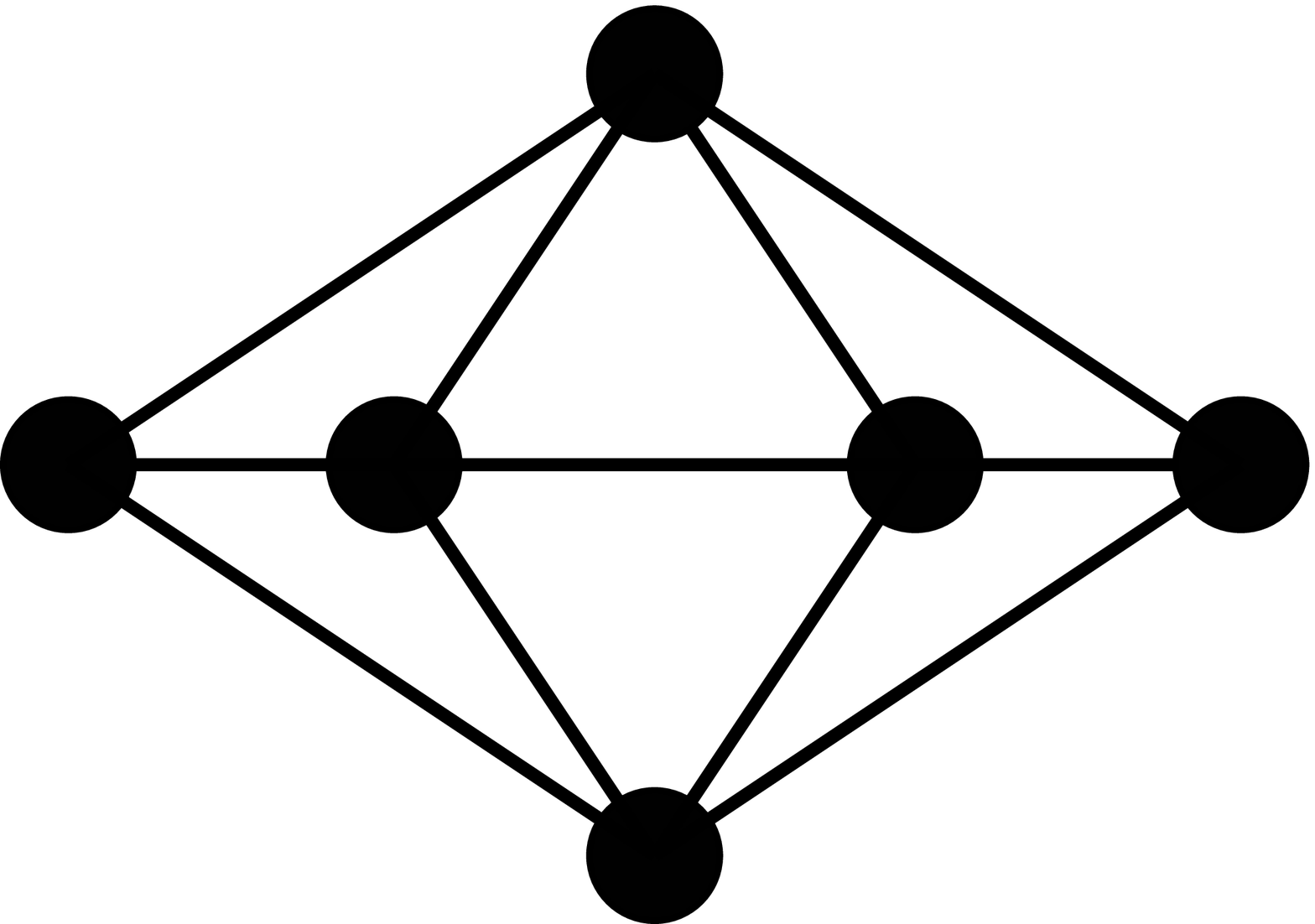} &
\includegraphics[scale=0.11]{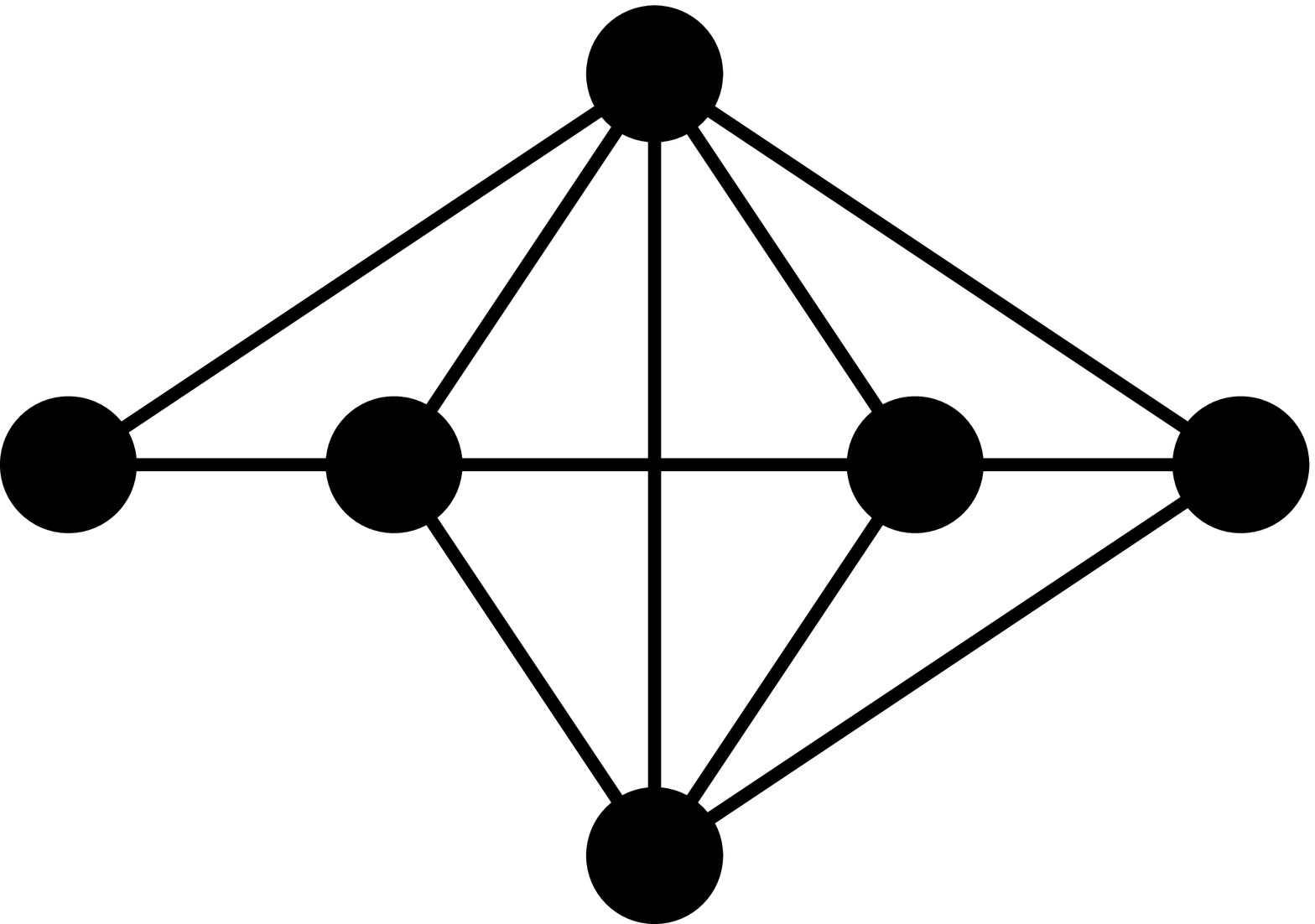} &
\includegraphics[scale=0.11]{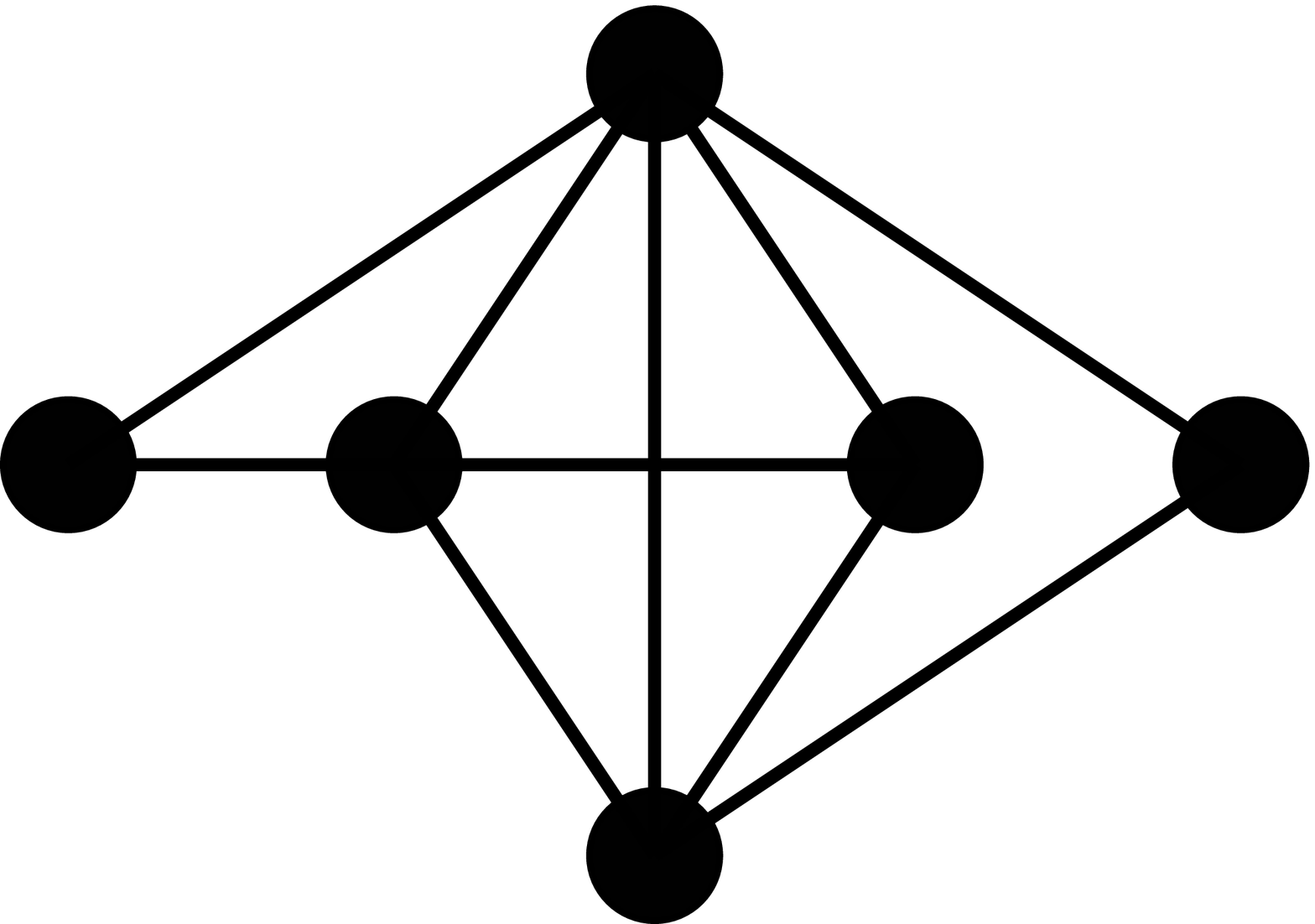} 
 \\
$G_{6,4}$ & $G_{6,5}$ & $G_{6,6}$ & $G_{6,7}$ & $G_{6,8}$ \\
&&&&\\

\includegraphics[scale=0.11]{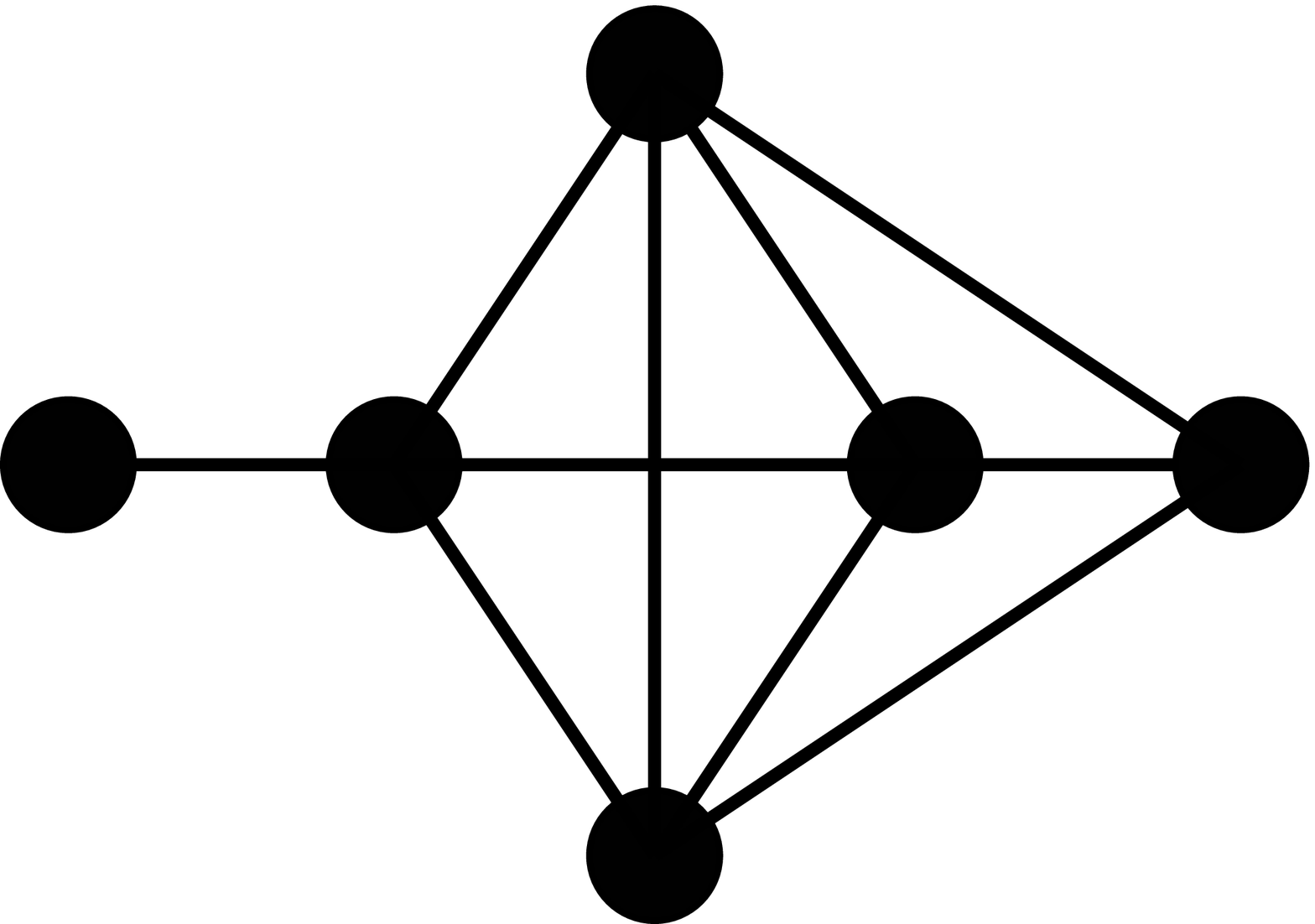} &
\includegraphics[scale=0.11]{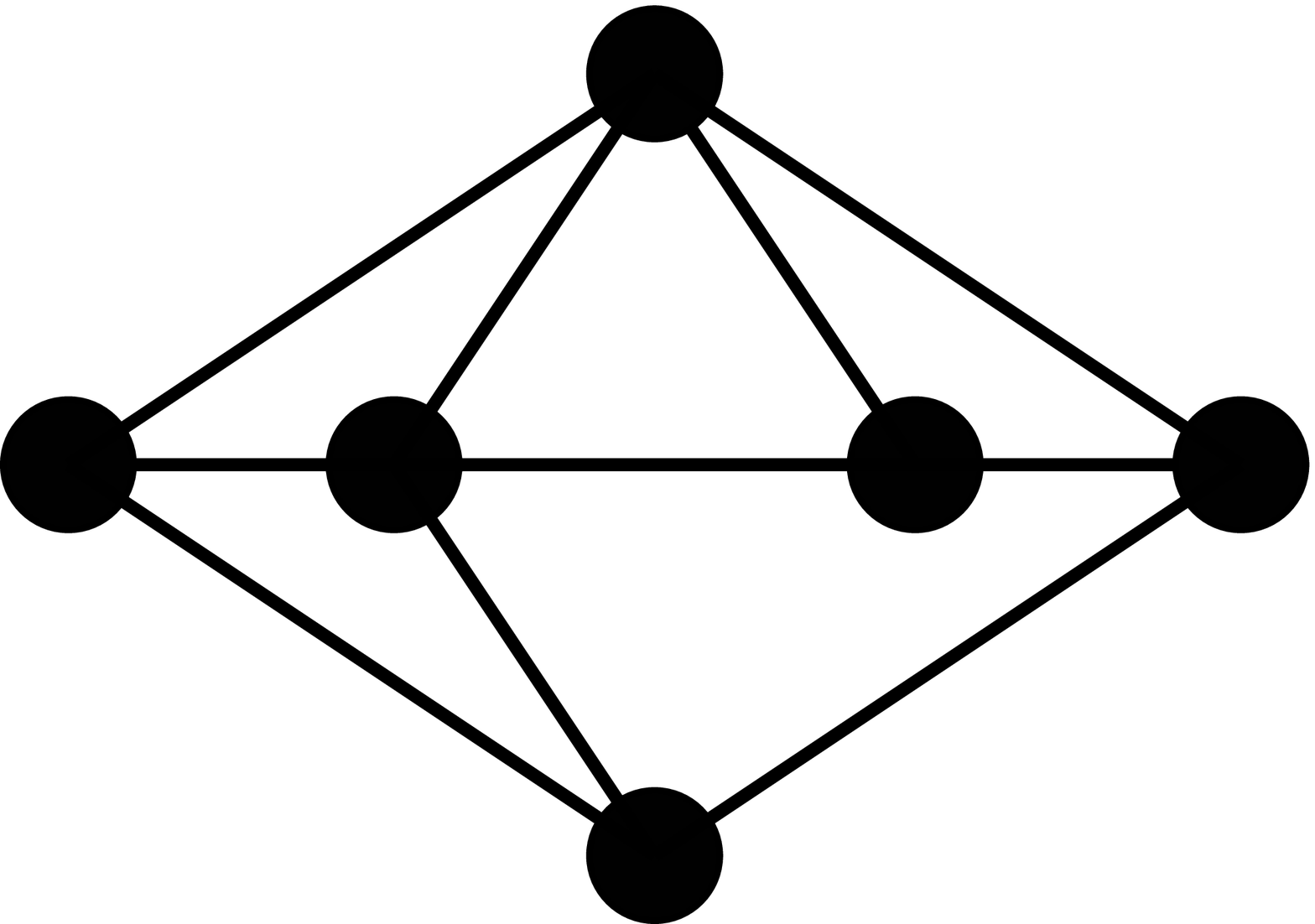} &
\includegraphics[scale=0.11]{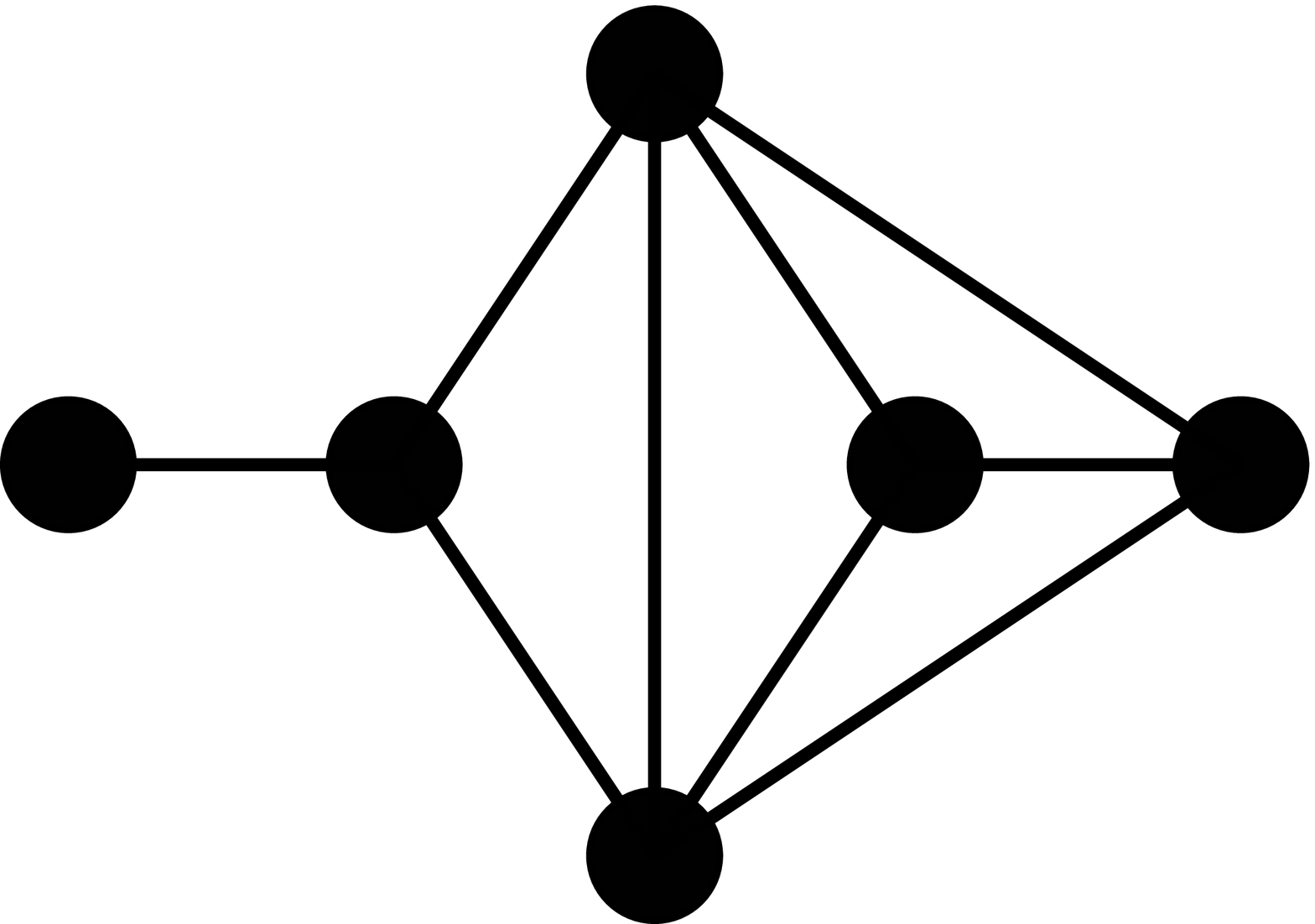} &
\includegraphics[scale=0.11]{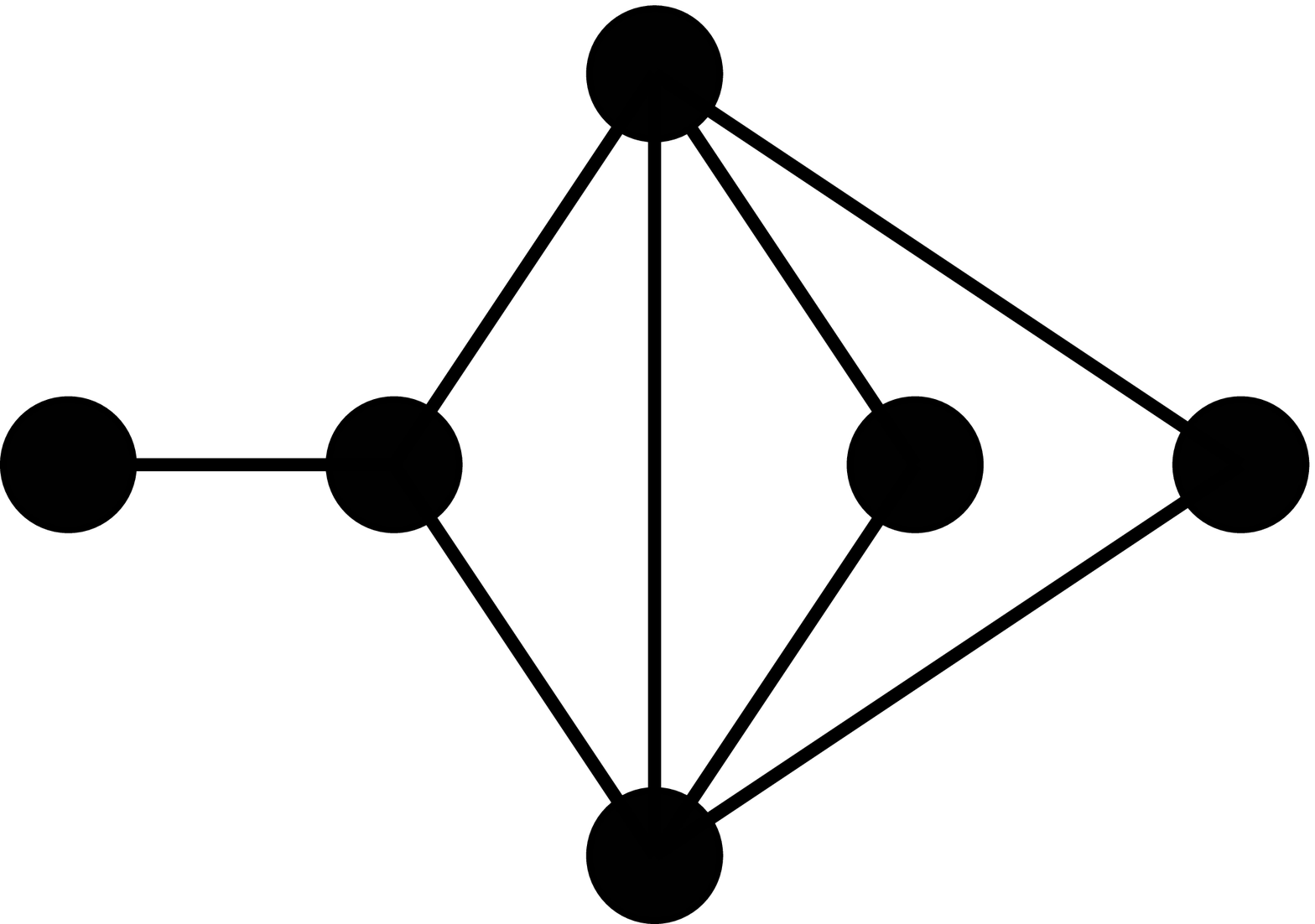} &
\includegraphics[scale=0.11]{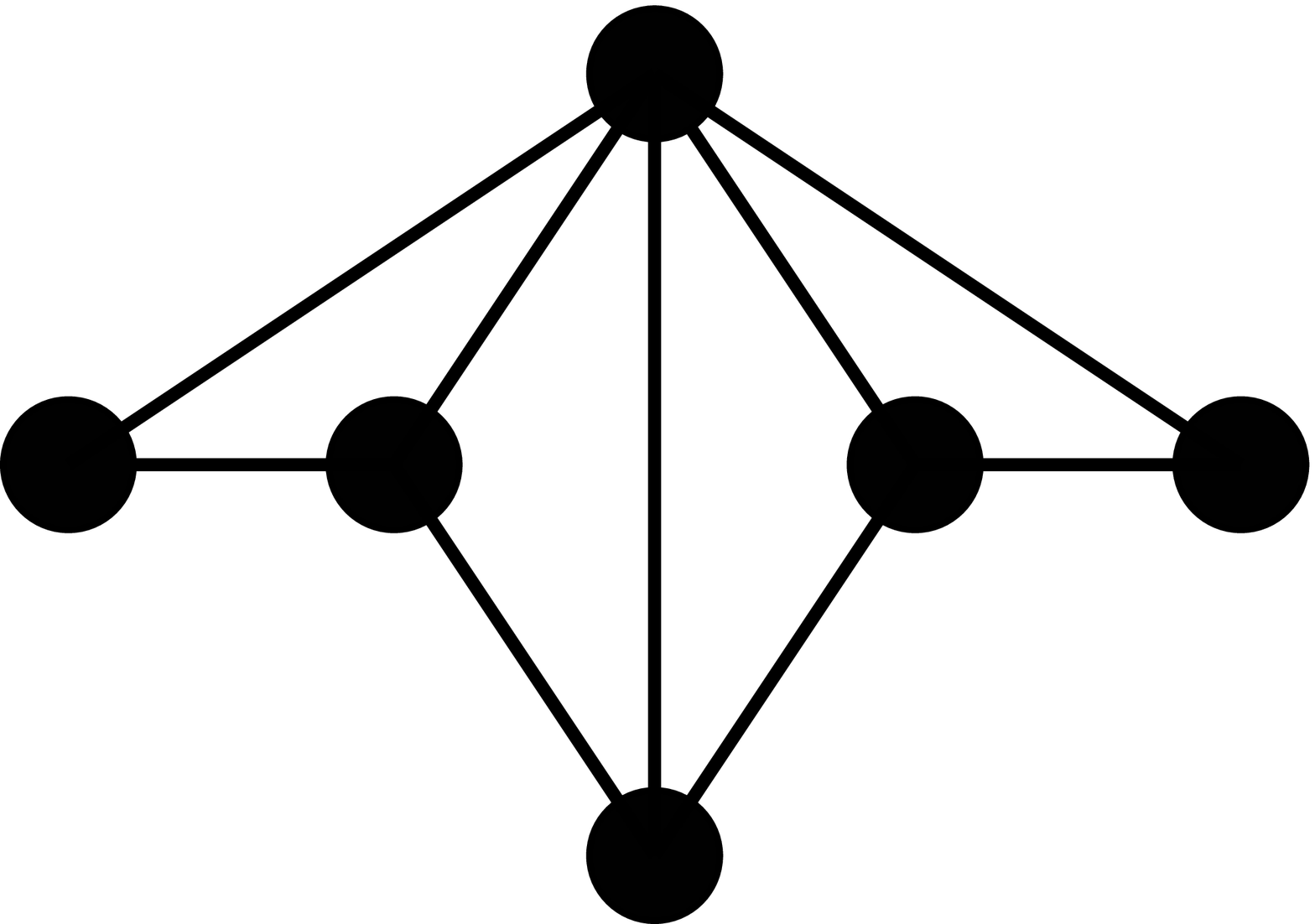}
 \\
$G_{6,9}$ & $G_{6,10}$ & $G_{6,11}$ & $G_{6,12}$ & $G_{6,13}$ \\
&&&&\\

\includegraphics[scale=0.11]{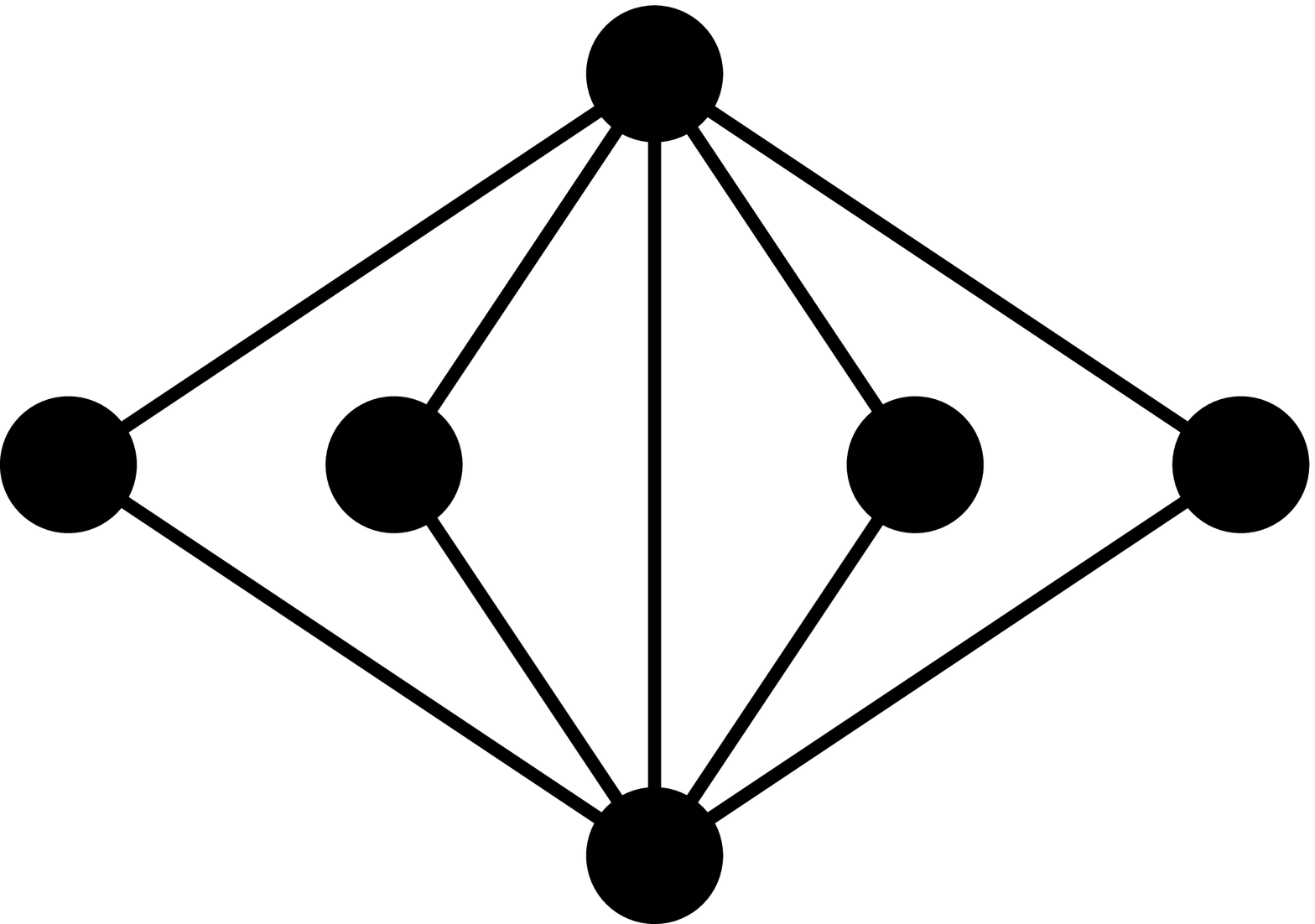} &
\includegraphics[scale=0.11]{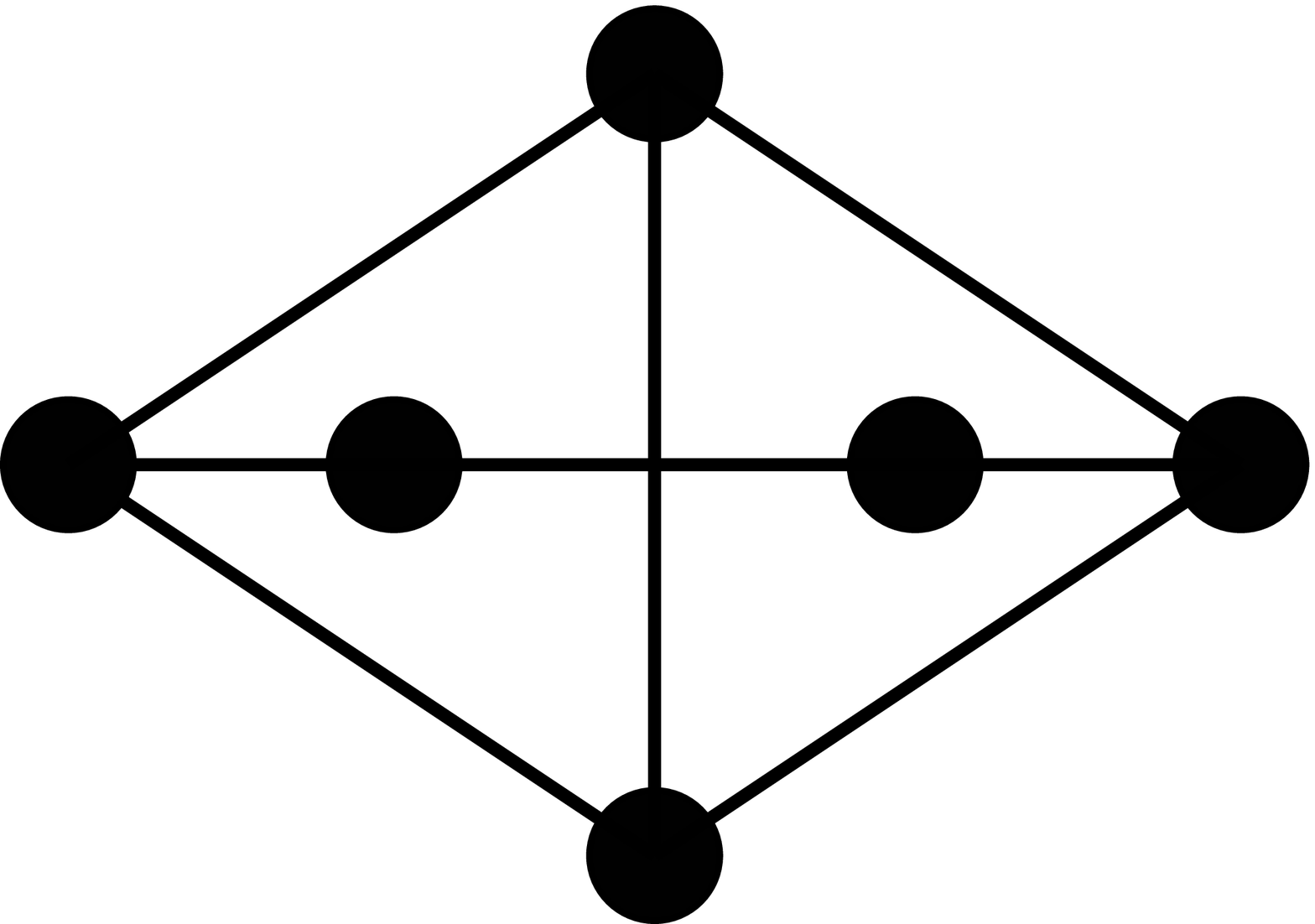} &
\includegraphics[scale=0.11]{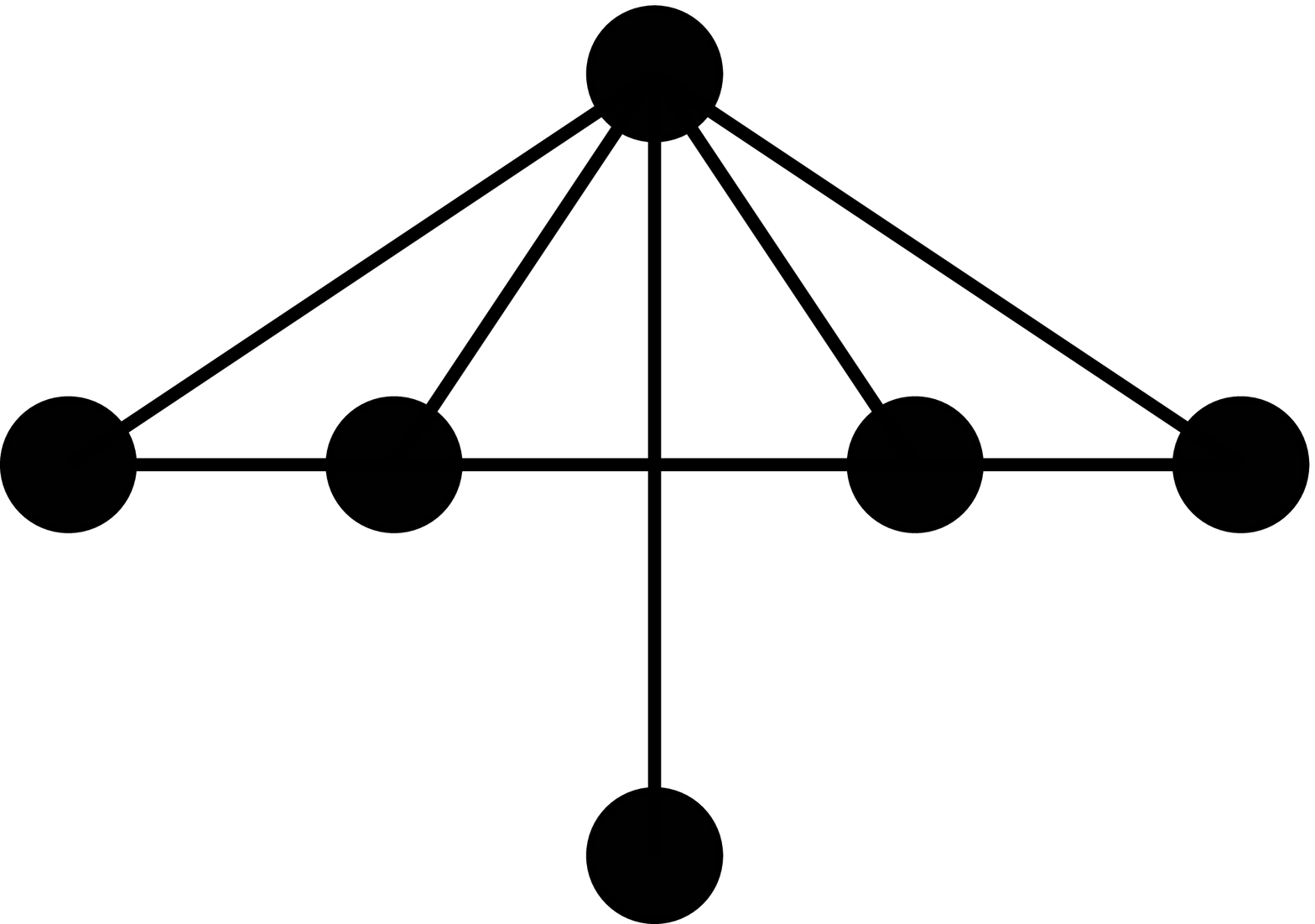} &
\includegraphics[scale=0.11]{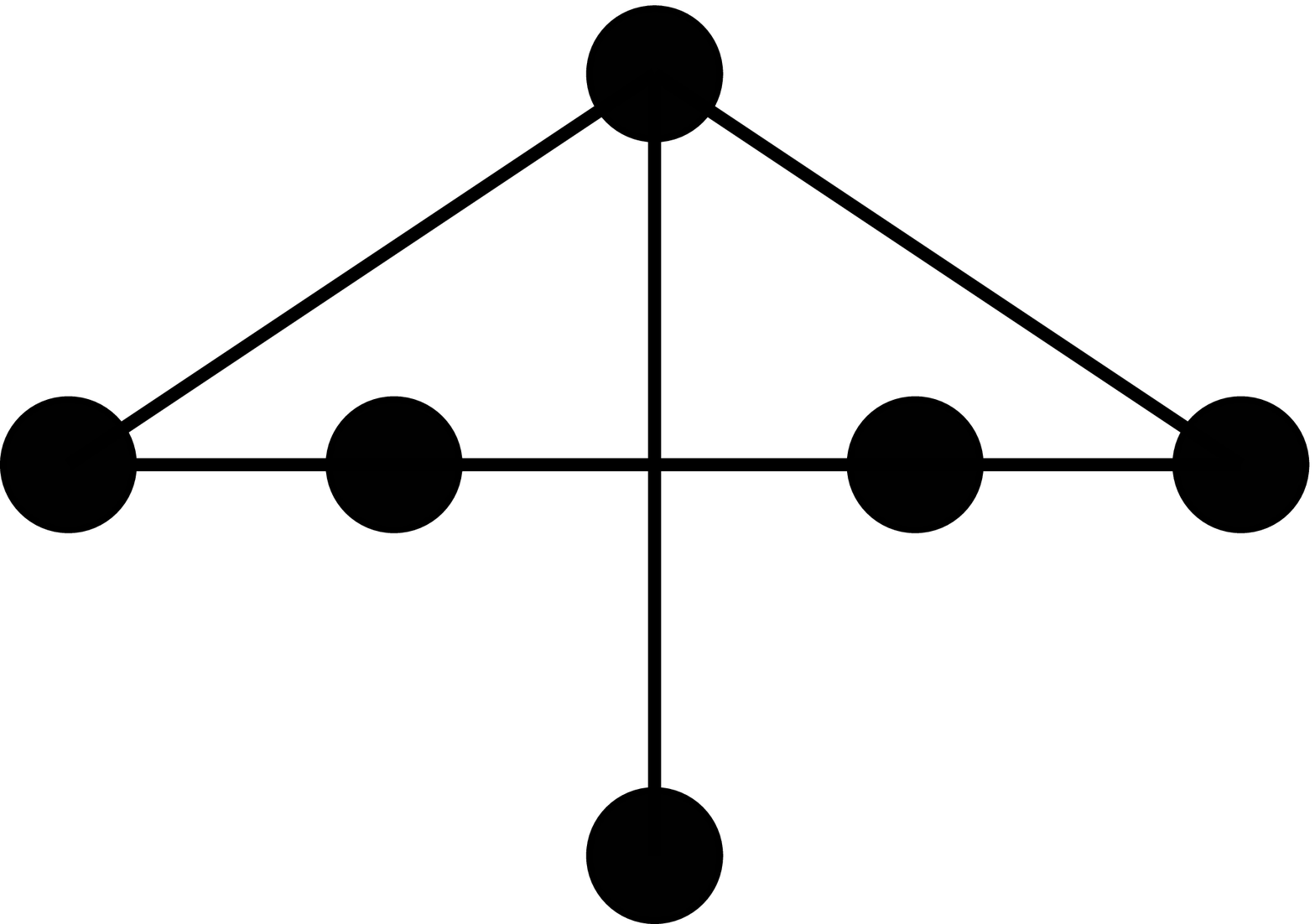} &
\includegraphics[scale=0.11]{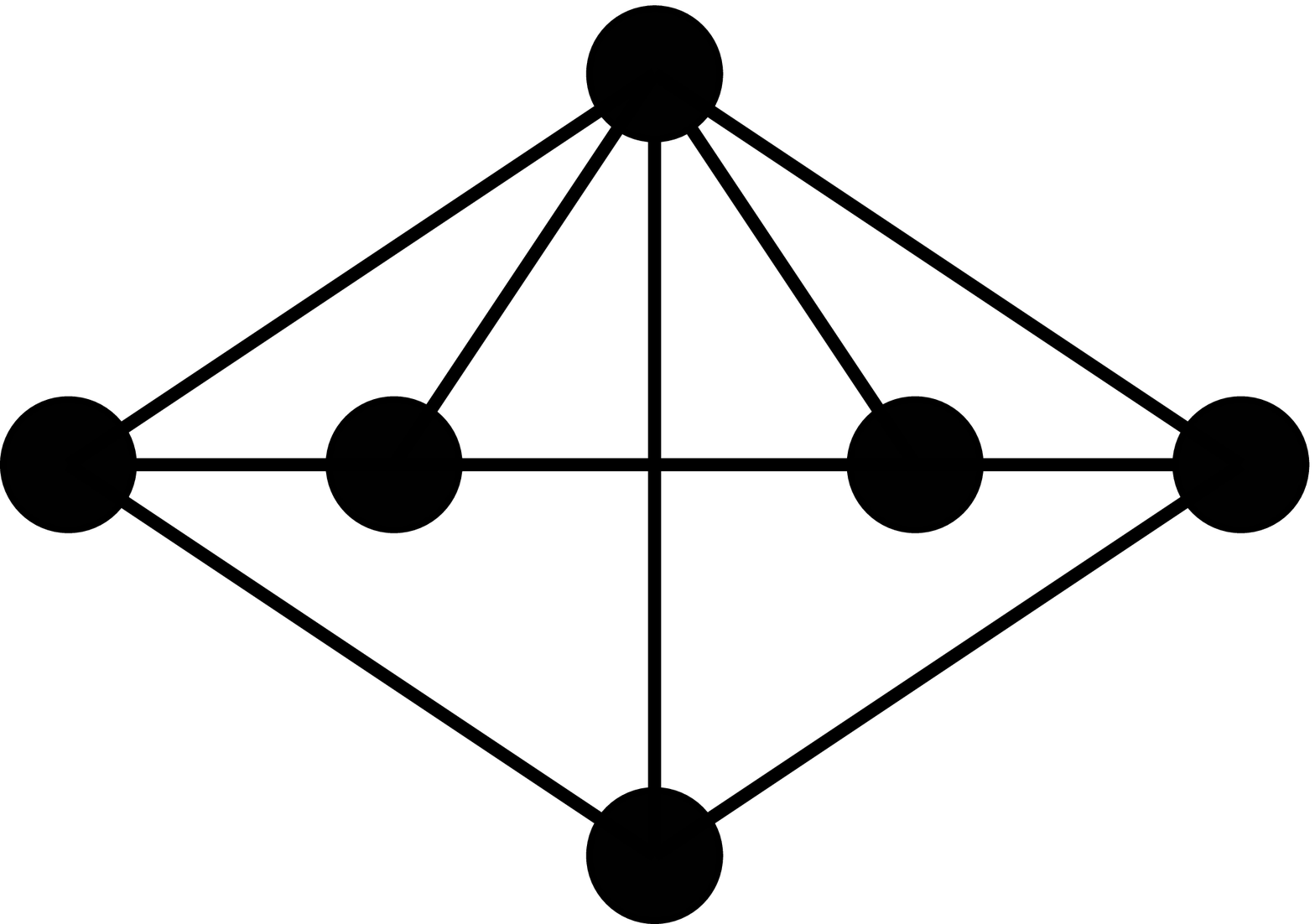}
 \\
$G_{6,14}$ & $G_{6,15}$ & $G_{6,16}$ & $G_{6,17}$ & $G_{6,18}$ \\
&&&&\\

\includegraphics[scale=0.11]{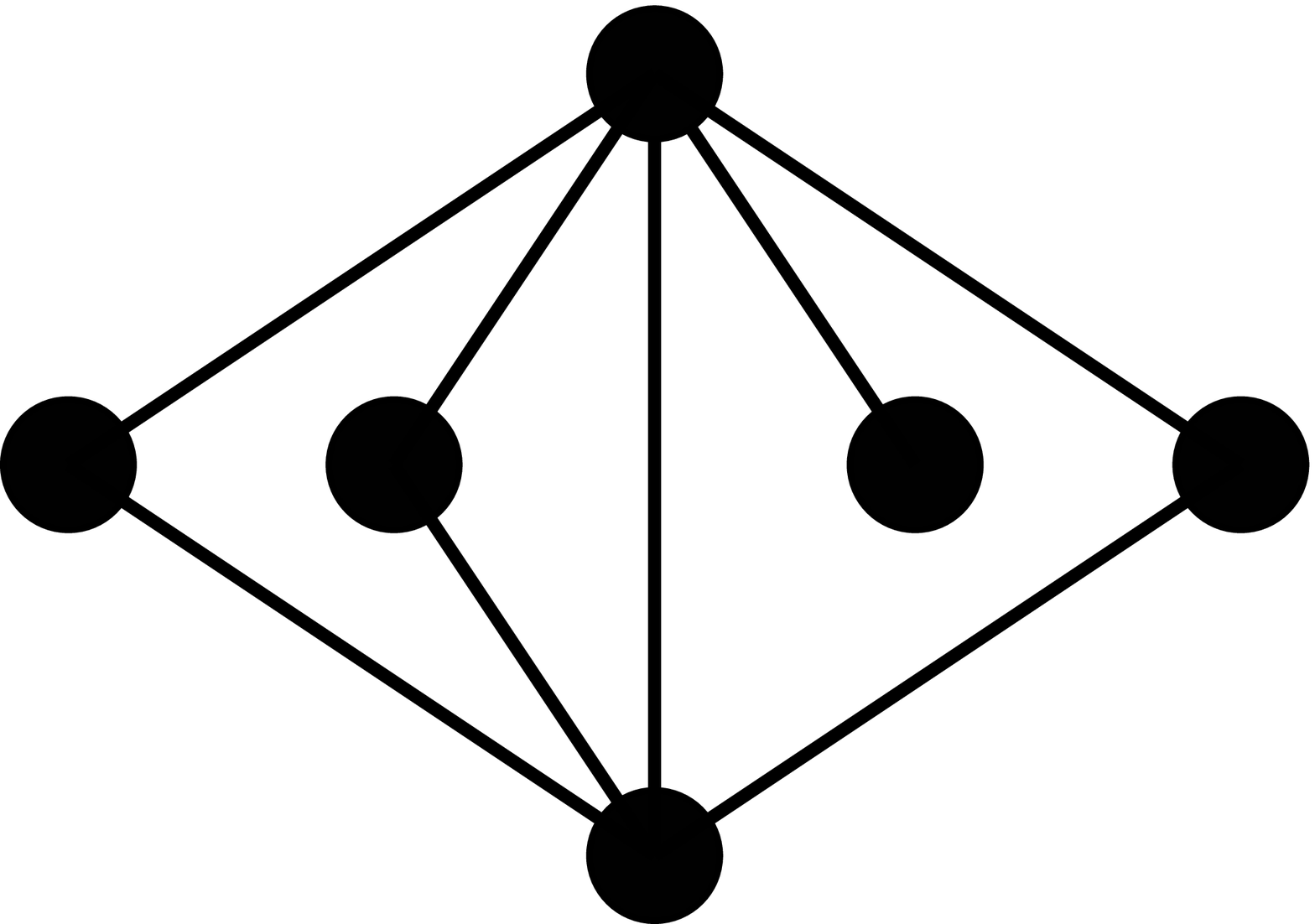} &
\includegraphics[scale=0.11]{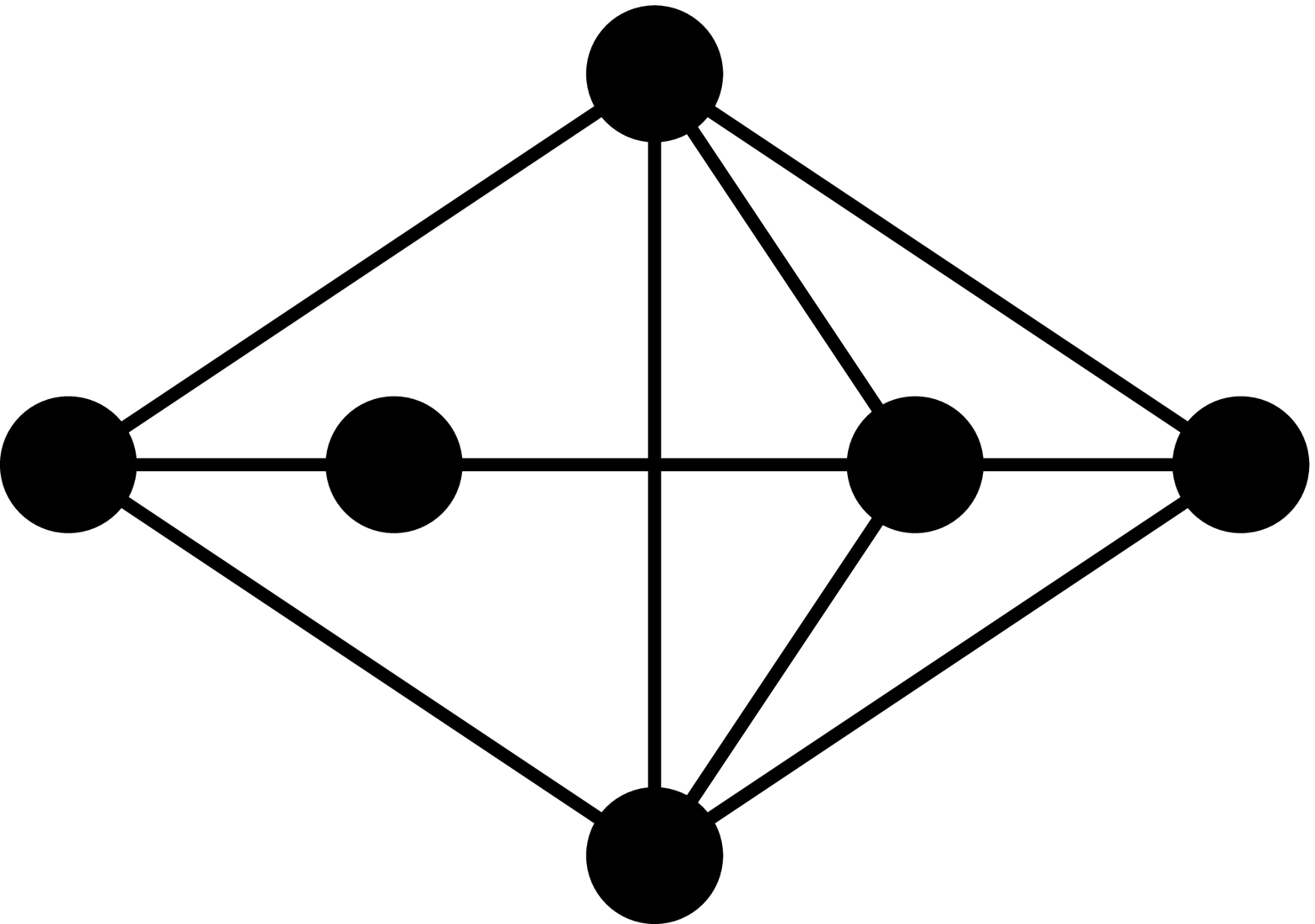} &
\includegraphics[scale=0.11]{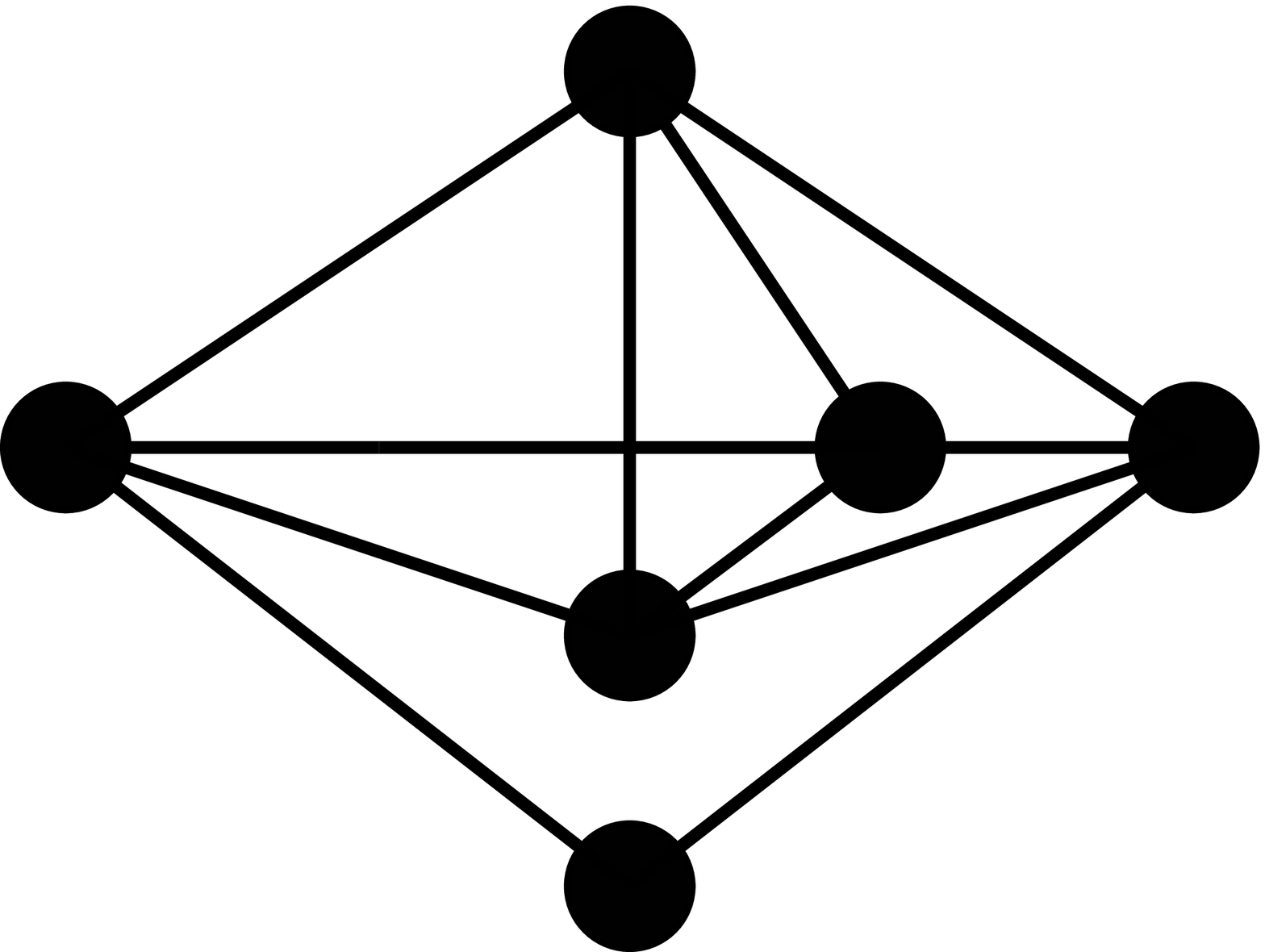} &
\includegraphics[scale=0.11]{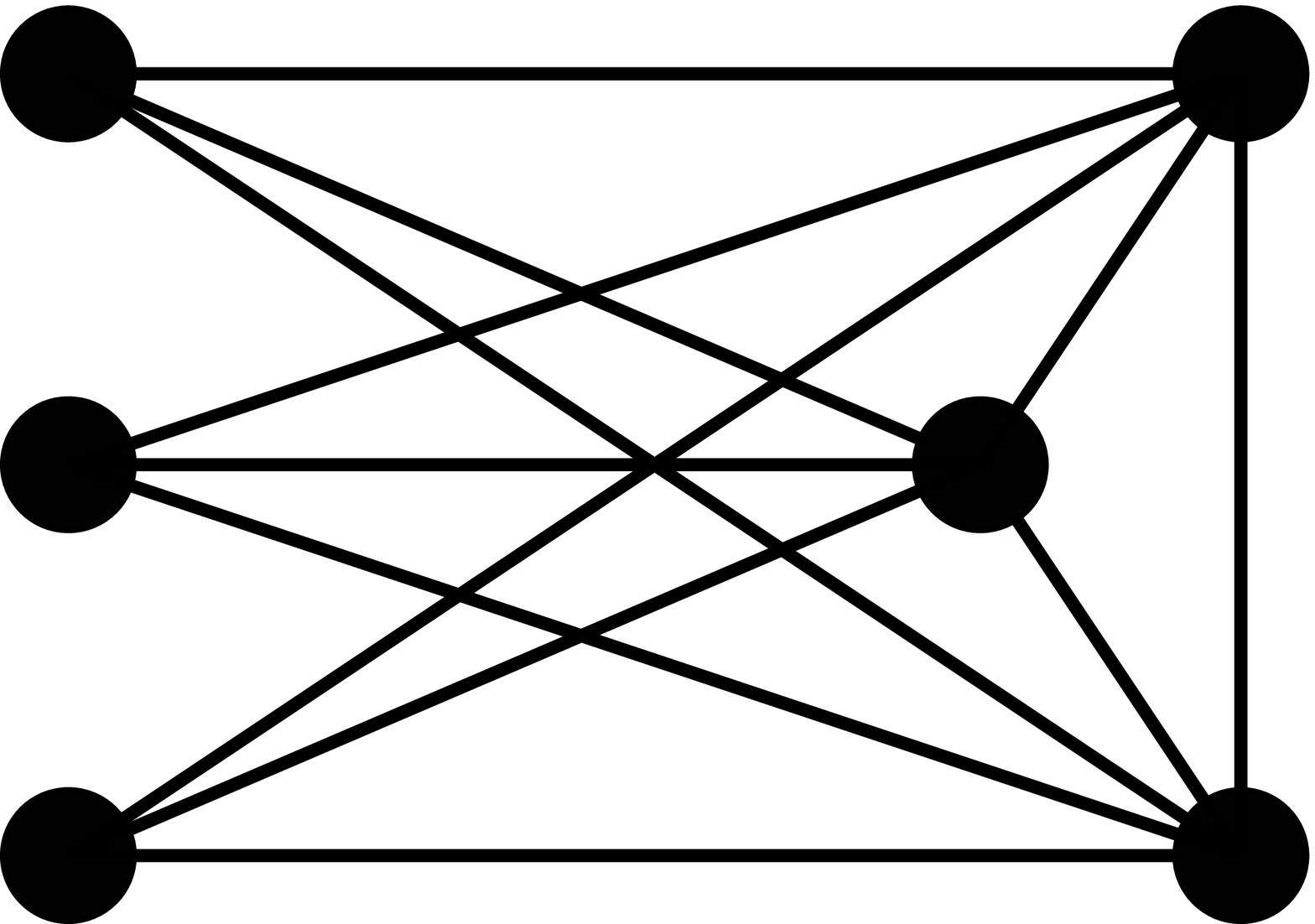} &
\includegraphics[scale=0.11]{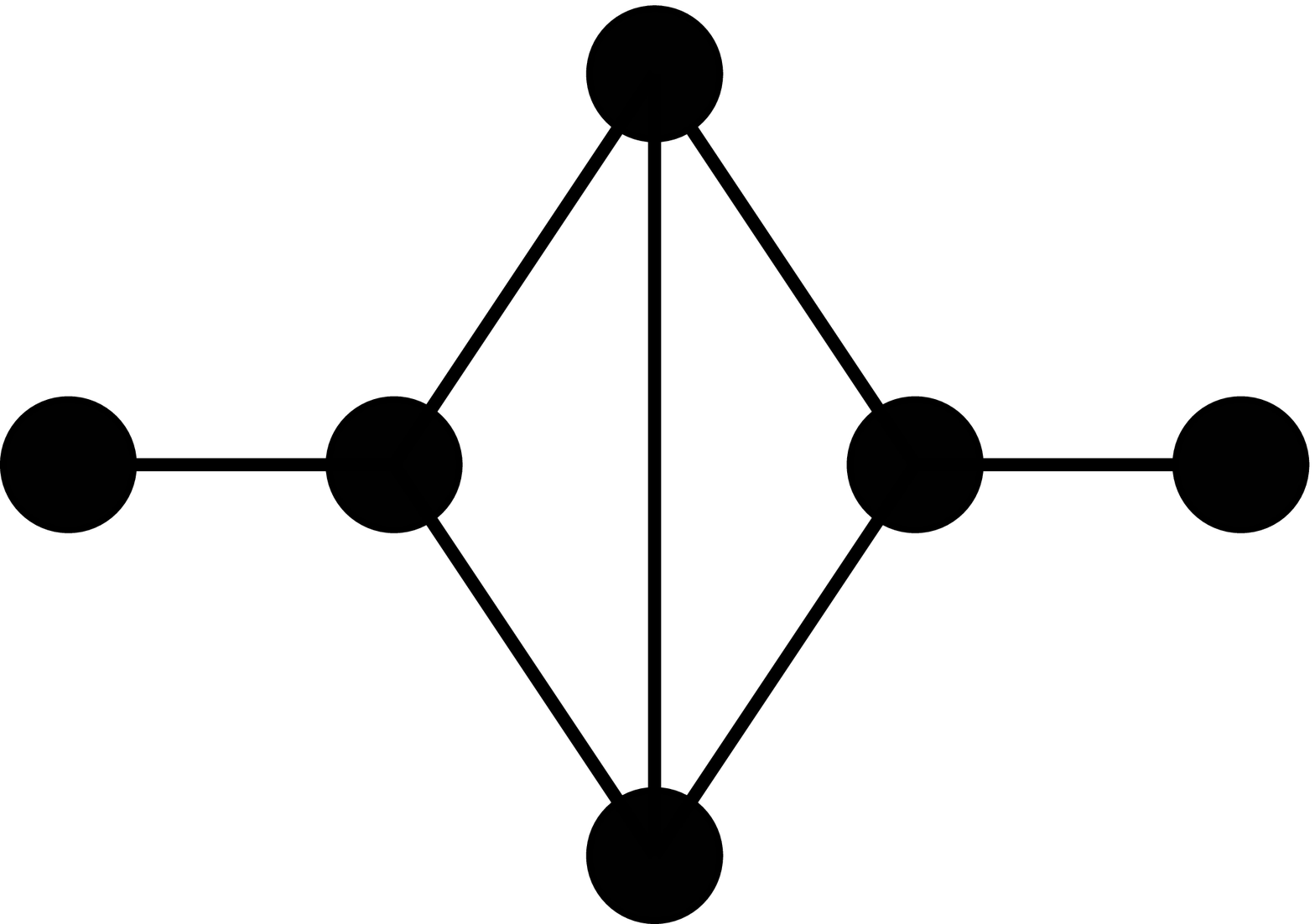}
 \\
$G_{6,19}$ & $G_{6,20}$ & $G_{6,21}$ & $G_{6,22}$ & $G_{6,23}$ \\
&&&&\\

\includegraphics[scale=0.11]{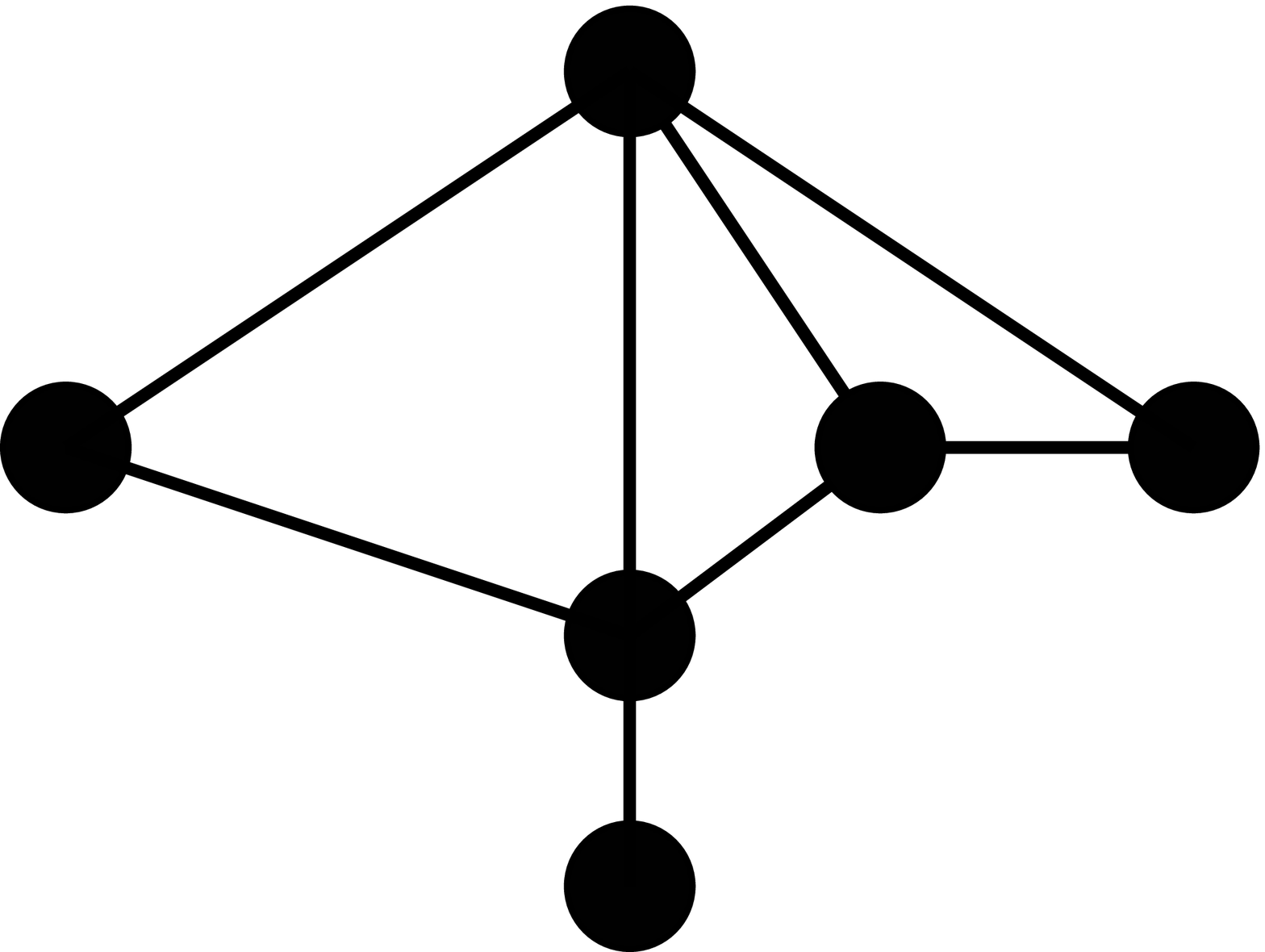} &
\includegraphics[scale=0.11]{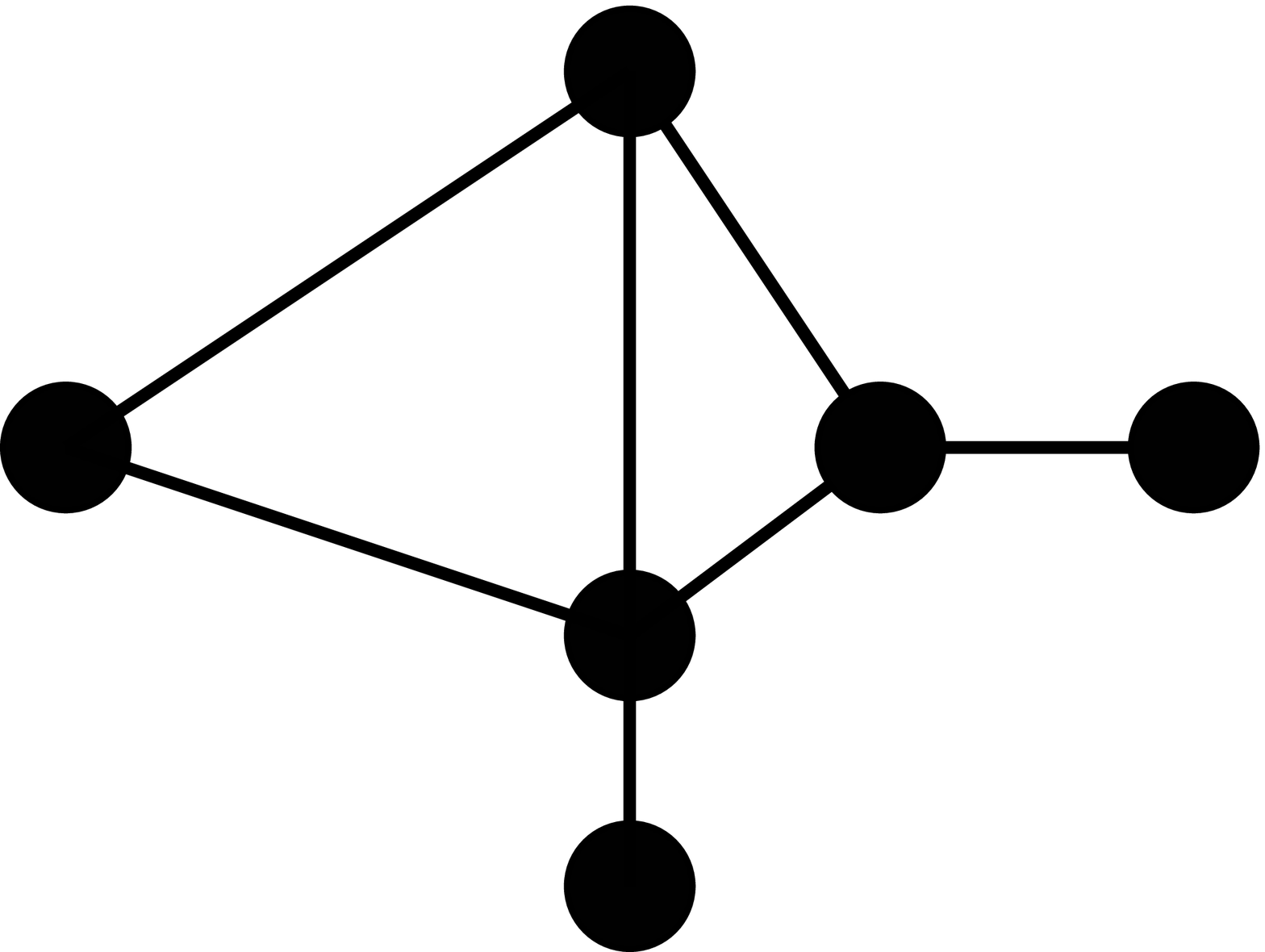} &
\includegraphics[scale=0.11]{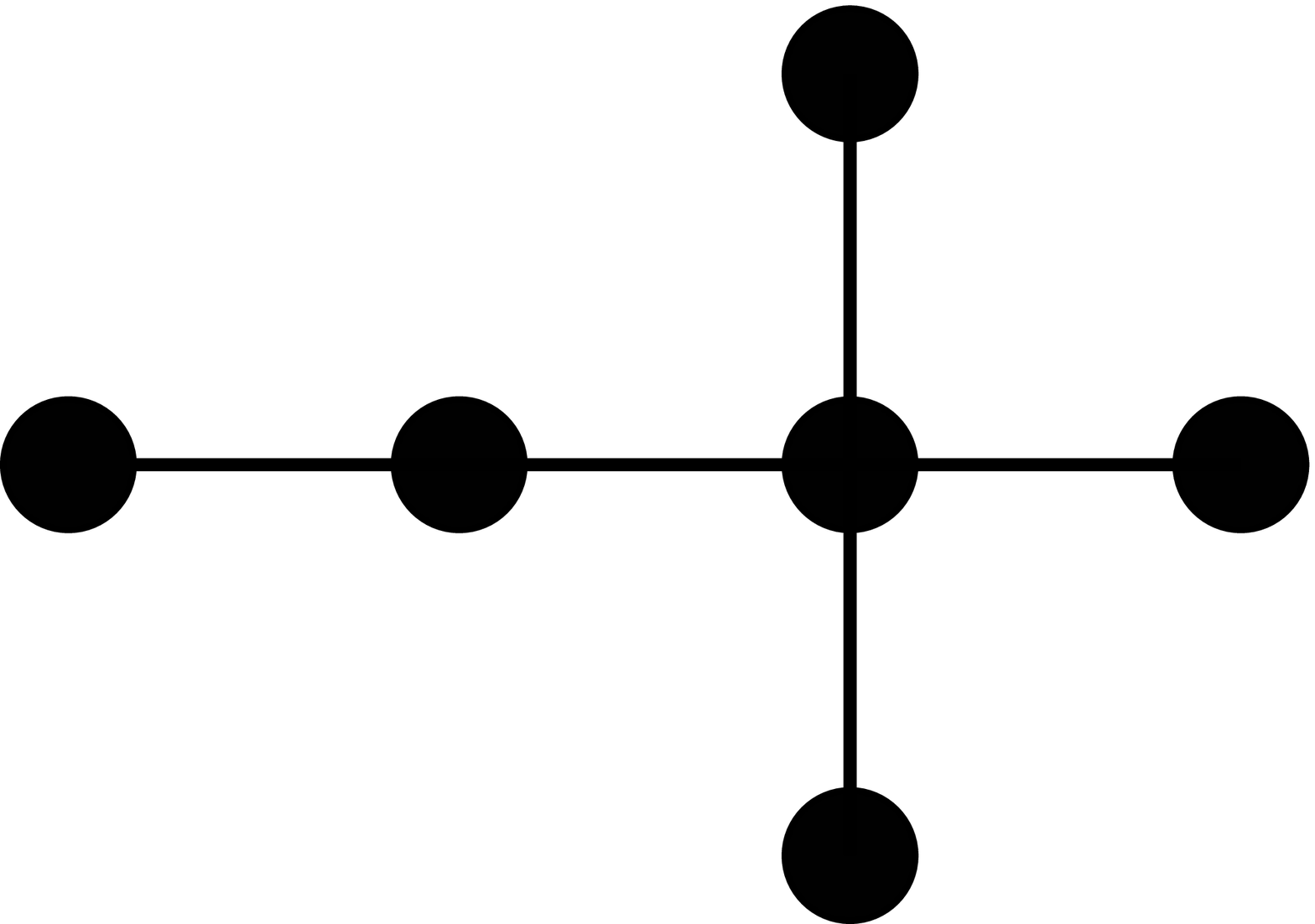} &
\includegraphics[scale=0.11]{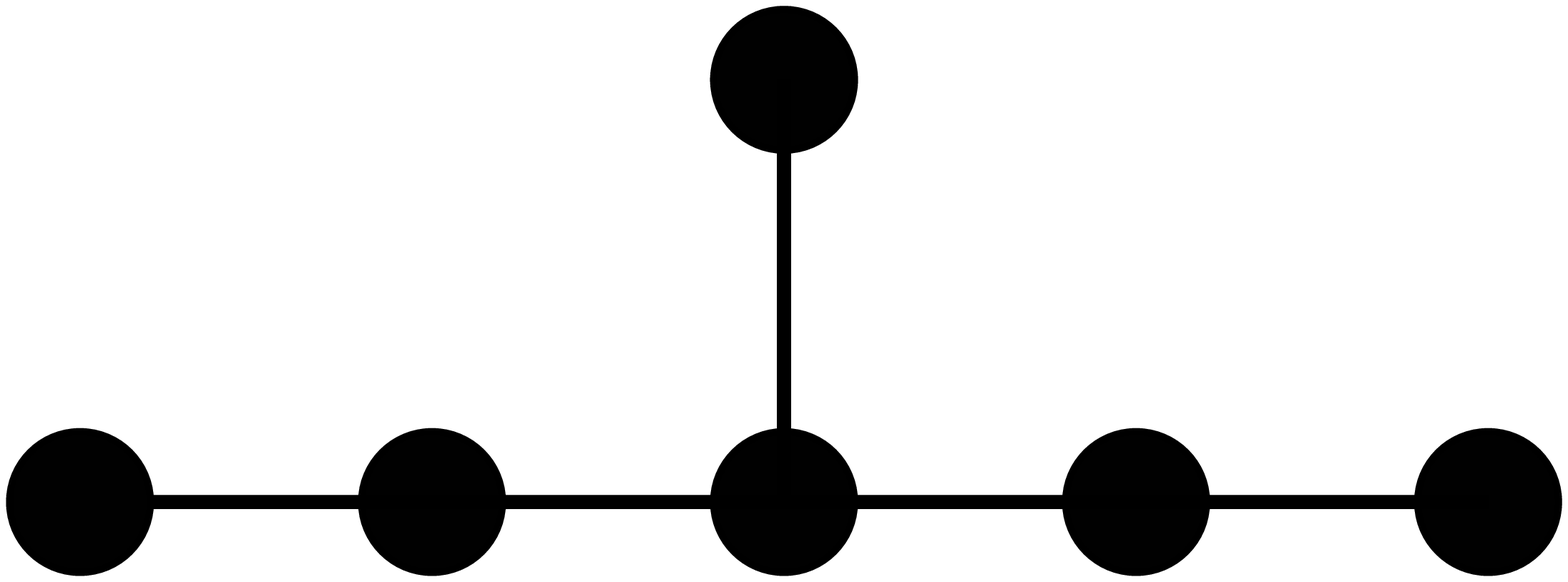} &
\includegraphics[scale=0.11]{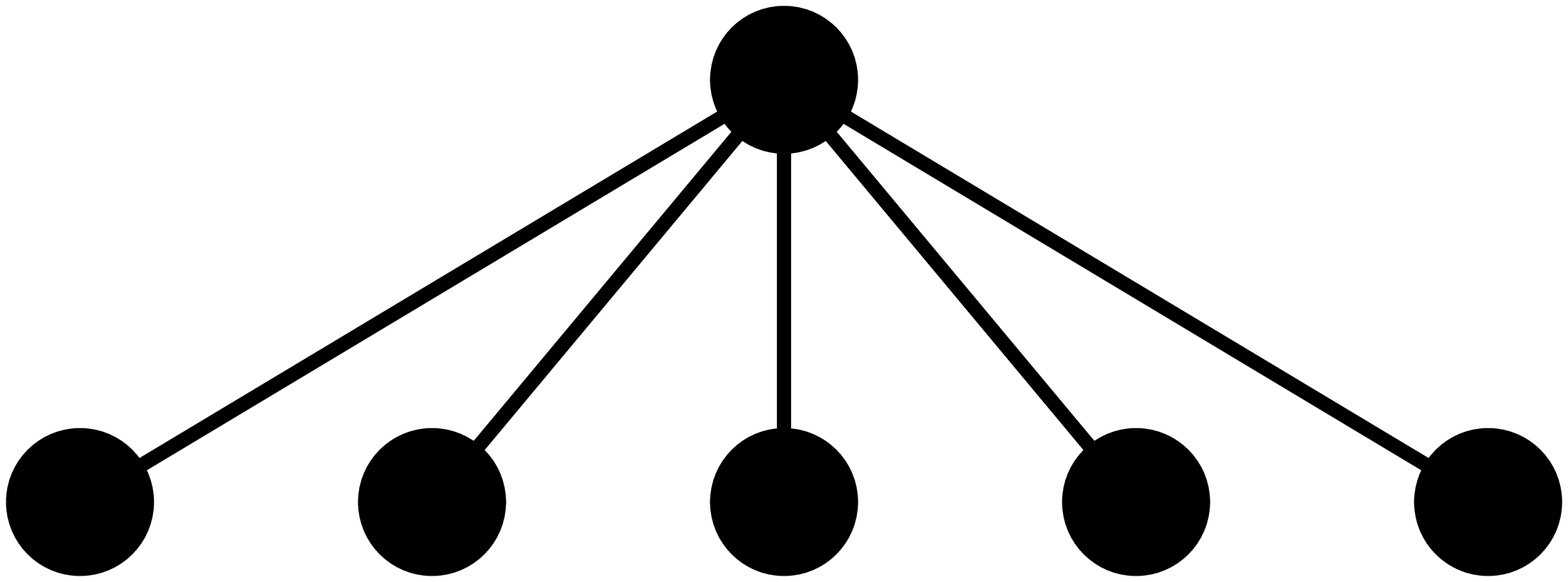}
 \\
$G_{6,24}$ & $G_{6,25}$ & $G_{6,26}$ & $G_{6,27}$ & $G_{6,28}$ \\
&&&&\\

\includegraphics[scale=0.11]{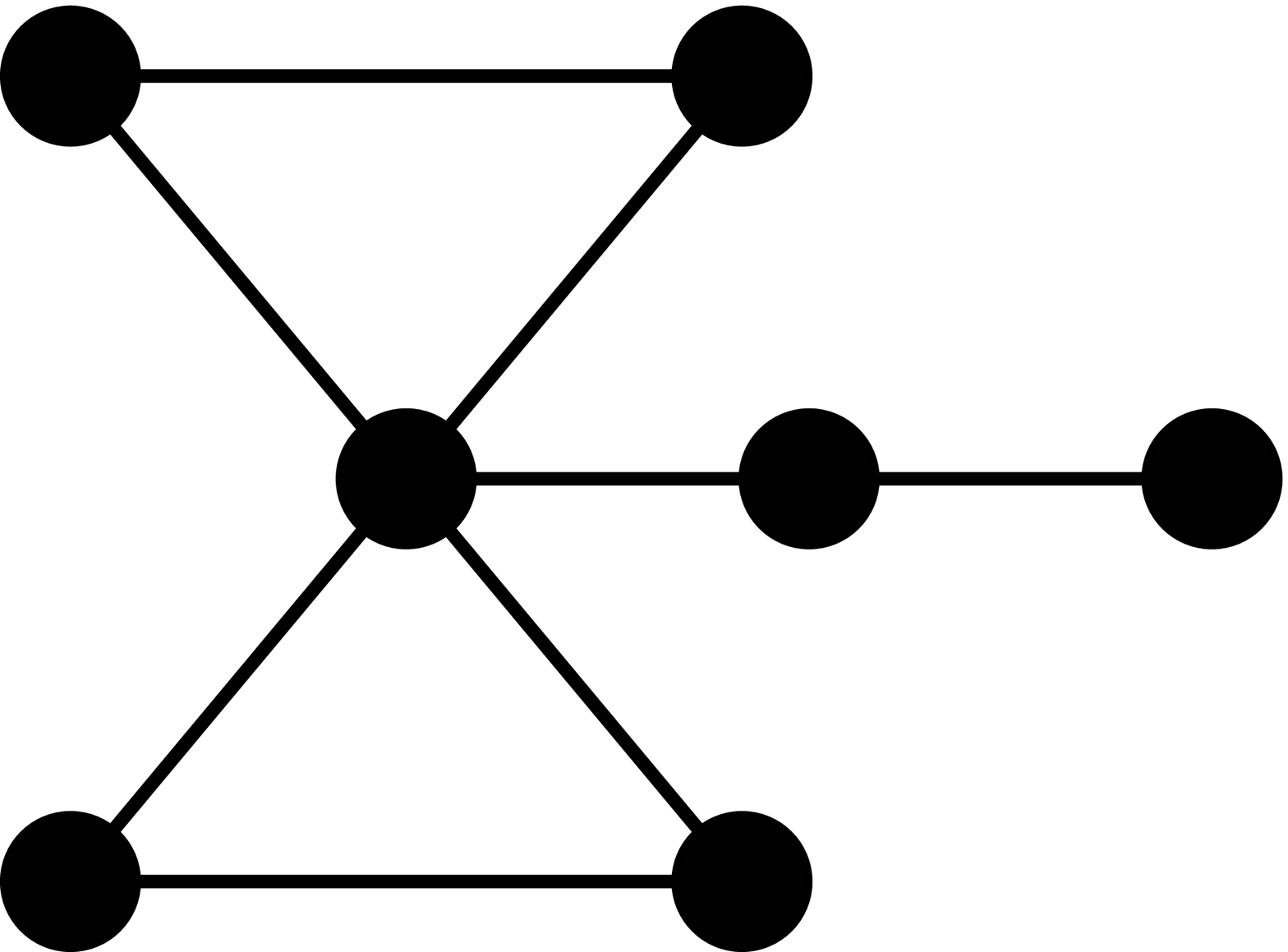} &
\includegraphics[scale=0.11]{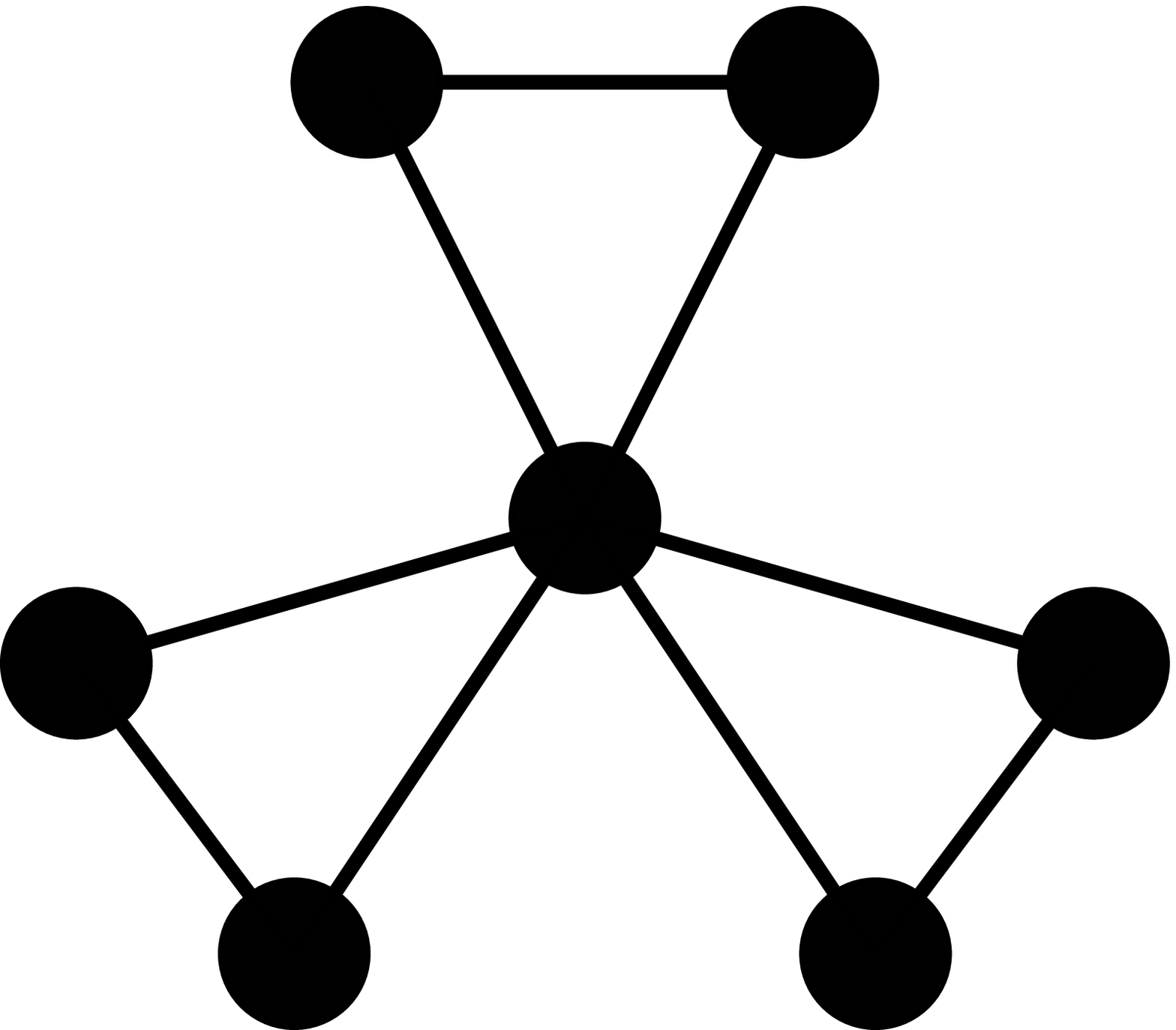} &
\includegraphics[scale=0.11]{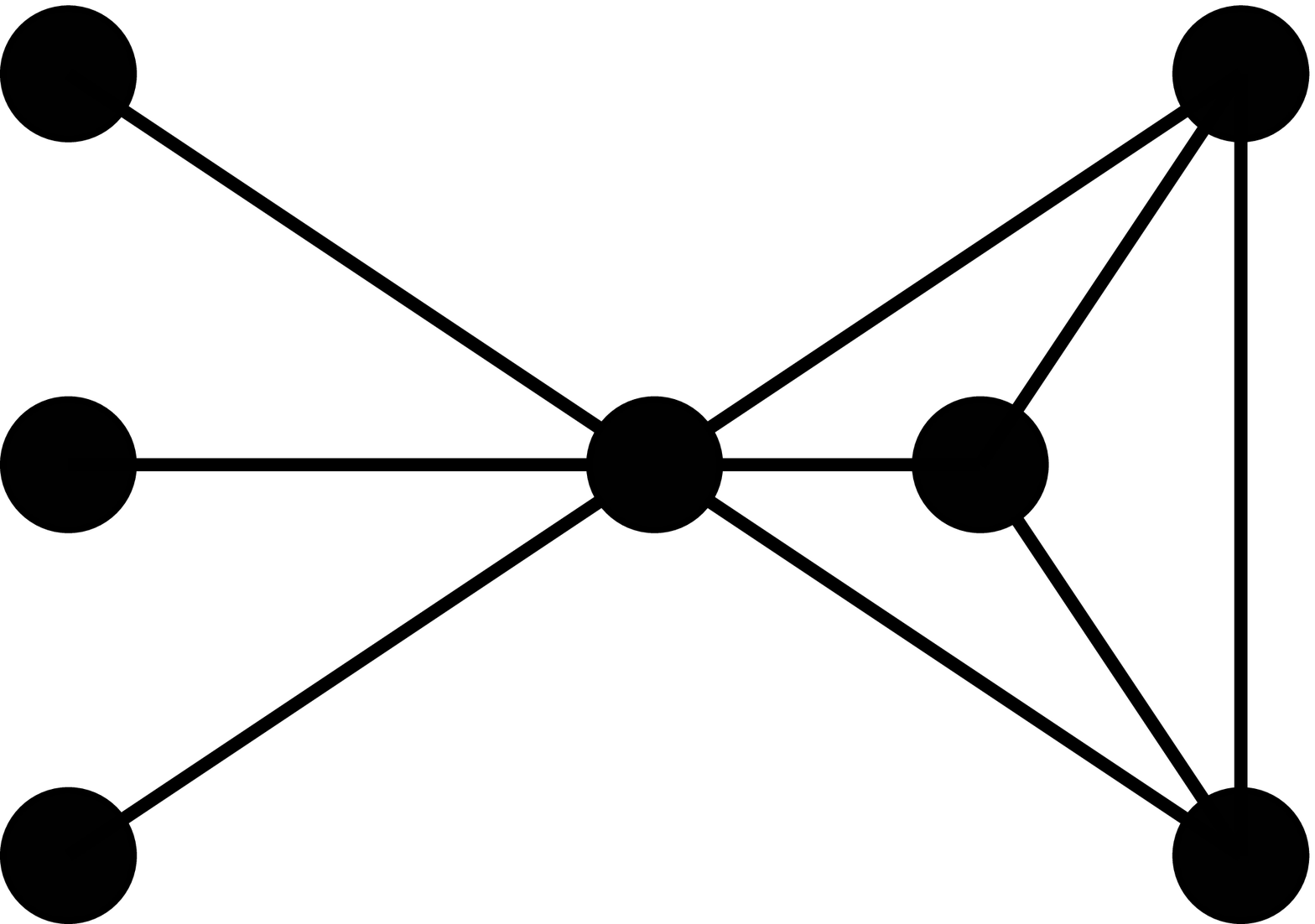} &
\includegraphics[scale=0.11]{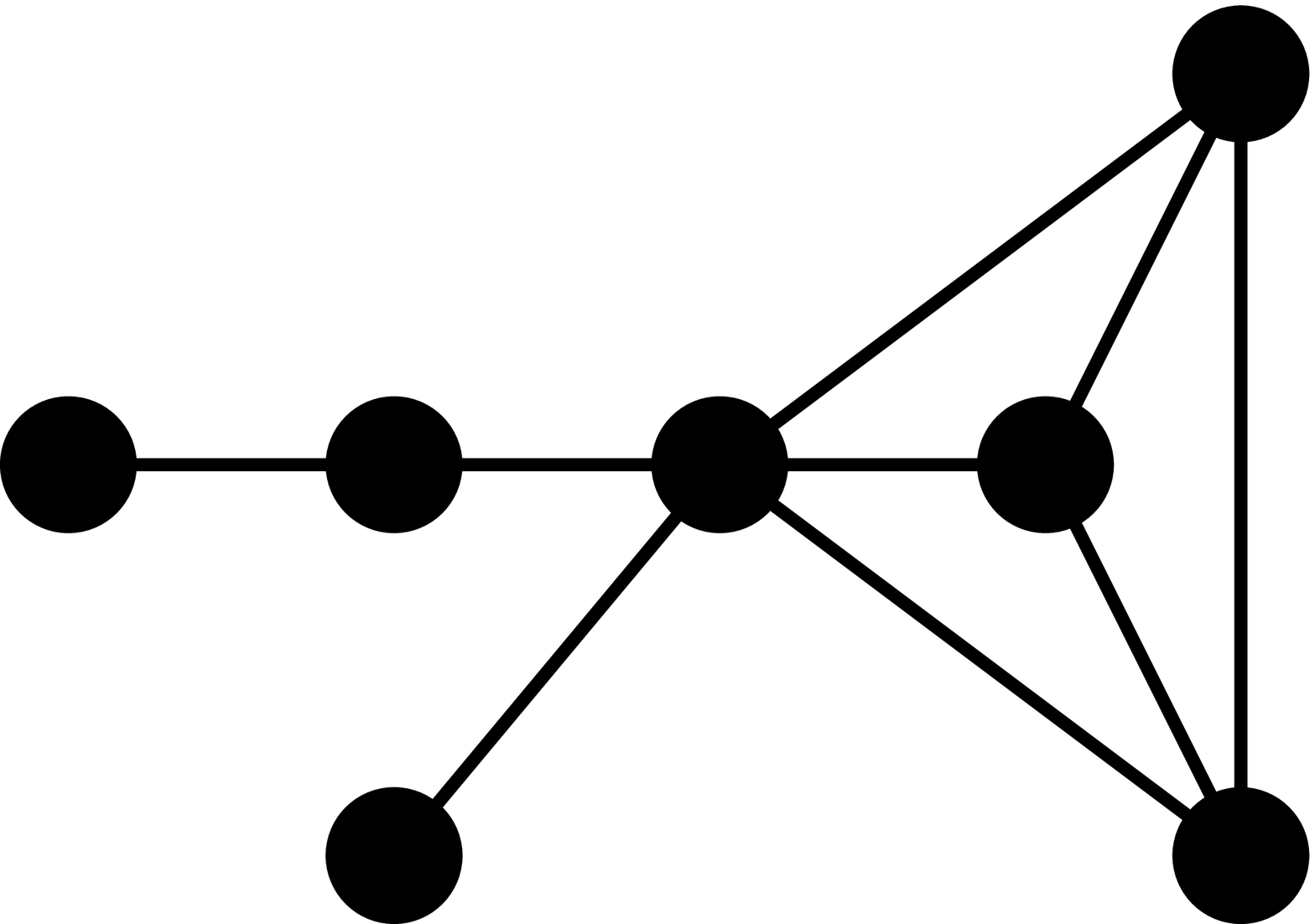} &
\includegraphics[scale=0.11]{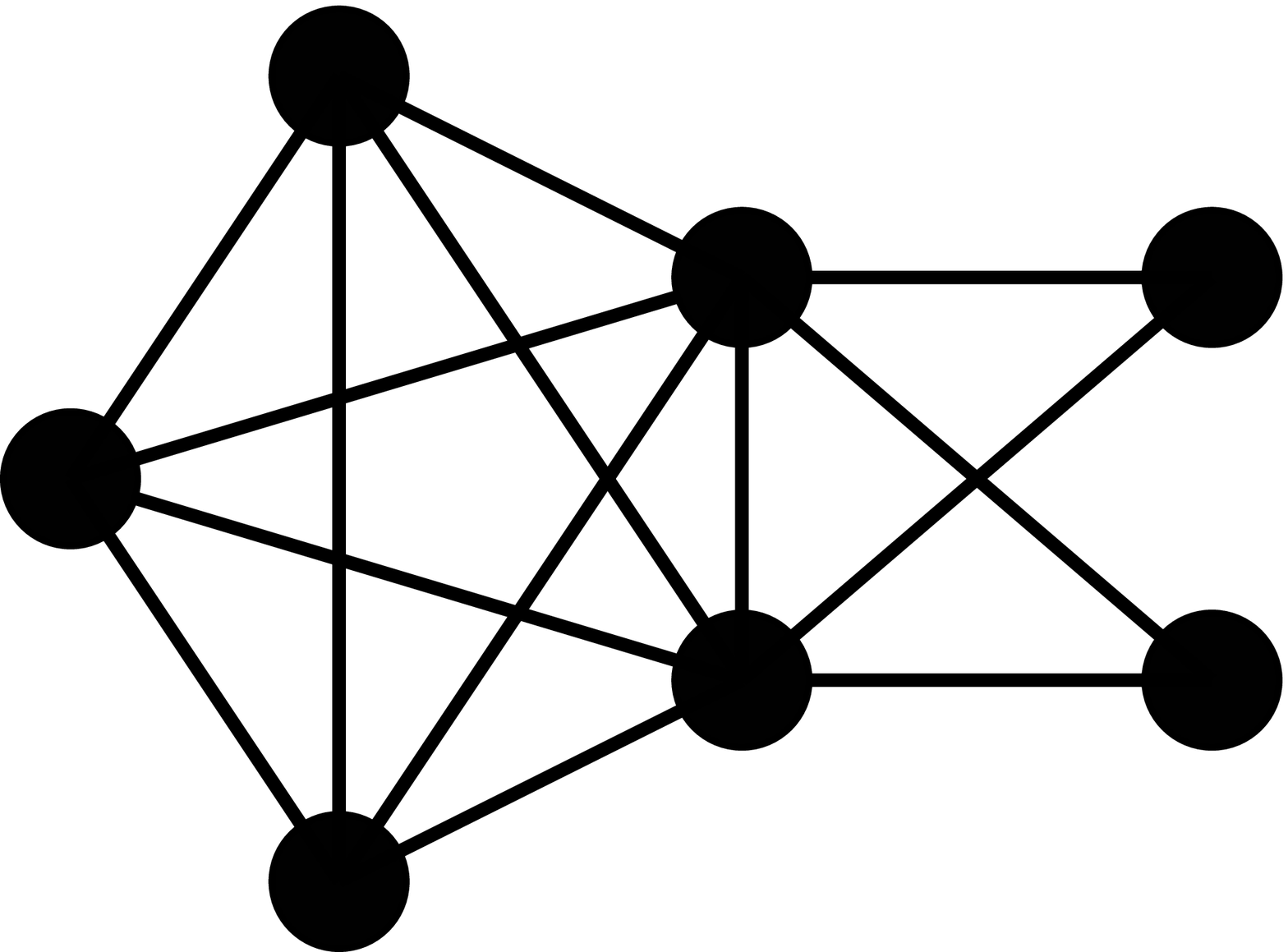}
 \\
$G_{7,1}$ & $G_{7,2}$ & $G_{7,3}$ & $G_{7,4}$ & $G_{7,5}$ \\
&&&&\\

\includegraphics[scale=0.11]{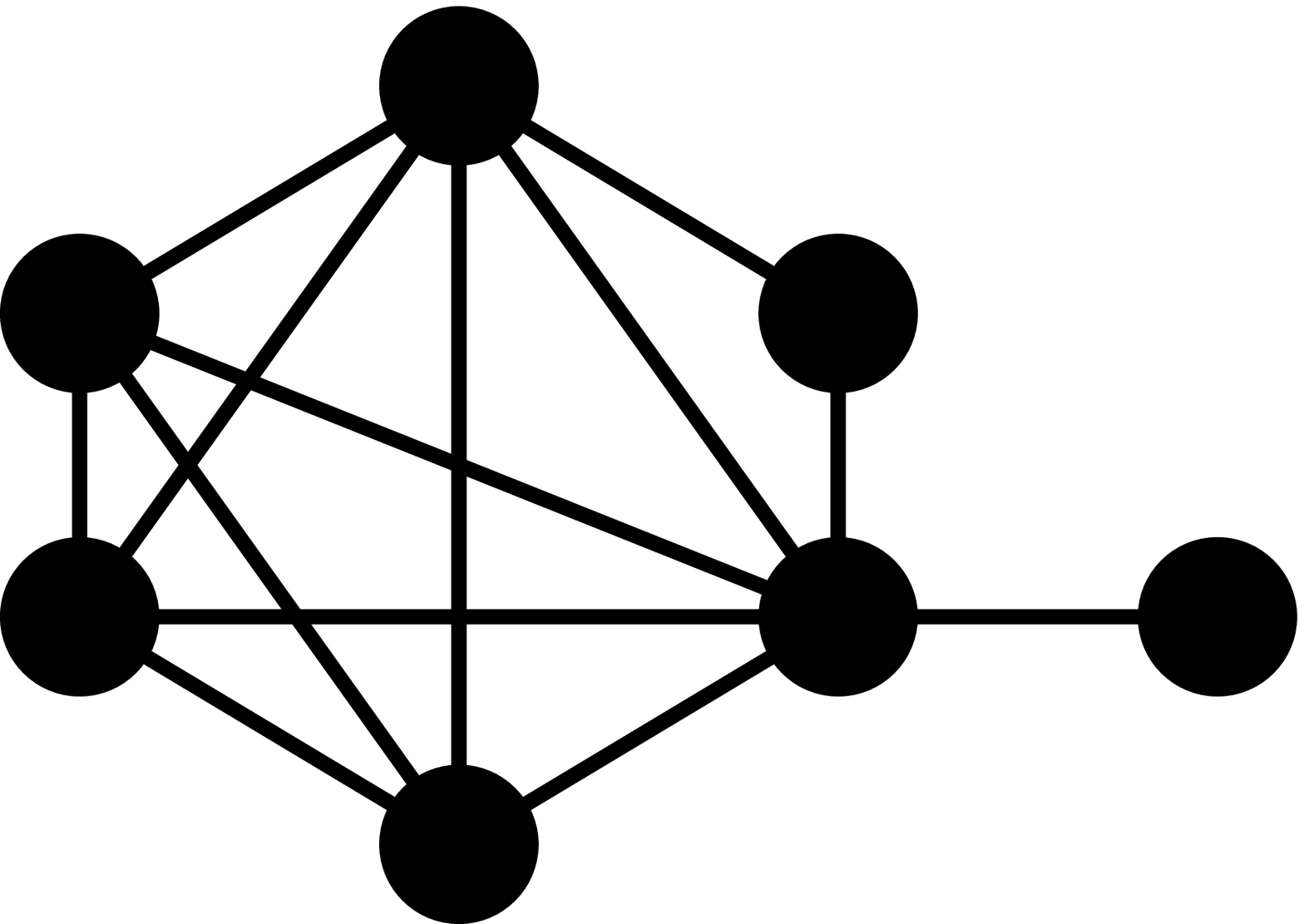} &
\includegraphics[scale=0.11]{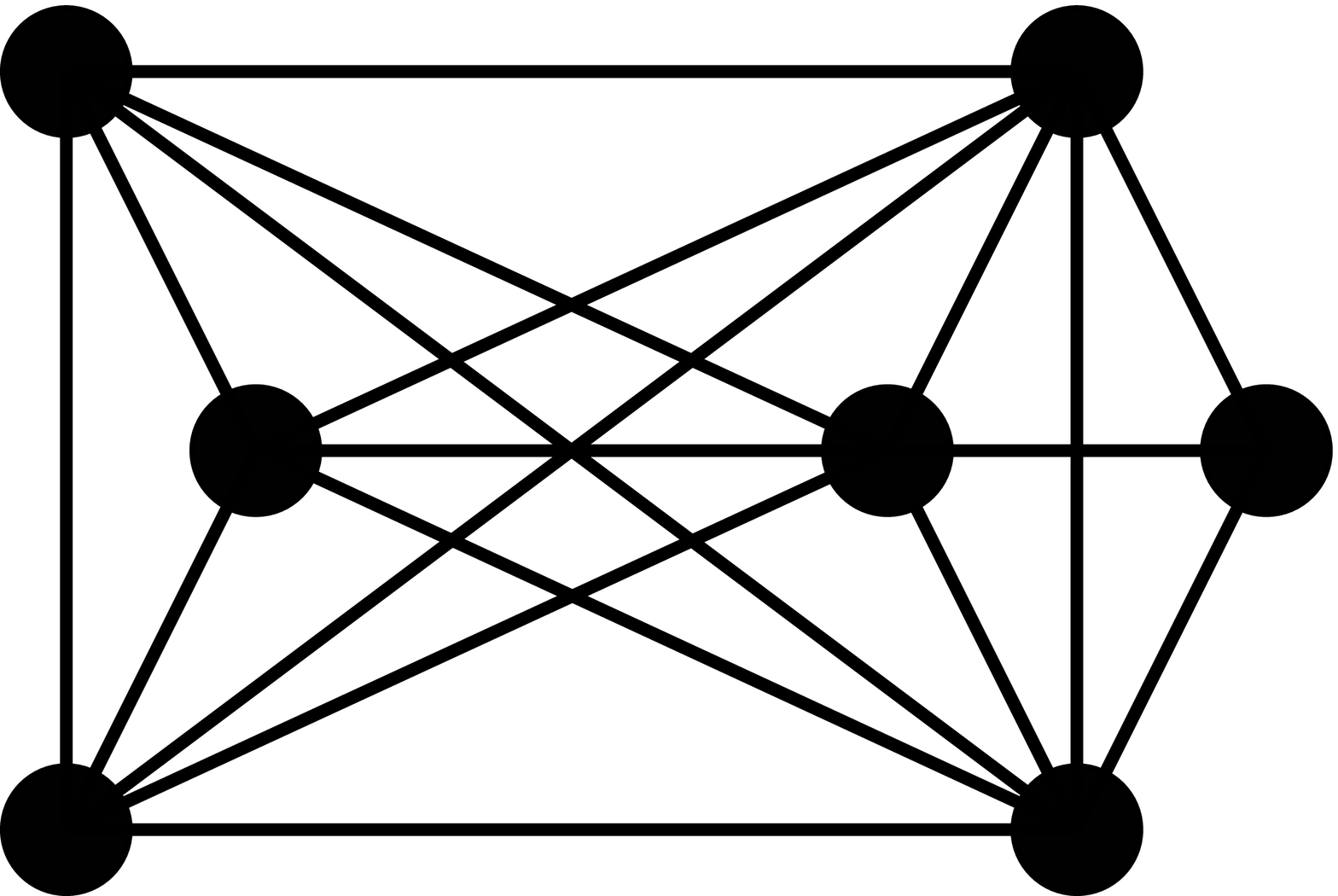} &
\includegraphics[scale=0.11]{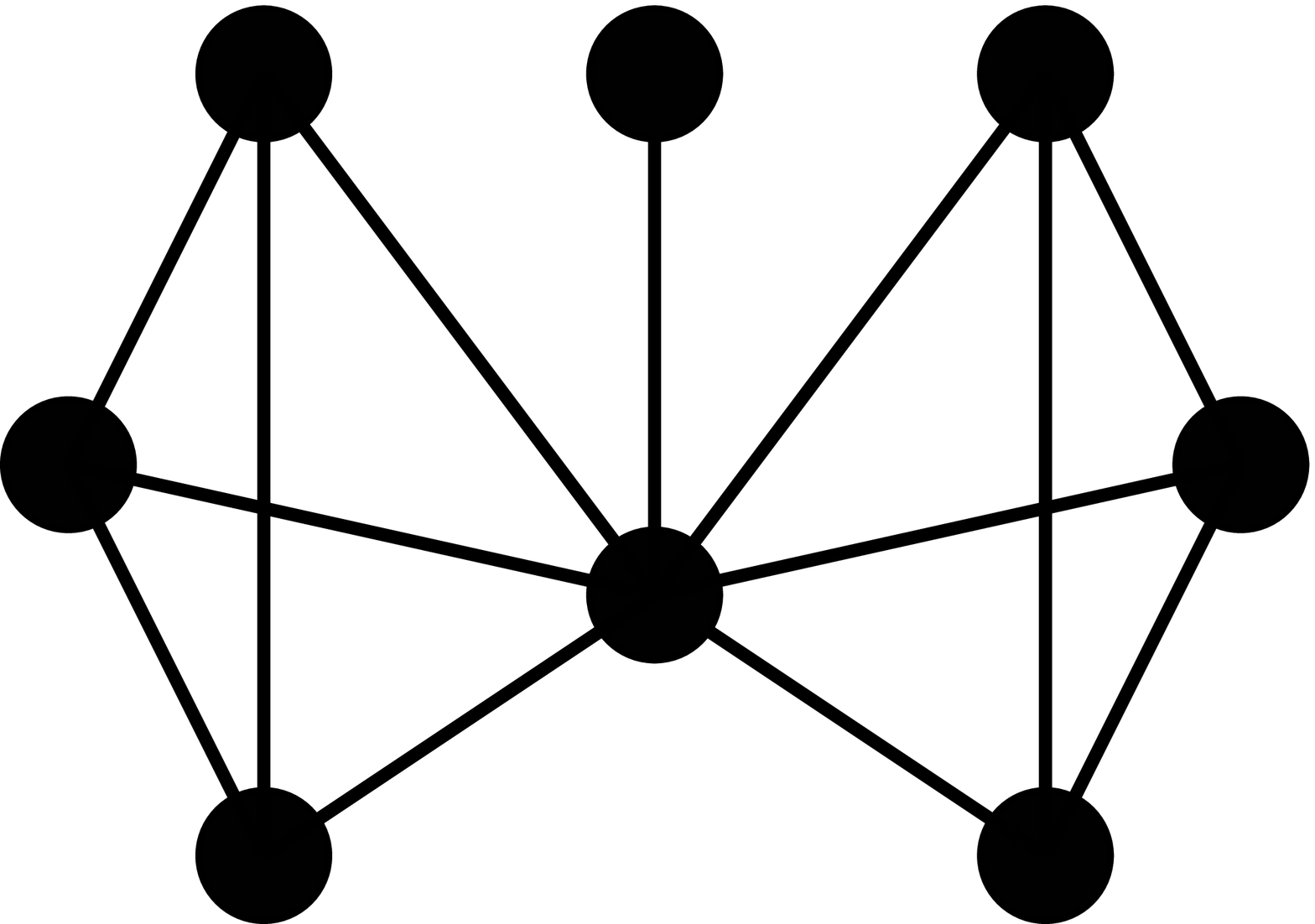}
 &
&
 \\
$G_{7,6}$ & $G_{7,7}$ & $G_{8,1}$ & & \\

\end{tabular}
}
\label{MFS}
\end{figure}

\section{SOME USEFUL LEMMAS}
\label{SomeUsefulLemmas}

A vertex of a graph is called a pendant vertex if it has degree $1$.
\begin{lm}\label{lem:4.2}
Let $H =H^0\uplus H^1$ be a connected graph.
Suppose that $V_f(H^0)\cap V_f(H^1)=\{\alpha\}$
	and $N_{H^0}^s(\alpha)=V_s(H^0)$.
Then $H^1$ is connected.
\end{lm}
\begin{proof}
Put $F=V_f(H^0)\setminus \{\alpha\}$
   and $K=H^0-F$.
Then $F\cap V_f(H^1)=\emptyset$.
Hence $H-F=K\uplus H^1$
	and $H-F$ is connected.
Since $\alpha$ is a unique fat vertex of $K$
   which is adjacent to all the slim vertices of $K$,
   Lemma 15 of \cite{paperI} implies that
   $H^1$ is connected.
\end{proof}

\begin{lm}\label{lm:not_contain_H14H15}
Let $\h$ be a family of isomorphism classes of Hoffman graphs,
	satisfying the following condition:
\[
	[H]\in\h,\; H\not\cong H_2
	\implies |N_H^f(x)|\leq1\quad\forall x\in V_s(H).
\]
Let $H$ be an $\h$-line graph.
Then,
\begin{enumerate}[{\rm (i)}]
	\item if $u\in V_s(H)$,
		then $|N_H^f(u)|\leq2$,
	\item if $u,v$ are distinct slim vertices of $H$,
		then $|N_H^f(u)\cap N_H^f(v)|\leq1$.
\end{enumerate}
\end{lm}
\begin{proof}
See Lemma 23 of \cite{paperI}.
\end{proof}

From \cite[\S6, Problem 6(c)]{exercises},
	we obtain the following lemma:
\begin{lm}
Let $\Gamma$ be a connected slim graph.
If $\Gamma$ is neither a complete graph nor a cycle,
then there exists a non-adjacent pair $\{x,y\}$ in $V(\Gamma)$
such that $\Gamma -\{x,y\}$ is connected.
\label{connected}
\end{lm}

For the remainder of this section,
	we assume $\h=\{[H_2],\ [H_3],\ [H_5]\}$ (cf. Figure~\ref{Hoffmans}).

\begin{lm}\label{lm:new2_1}
Let $H=\biguplus_{i=0}^nH^i$ be a Hoffman graph
   satisfying $[H^j]\in \h$ for $j=0,1,\ldots,n$.
Let $V$ be a subset of $V_s(H)$,
   and let $K=\subgg{V}{H}$.
Then there exist subgraphs $K^i\ (i=0,1,\ldots,n')$ of $K$
   such that
\[
K=\biguplus_{i=0}^{n'}K^i,
\ [K^j]\in \h\cup\{[H_1]\}\mbox{ for }j=0,1,\ldots,n'.
\]
\end{lm}
\begin{proof}
Put $L^i=\subgg{V\cap V_s(H^i)}{H^i}$.
Obviously
   $[L^i]\in \h\cup \{[\phi],[H_1],[H']\}$,
   where
   $H'$ is the sum $H_1\uplus H_1$ of two copies of
   $H_1$ sharing a fat vertex.
Since $K=\biguplus_{i=0}^nL^i$ by \cite[Lemma 12]{paperI},
   the lemma holds.
\end{proof}
\begin{lm}\label{lem:4.4}
Let $\Gamma$ be a connected slim $\h$-line graph.
Then there exists
	a connected strict $\h$-cover graph $H=\biguplus_{i=0}^n H^i$
	of $\Gamma$.
Conversely,
	if $H=\biguplus_{i=0}^n H^i$ is a connected graph
	with $[H^i]\in\h$ and $n>0$,
	then $\Gamma=\langle V_s(H)\rangle_H$
is connected.
\end{lm}
\begin{proof}
The first part follows from Example 22 of \cite{paperI}.
We prove the second part by induction on $n$.
The assertion is easy to verify when $n=1$.
Suppose $n>1$,
	and let $H'=\subg{\bigcup_{i=1}^n V(H^i)}{H}$.
Since $H$ is connected, $V_f(H^0)\cap V_f(H')\neq\emptyset$.
Pick $\alpha\in V_f(H^0)\cap V_f(H')$.
Then every slim vertex of $H^0$ is adjacent to $\alpha$,
	and hence every slim vertex of $H^0$ has a slim neighbour in $H'$.
Since $H'=\biguplus_{i=1}^n H^i$ is connected by inductive hypothesis,
	we see that $\Gamma$ is connected.
\end{proof}

\begin{lm}
If $\biguplus_{i=0}^{m_1}K^i=\biguplus_{i=0}^{m_2}L^i$
	and $[K^i],[L^i]\in\h$ for each $i$,
	then $m_1=m_2$,
	and
\[
	\{K^i\mid 0\le i\le m_1\}=\{L^i\mid 0\le i\le m_2\}.
\]
\label{newlemma}
\end{lm}
\begin{proof}
It suffices to prove $K^i=L^j$ whenever $V_s(K^i)\cap V_s(L^j)\neq\emptyset$. 
We may suppose without loss of generality that $i=j=0$.
If $K^0\cong H_2$,
	then $K^0$ has a unique slim vertex,
	so $V_s(K^0)\subset V_s(L^0)$.
By Definition~\ref{df:1}(iii),
	we have $K^0\subset L^0$.
This implies $|V_f(L^0)|\geq2$,
	hence $L^0\cong H_2$,
	and therefore $K^0=L^0$.
The same conclusion holds when $L^0\cong H_2$,
	so we suppose $[K^0],[L^0]\in \{[H_3],[H_5]\}$
	for the rest of the proof.
If $s_1\in V_s(K^0)\cap V_s(L^0)$,
	then there exists $s_2\in V_s(K^0)$ not adjacent to $s_1$.
Since $s_1$ and $s_2$ have a common fat neighbour in $K^0$,
	Definition~\ref{df:1}(iv) forces $s_2\in V_s(L^0)$.
This implies $V_s(K^0)\subset V_s(L^0)$ if $K_0\cong H_3$.
If $K^0\cong H_5$,
	then consider the third slim vertex $s_3$ of $K^0$.
We may assume without loss of generality that $s_3$ is not adjacent to $s_1$.
Since $s_1$ and $s_3$ have a common fat neighbour in $K^0$,
	Definition~\ref{df:1}(iv) forces $s_3\in V_s(L^0)$.
Thus $V_s(K^0)\subset V_s(L^0)$.
Switching the roles of $K^0$ and $L^0$,
	we obtain $V_s(L^0)\subset V_s(K^0)$.
Therefore
	we conclude $V_s(K^0)=V_s(L^0)$,
	and hence $K^0=L^0$.
\end{proof}

\begin{lm}\label{lm:subgraph}
Suppose $H=H^0\uplus H^1$,
	$S\subset V_s(H^1)$,
	and $H^2=\subgg{S}{H^1}$.
Then $\subg{V(H^0)\cup V(H^2)}{H}=H^0\uplus H^2$.
\end{lm}
\begin{proof}
Routine verification.
\end{proof}

\begin{center}
\begin{table}[htbp]
\begin{tabular}{|c||c|c|c|c|}
\hline
& (i)& (ii)& (iii)& (iv)\\
\hline
&&&& \\
{\large \raise2ex\hbox{$H^0$}} &
\includegraphics[scale=0.07]
{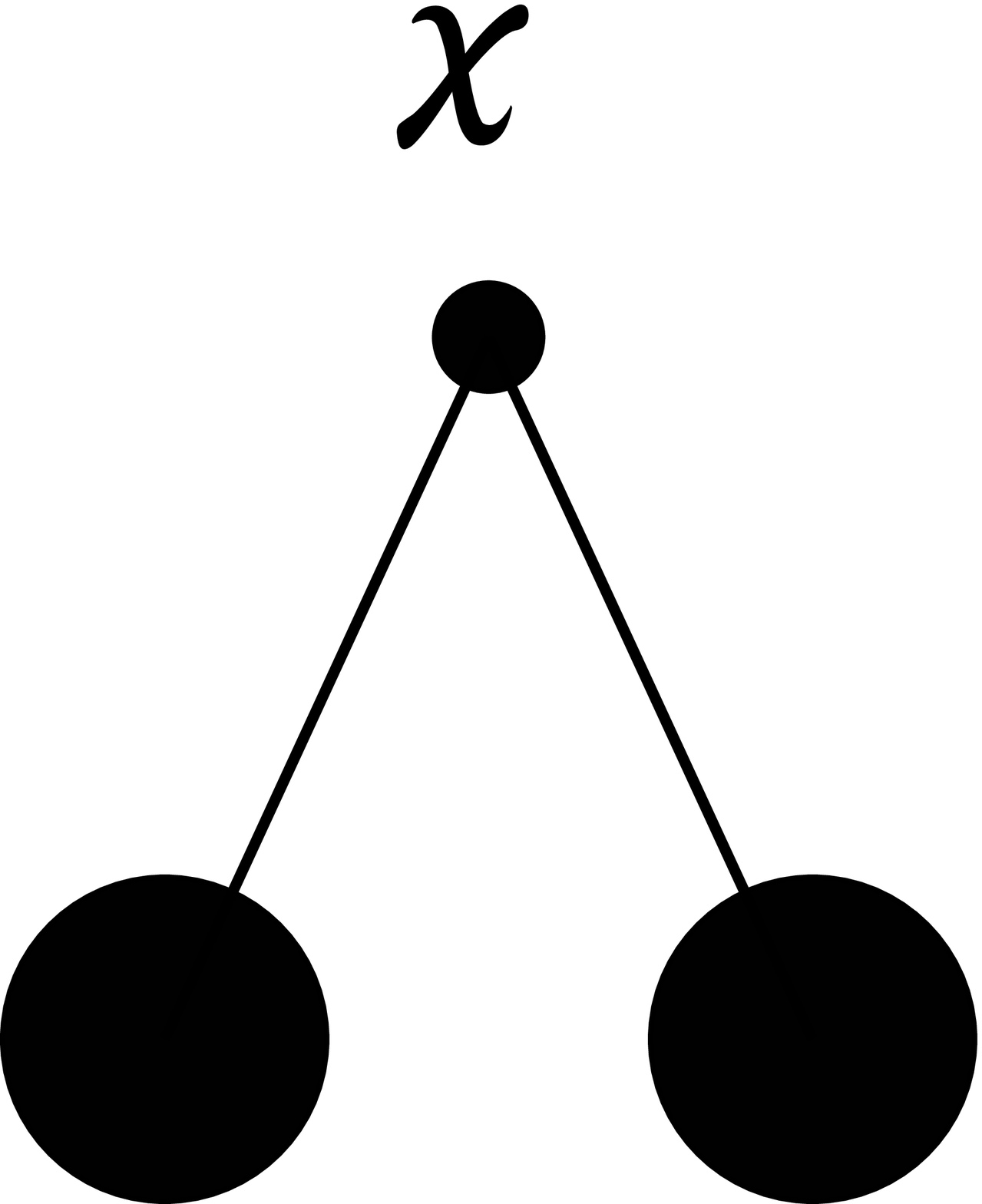} &
\includegraphics[scale=0.07]
{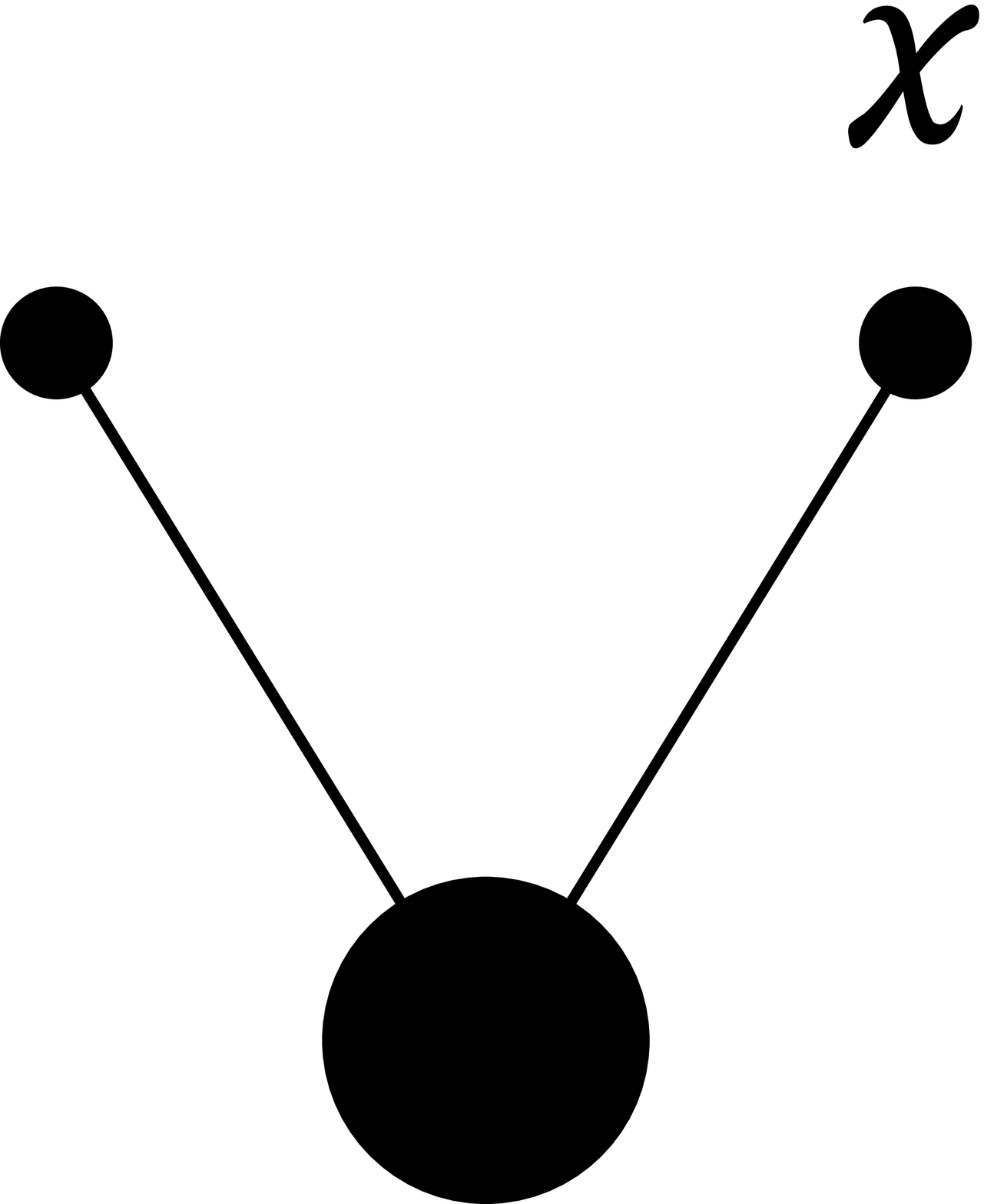} &
\includegraphics[scale=0.07]
{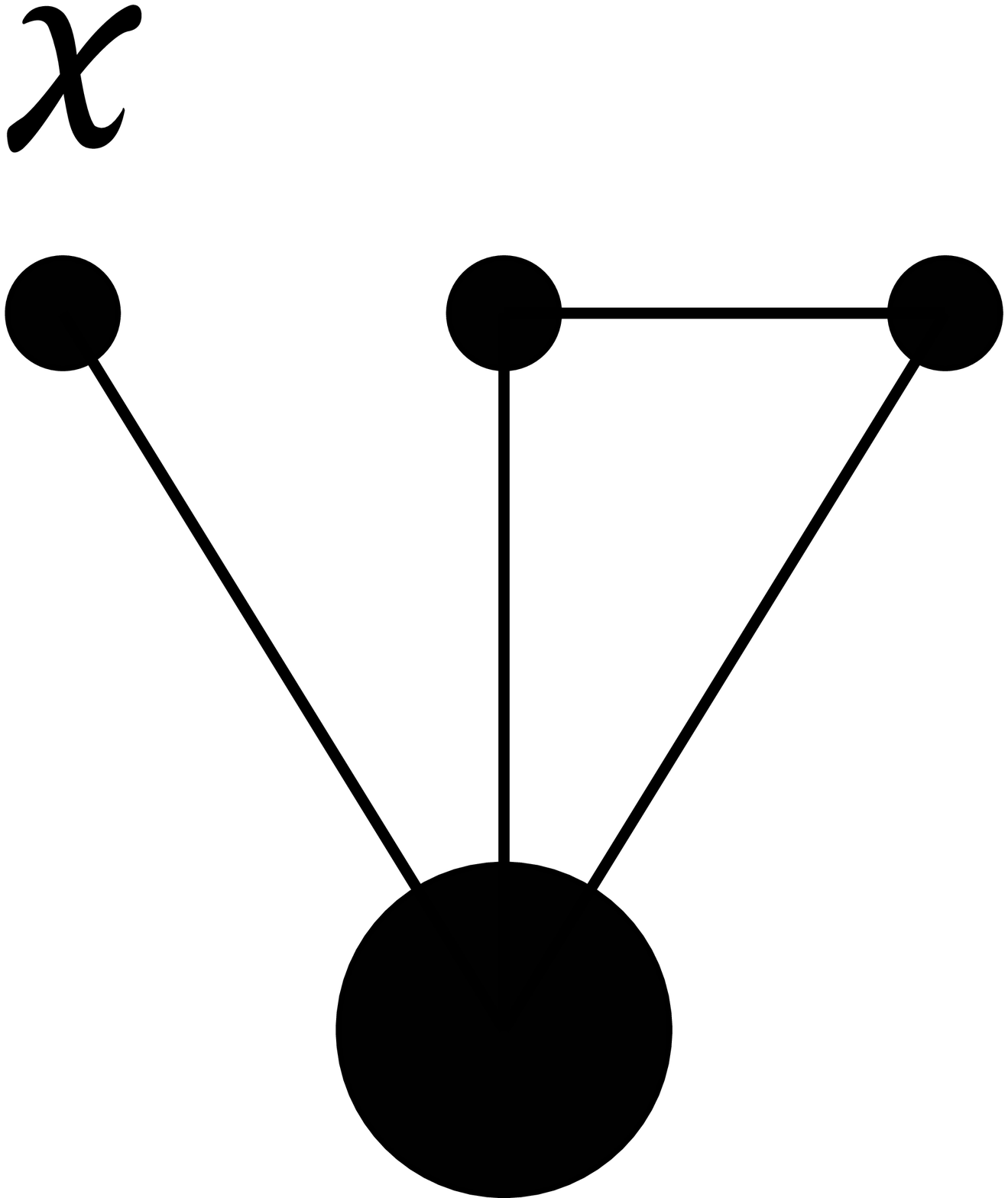} &
\includegraphics[scale=0.07]
{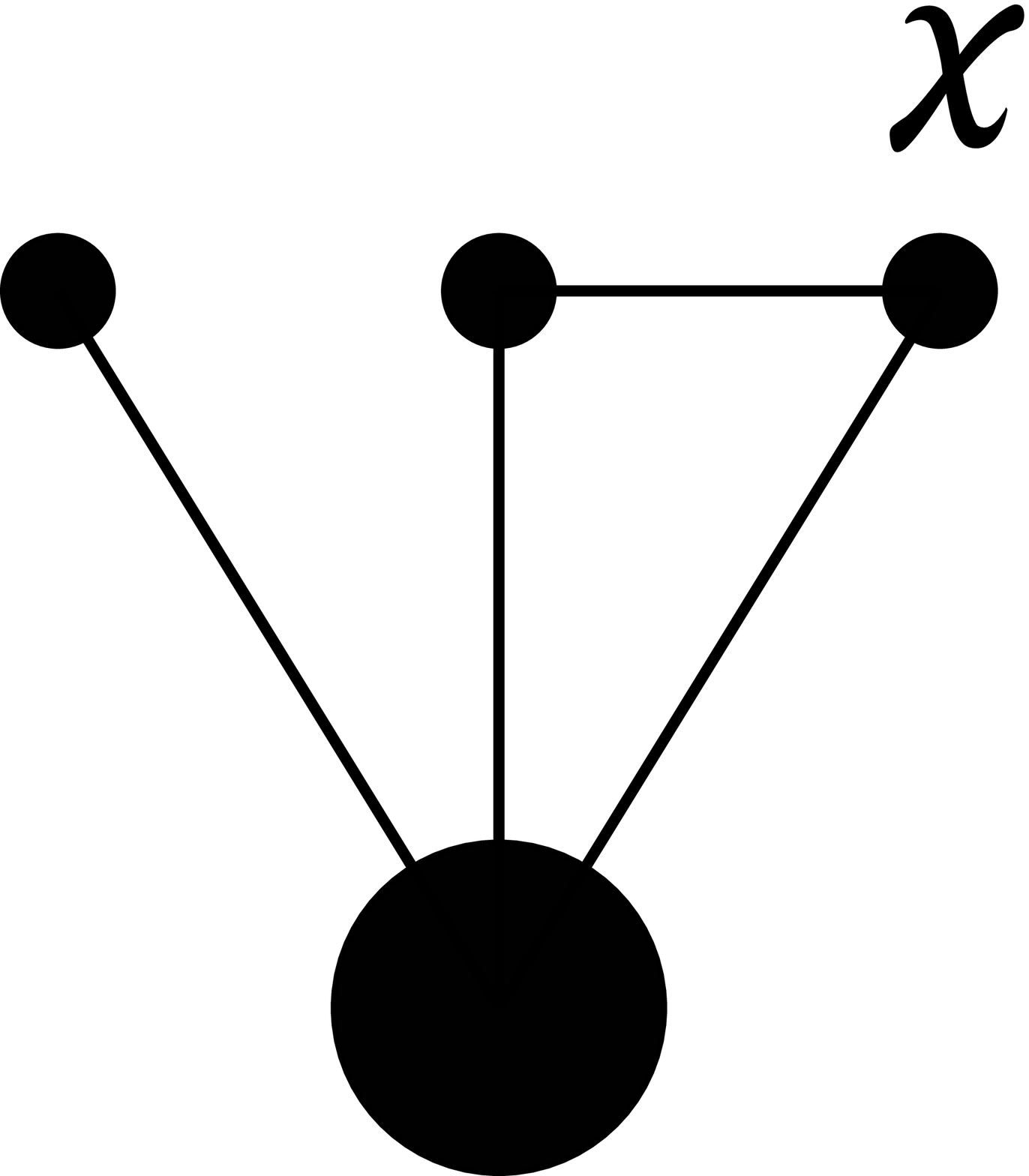}
 \\

\hline
&&&& \\
{\large \raise2ex\hbox{$H^0-x$}} &
{\large \raise2ex\hbox{$\phi$}} &
\includegraphics[scale=0.07]
{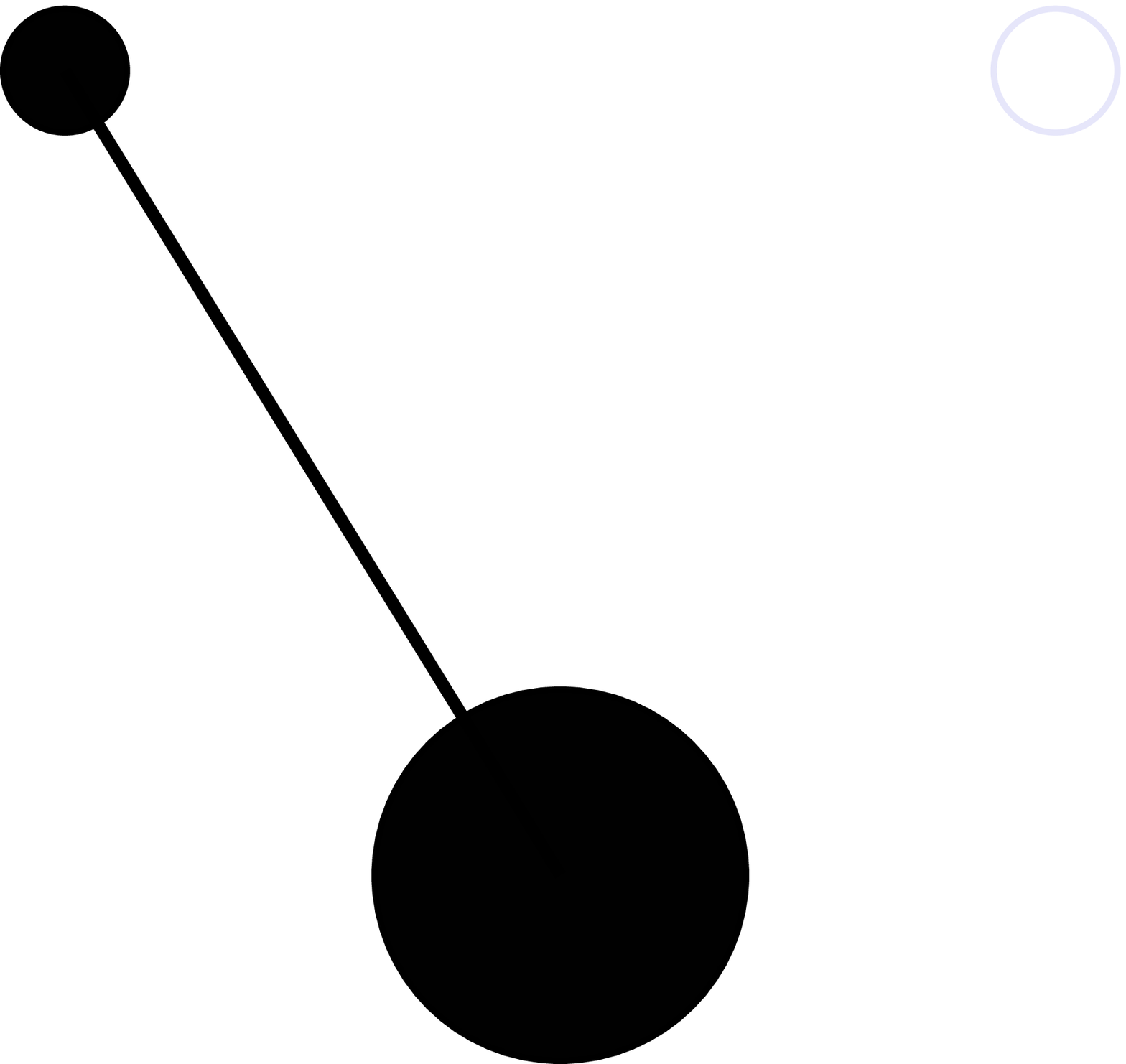} &
\includegraphics[scale=0.07]
{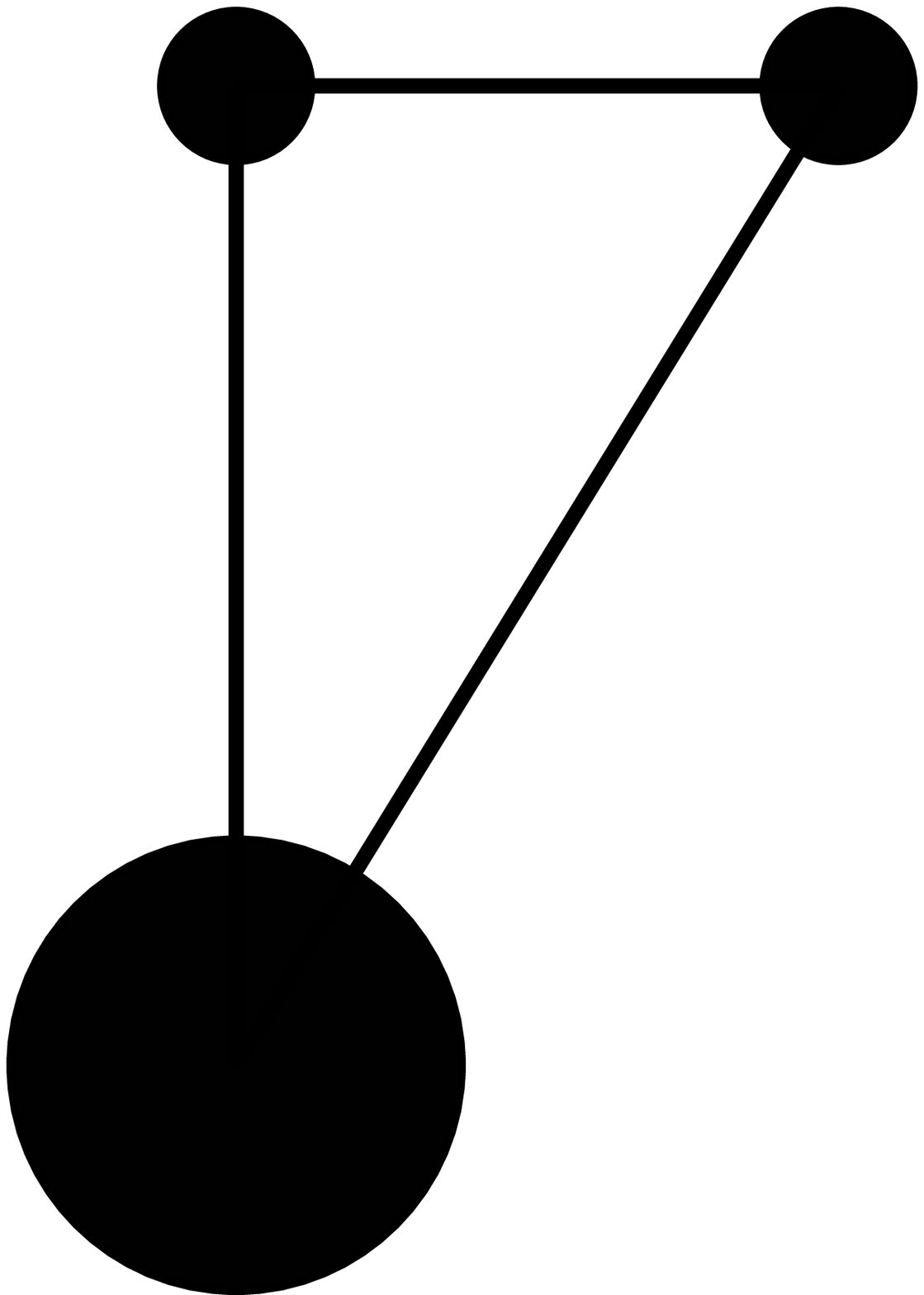} &
\includegraphics[scale=0.07]
{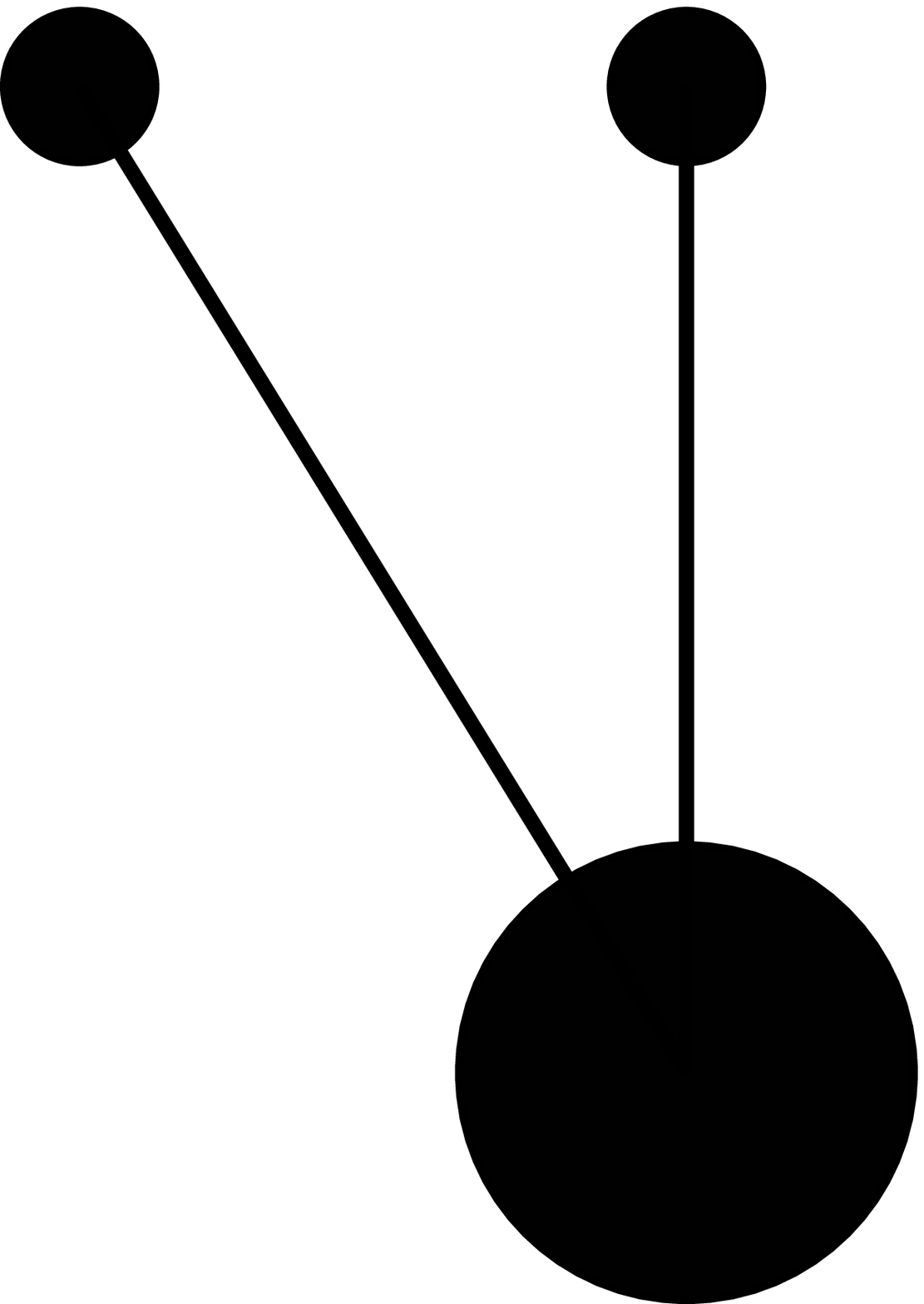}
 \\

\hline
&&&& \\
{\large \raise2ex\hbox{$\tilde{H}^0$}} &
{\large \raise2ex\hbox{$\phi$}} &
\includegraphics[scale=0.07]
{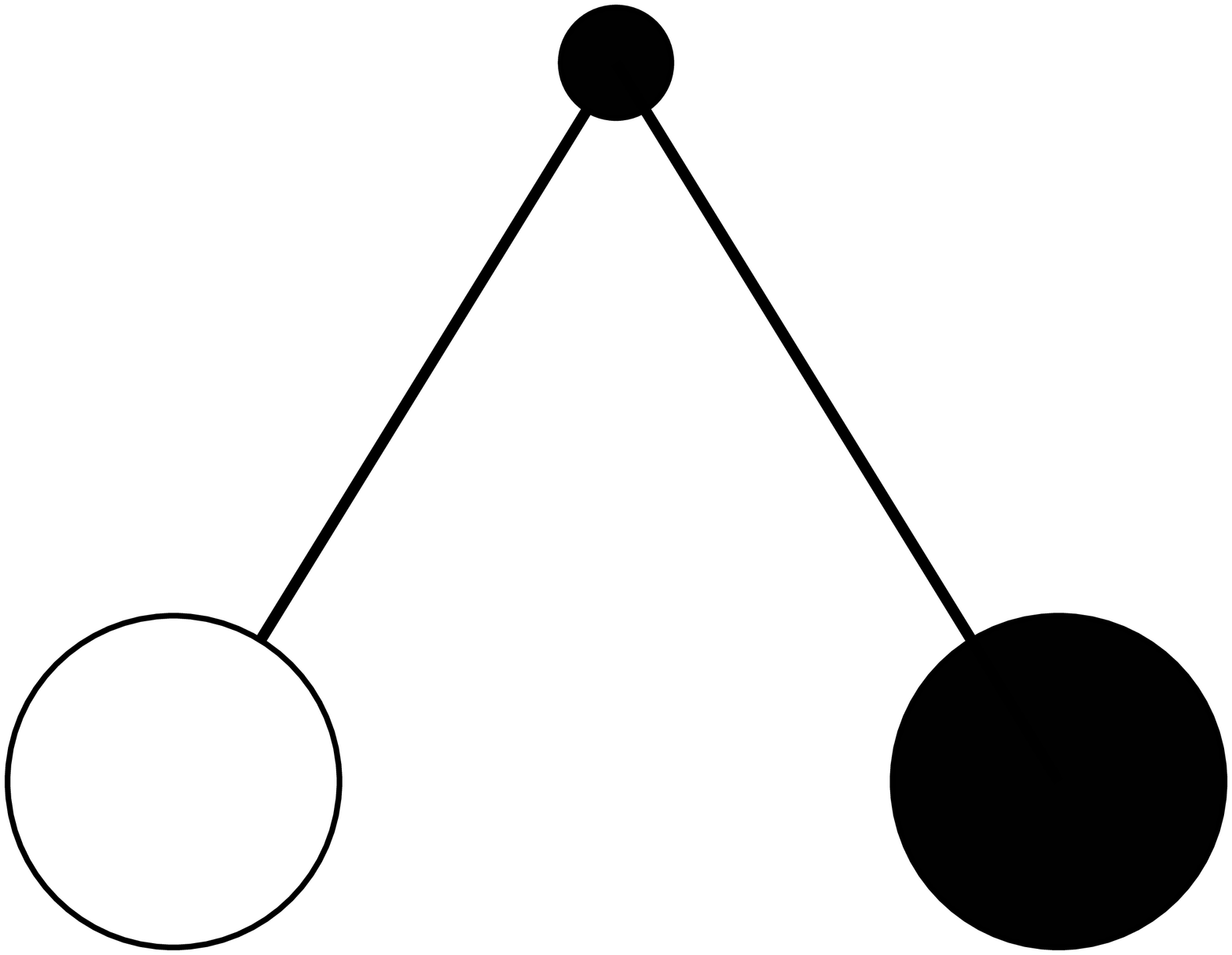} &
\includegraphics[scale=0.07]
{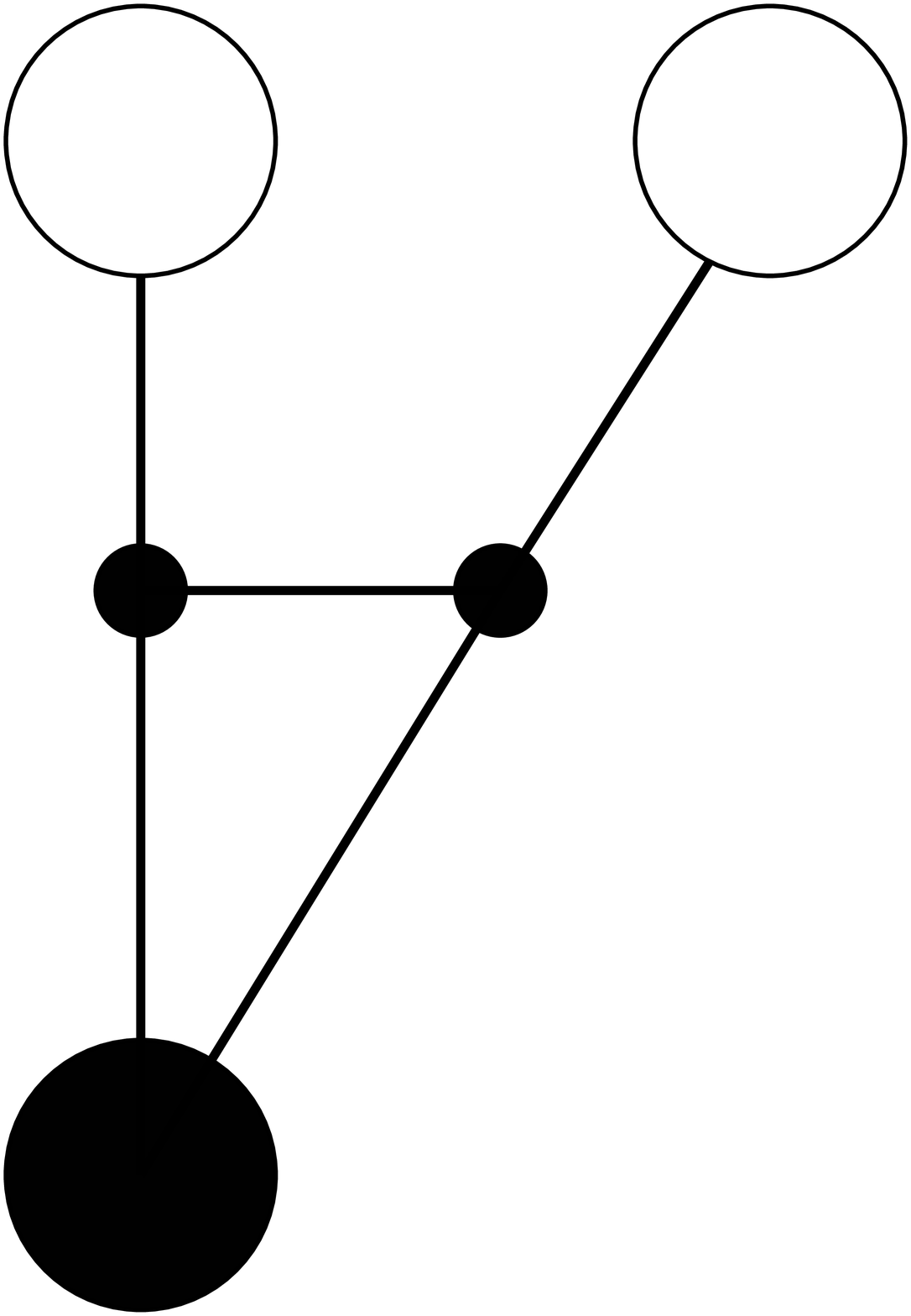} &
\includegraphics[scale=0.07]
{./X0_xh5_2.eps}
 \\

\hline
\end{tabular} \caption{}\label{table2}
\end{table}
\end{center}

\begin{lm}\label{lm:newLemma001}
Let $H=H^0\uplus H^1$ be a connected Hoffman graph
   satisfying $[H^0]\in \h$.
Let $x$ be a slim vertex of $H^0$.
Then there exists a strict $\h$-cover graph
   $\tilde{H}=\tilde{H}^0\uplus H^1$ of $H-x$,
   and one of the following holds:
\begin{enumerate}[(i)]
\item $\tilde{H}^0=\phi$,
\item $\tilde{H}^0\cong H_2$,
   and one of the fat vertices of $\tilde{H}^0$
   is a pendant vertex in $H$,
\item $\tilde{H}^0= K^1\uplus K^2$,
   $K^1\cong K^2\cong H_2$,
   $K^1$ and $K^2$ have a fat vertex in common,
   and the other fat vertices of $\tilde{H}^0$
   are pendant vertices in $H$,
\item $\tilde{H^0}\cong H_3$.
\end{enumerate}

\end{lm}
\begin{proof}
This is shown in the proof of Theorem~31 in \cite{paperI},
   using Table~1, Lemma~12 and Lemma~13 in \cite{paperI}.
\end{proof}

For a Hoffman graph $H=\biguplus_{i=0}^nH^i$
	and a subset $J$ of $\{0,1,\ldots,n\}$,
	we write $H(J)=\biguplus_{i\in J}H^i$.
\begin{lm}\label{lm:Fpre2}
Let $H=\biguplus_{i=0}^nH^i$ be a connected Hoffman graph
	satisfying $H^j\cong H_2,H_3$ or $H_5$
	for $j=0,1,\ldots,n$.
Let $V$ be a subset of $V_s(H)$
	such that $\subgg{V}{H}$ is connected.
Let $I=\{i\mid H^i\cong H_2,\ 0\le i\le n\}$,
	and let $I'=\{i\in I\mid V_s(H^i)\subset V\}$.
Then,
\begin{enumerate}[(i)]
\item if $I'\neq\emptyset$,
	then $H(I')$ is connected,
	and in particular,
	$H(I)$ is connected,
\item if $I\neq\emptyset$,
	then $V_f(H(I))=V_f(H)$.
\end{enumerate}
\end{lm}
\begin{proof}
Put $J=\{i\mid 0\le i\le n,\ V_s(H^i)\cap V\neq\emptyset\}$
   so that $I'=I\cap J$.
Since $\subgg{V}{H}$ is connected,
	so is $H(J)$.
Since the removal of $V_s(H^i)$ with $i\in J\setminus I'$
	preserves connectivity
	by Lemma~\ref{lem:4.2},
	we conclude that $H(I')$ is connected.

Suppose $V_f(H(I))\neq V_f(H)$.
Then there exists a fat vertex $f\in V_f(H)\setminus V_f(H(I))$.
Since $\subgg{N_H^s(f)}{H}$ has the unique fat vertex $f$,
	it is a connected component of $H$.
But this contradicts the assumption
	that $H$ is connected and $I\neq\emptyset$.
Hence $V_f(H(I))=V_f(H)$.
\end{proof}

\section{MAIN THEOREM: THE MINIMAL FORBIDDEN SUBGRAPHS}
\label{proofofmain}

In this section,
	we assume $\h=\{[H_2],\ [H_3],\ [H_5]\}$ (cf. Figure~\ref{Hoffmans}).
Let $F_1,F_2,\ldots,F_9$ be the Hoffman graphs depicted in Figure~\ref{FG2}.



\begin{center}
\begin{figure}[hbtp]
\caption{}
\begin{tabular}{ccccc}
\includegraphics[scale=0.11]
{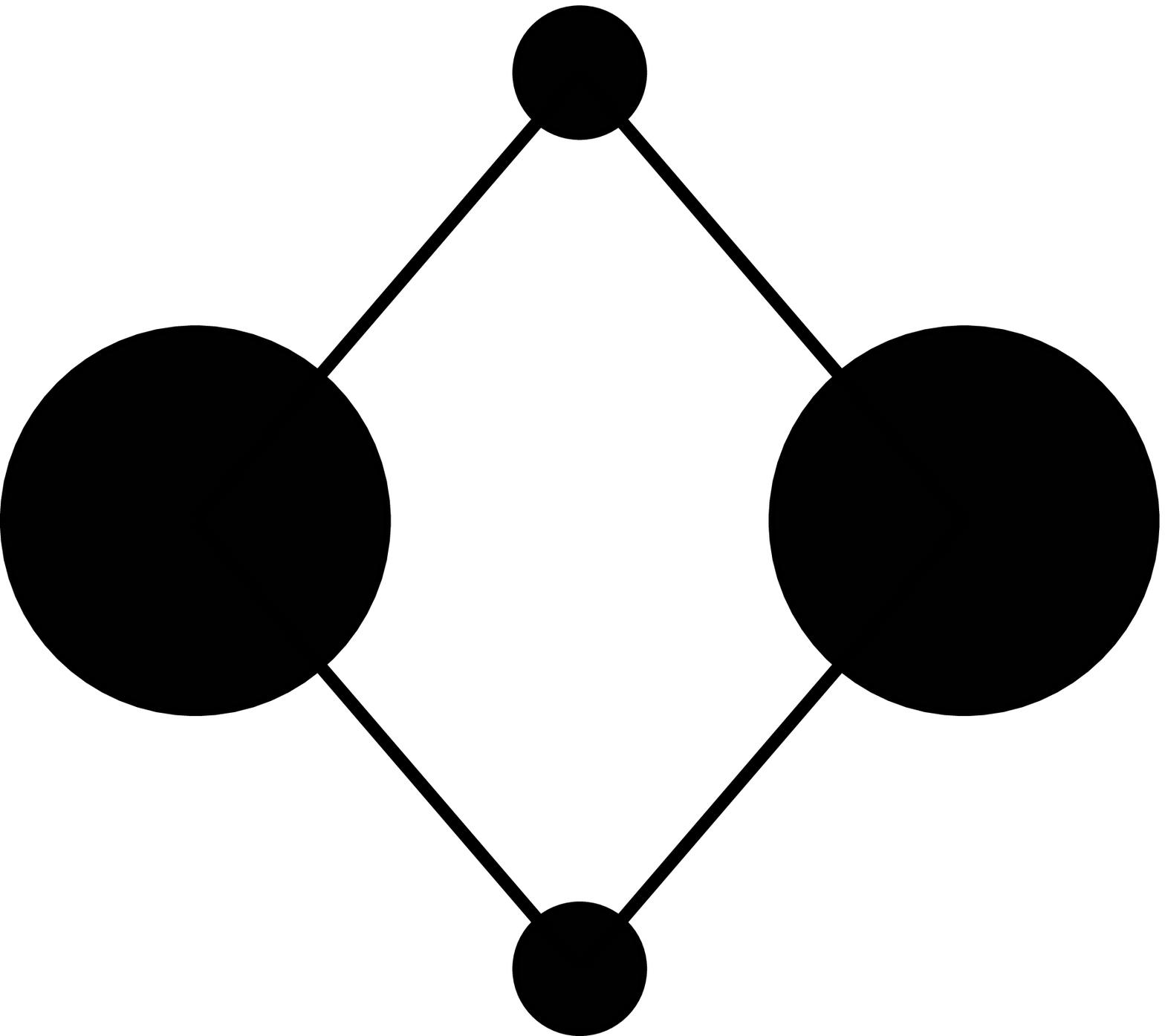} &
\includegraphics[scale=0.11]
{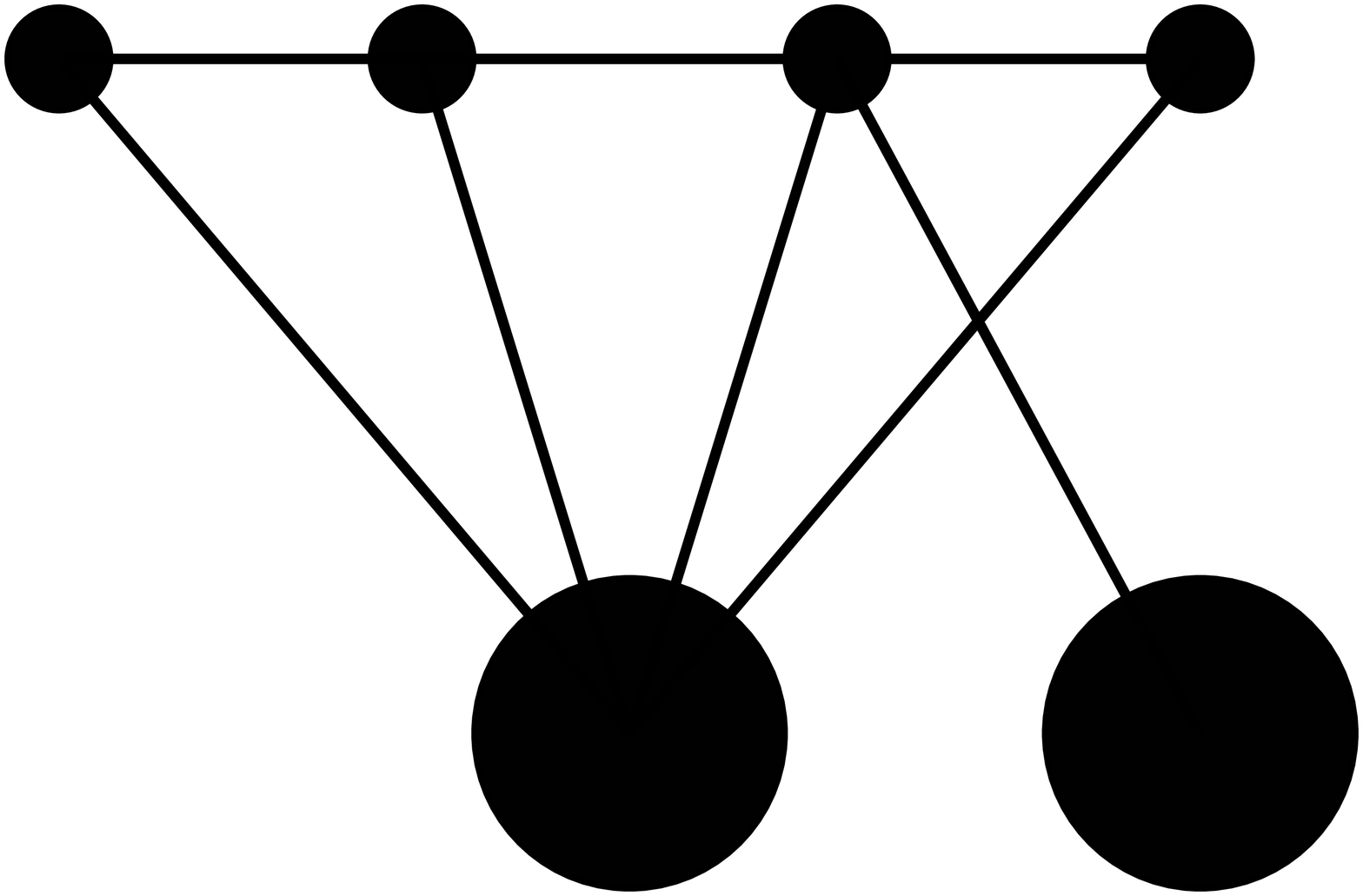} &
\includegraphics[scale=0.11]
{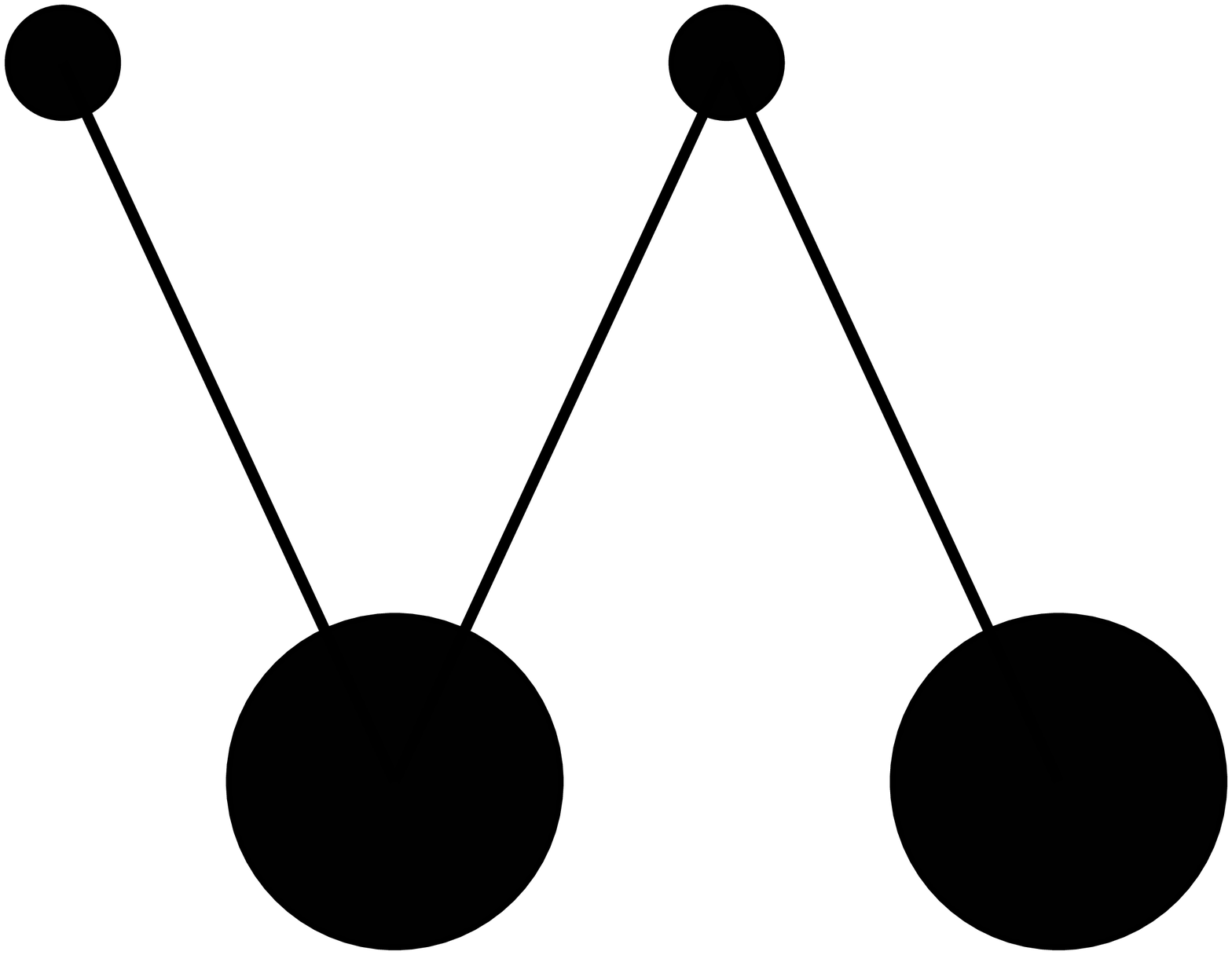} &
\includegraphics[scale=0.11]
{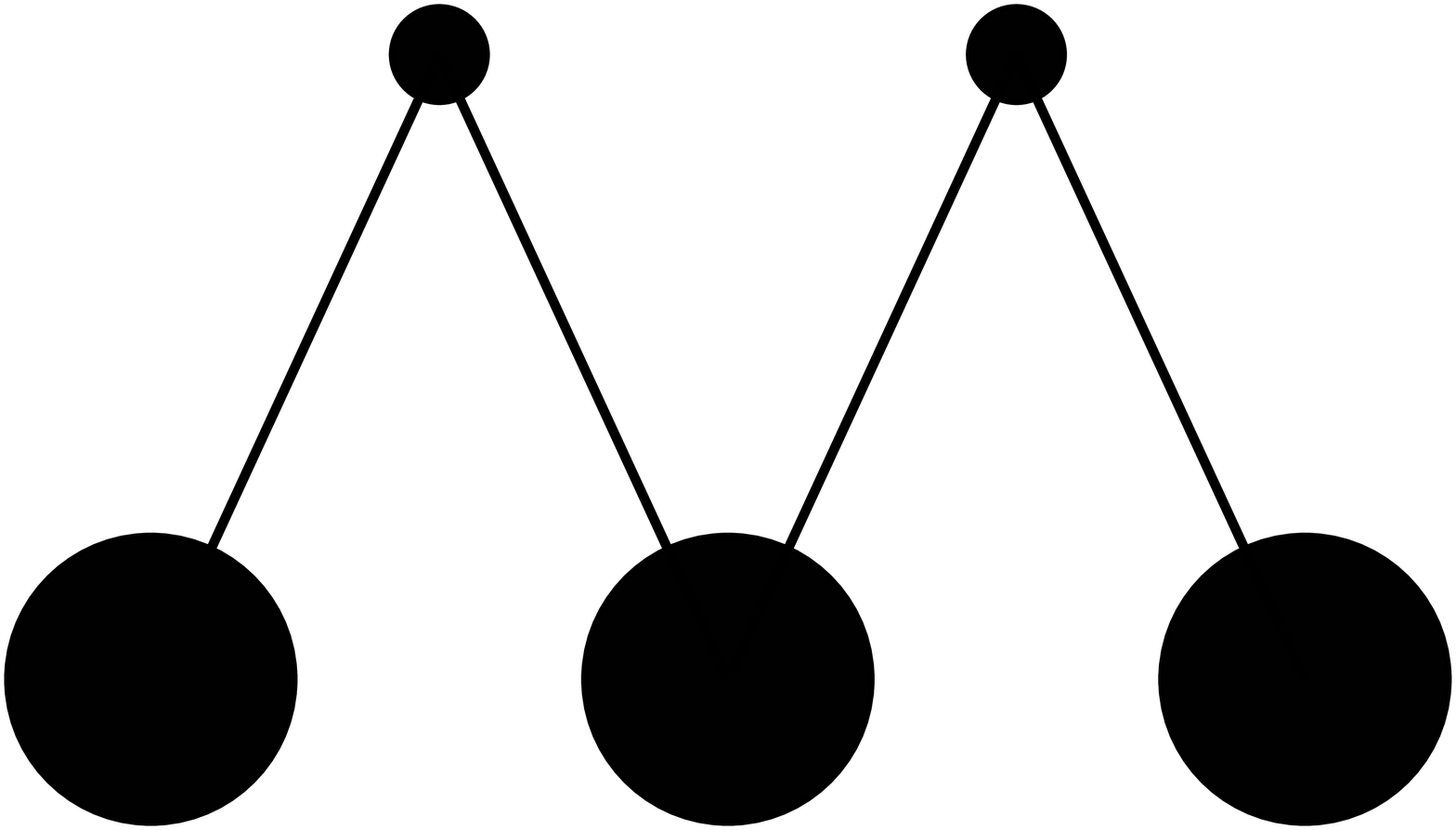} &
\includegraphics[scale=0.11]
{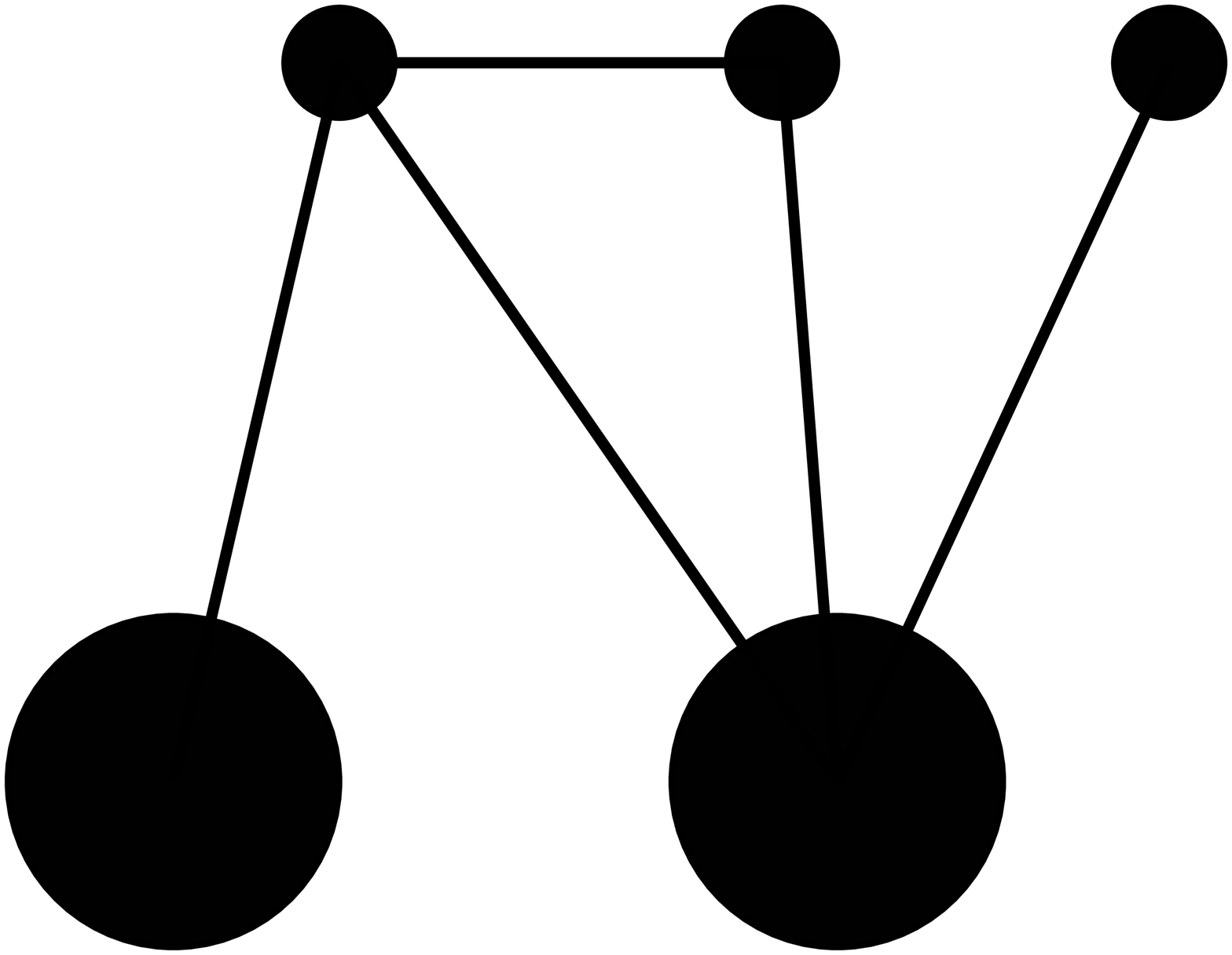}\\
$F_1$ & $F_2$ & $F_3$ & $F_4$ & $F_5$\\
&&& \\
\includegraphics[scale=0.11]
{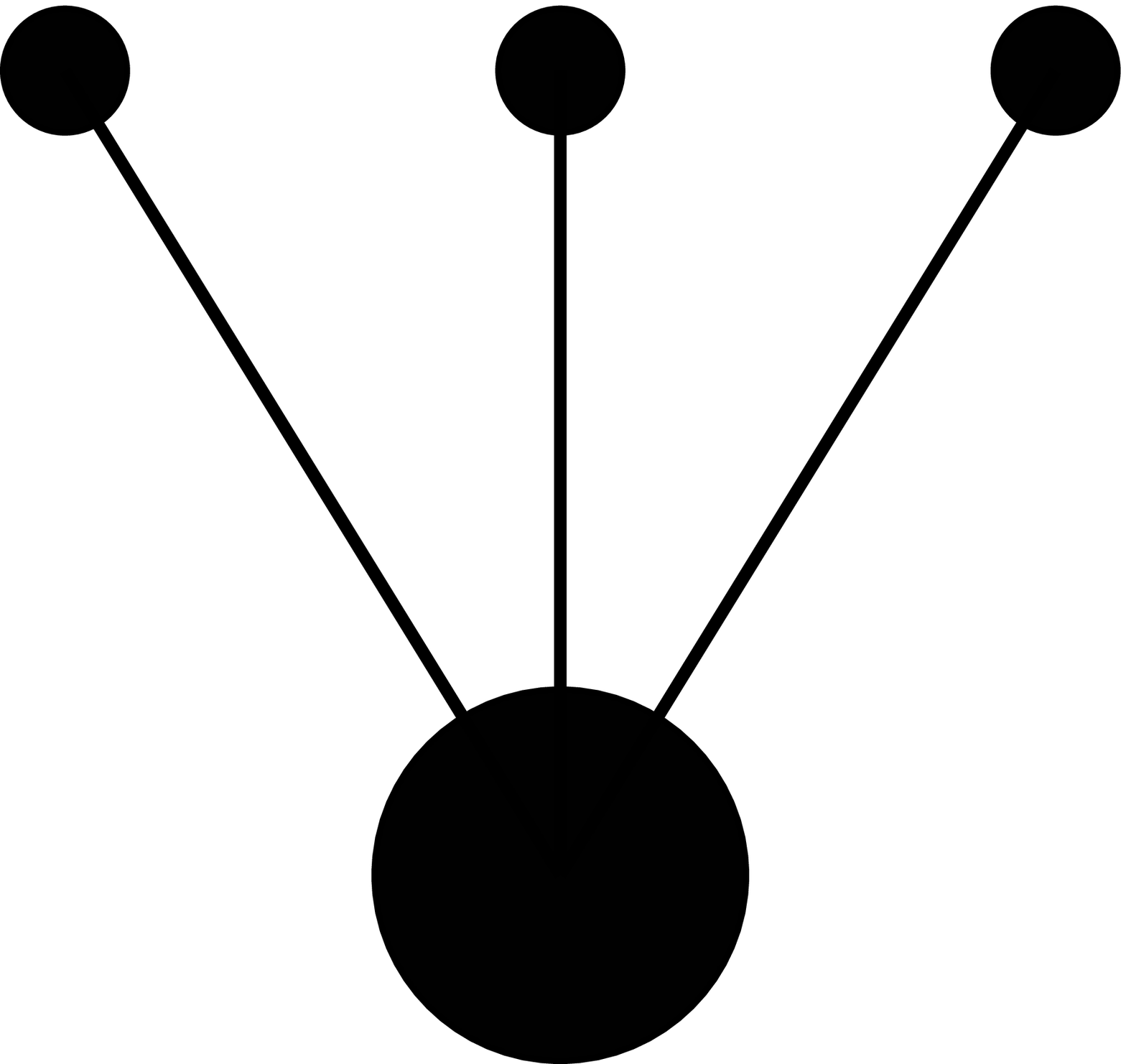} &
\includegraphics[scale=0.11]
{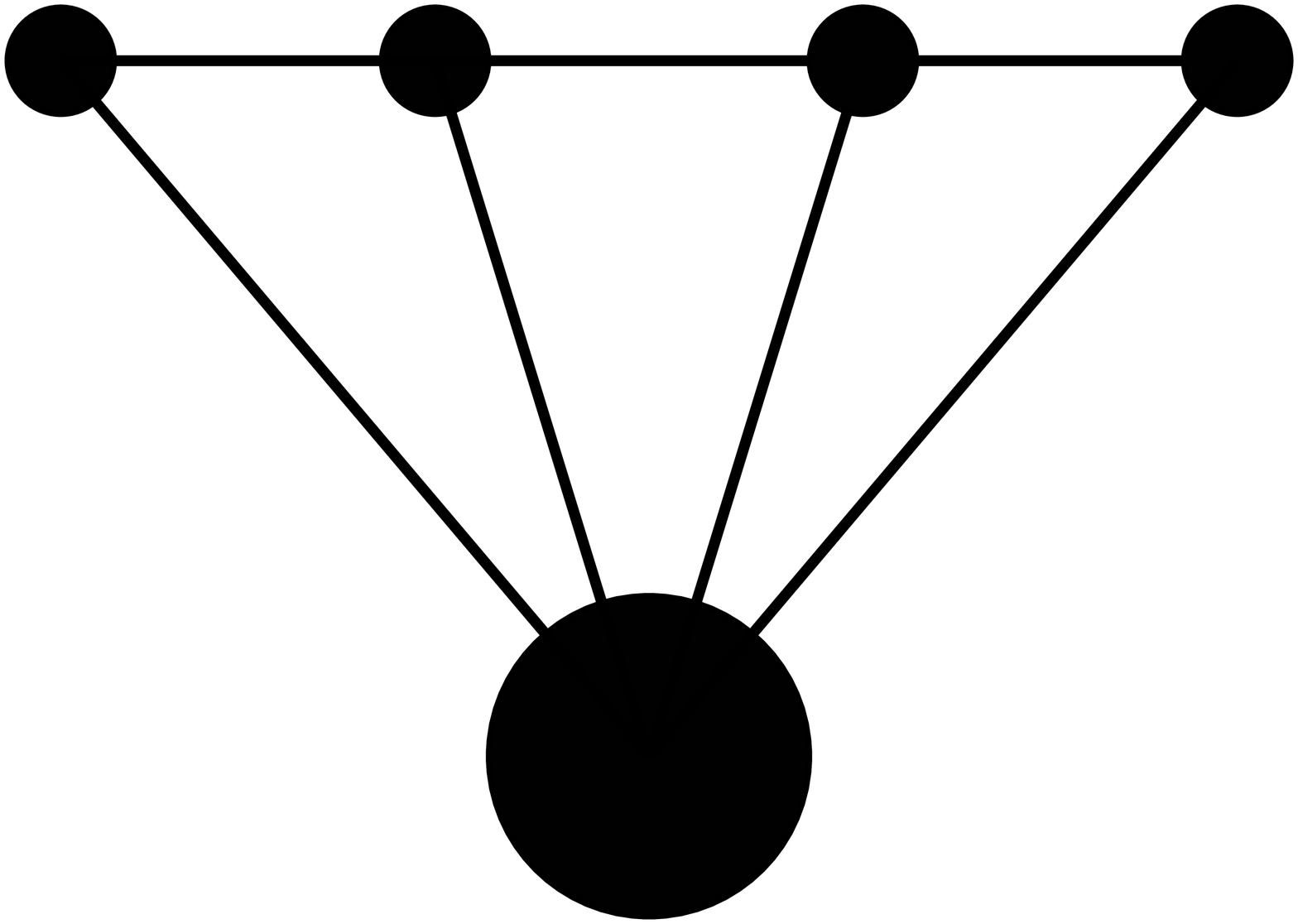} &
\includegraphics[scale=0.11]
{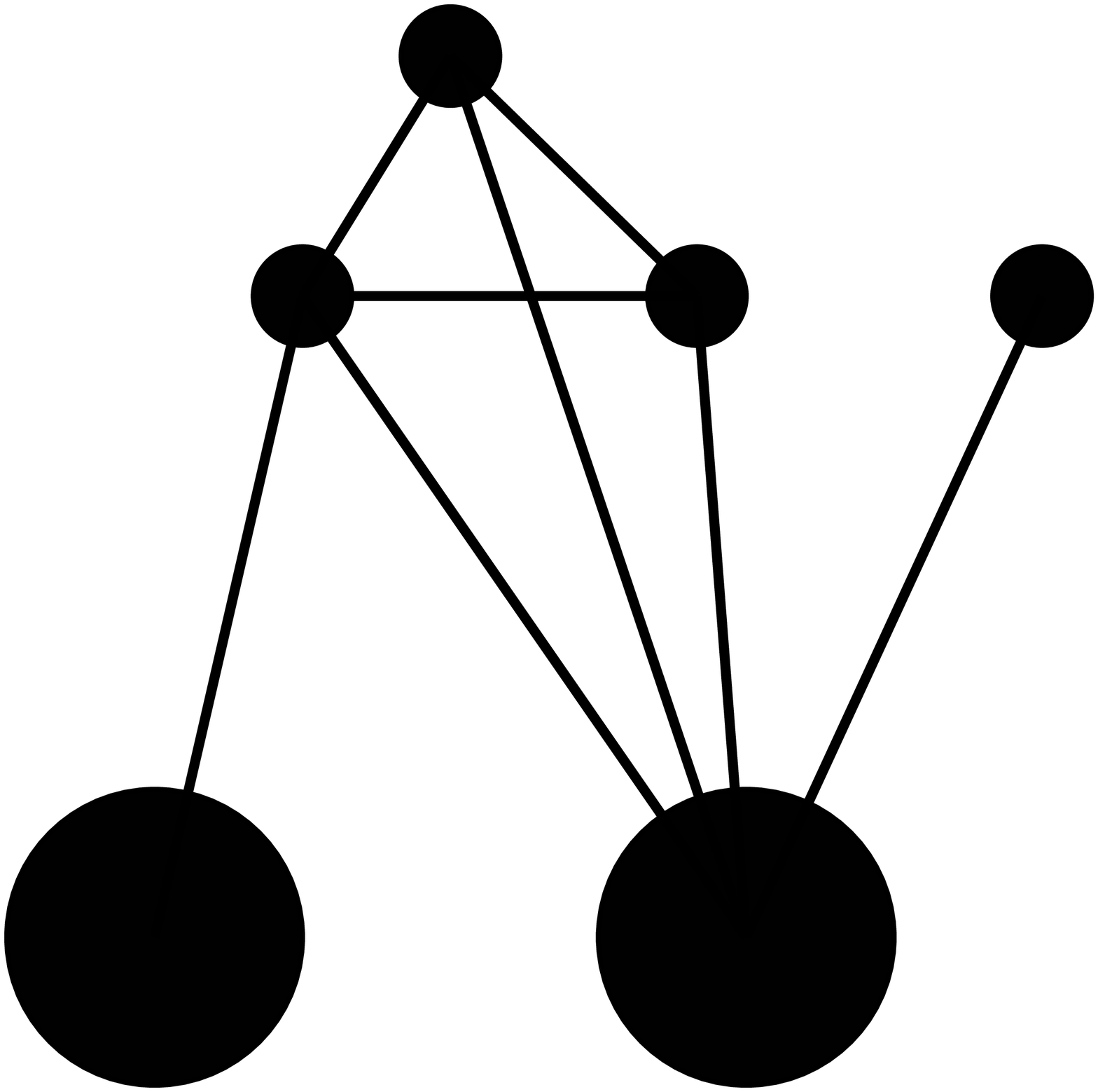} &
\includegraphics[scale=0.11]
{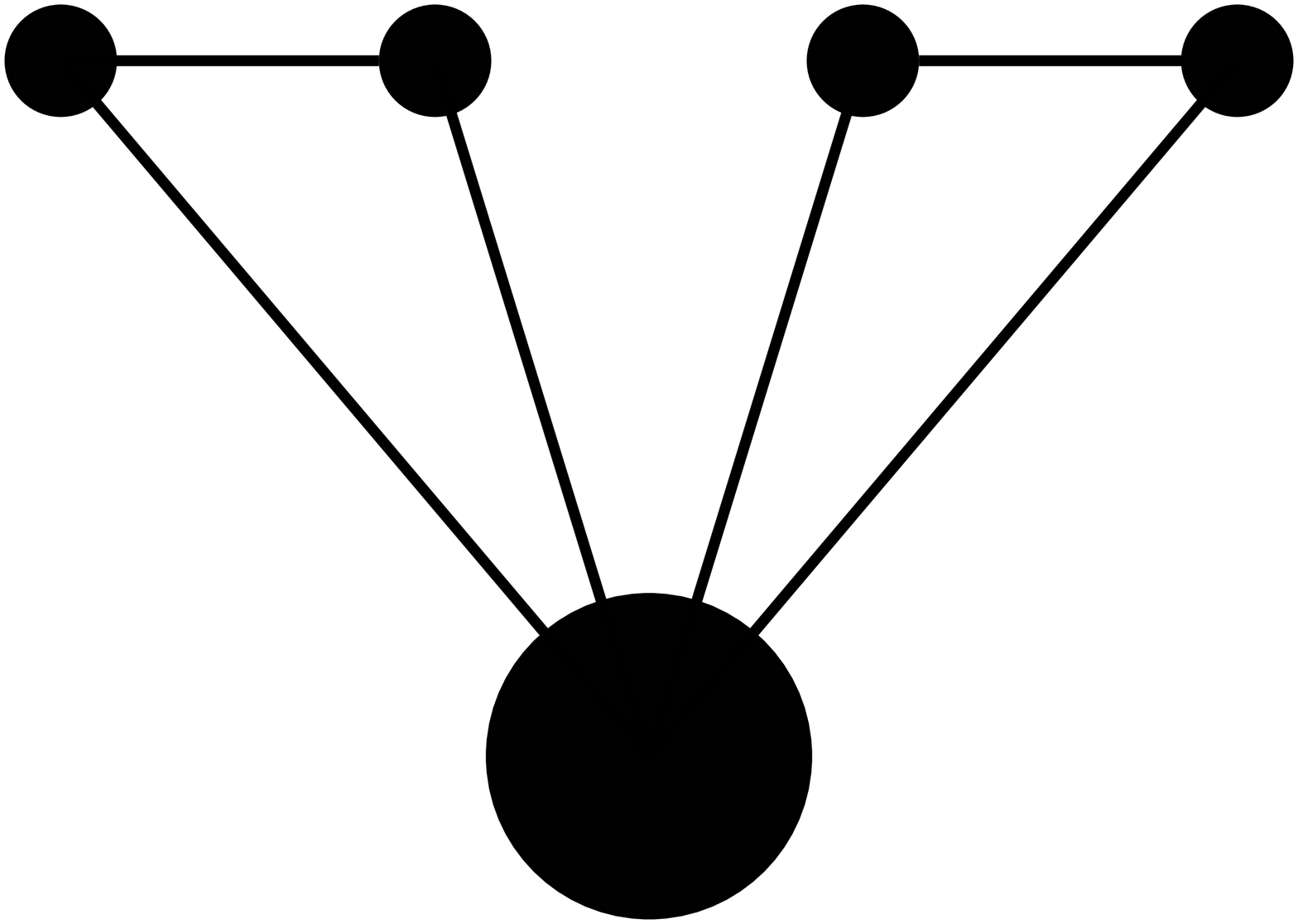} &
 \\

$F_6$ & $F_7$ & $F_8$ & $F_9$ &
\\

\end{tabular}
\label{FG2}
\end{figure}
\end{center}

Let $G=F\uplus K$ be a connected Hoffman graph
	such that $V_f(F)\subset V_f(K)$
   and
\begin{equation}\label{eq:H}
\ K=\biguplus_{i=0}^nH^i,
\ [H^j]\in\h
\cup\{[H_1]\}
\mbox{ for }j=0,1,\ldots,n.
\end{equation}
When $F\cong F_1,F_3,F_4,F_6,F_7$ or $F_9$,
Table~\ref{table:1} gives a list of slim subgraphs $G'$ guaranteed to exist
in $G$, under some additional assumptions. The assumptions are given in
terms of $c(K)$ and $|V_s(K)|$, where $c(K)$
denotes the number of connected components of $K$.
For example,
   if $F\cong F_1$, $c(K)=2$, and $|V_s(K)|=4$,
       then $G$ has a slim subgraph $G'$ isomorphic to
       $G_{5,1}$, $G_{5,2}$, $G_{6,3}$, 
       or $G_{6,21}$,
while if $F\cong F_3$ and $c(K)=2$, then
Table~\ref{table:1} gives no conclusion.
The results in Table~\ref{table:1} were obtained by computer.

\begin{table}[htbp]
\begin{center}
\caption{}
\label{table:1}
\begin{tabular}{c|c|c|c||l@{\hspace{2mm}}l|l@{\hspace{2mm}}l
	@{\hspace{2mm}}l@{\hspace{2mm}}l@{\hspace{2mm}}l
	@{\hspace{2mm}}
	l@{\hspace{2mm}}|ll}
	\hhline{|====:t:==========|}
	&$F$ & $c(K)$ & $|V_s(K)|$ & \multicolumn{10}{c}{$G'$} \\
	\hhline{|====::==========|}
	(a)&\multirow{2}{*}{$F_1$} & $1$ & $5$ & $G_{5,1}$& $G_{5,2}$
	& $G_{6,3}$
	& $G_{6,6}$ & $G_{6,12}$& $G_{6,14}$& $G_{6,21}$& & $G_{7,5}$ &\\
	\hhline{-|~--||----------|}
	(b)& & $2$ & $4$ & $G_{5,1}$& $G_{5,2}$& $G_{6,3}$
	& $G_{6,21}$ & & & & & &  \\
	\hhline{-|---|----------|}
	\multirow{2}{*}{(c)}&\multirow{2}{*}{$F_3$} & \multirow{5}{*}{$1$}
	& \multirow{2}{*}{$5$} & $G_{5,1}$& & $G_{6,5}$
	& $G_{6,7}$ & $G_{6,9}$& $G_{6,11}$& $G_{6,12}$& $G_{6,13}$
	& $G_{7,6}$& \\
	& & & & & & $G_{6,19}$
	&  $G_{6,17}$&$G_{6,23}$ & $G_{6,24}$& $G_{6,25}$& $G_{6,27}$& 
	&   \\
	\hhline{-|-|~|-||----------|}
	(d)&$F_4$ &
	  & \multirow{2}{*}{$4$} & $G_{5,1}$& & $G_{6,5}$
	& $G_{6,8}$ & $G_{6,15}$& $G_{6,18}$& & & &  \\
	\hhline{-|-|~|~||----------|}
	(e)&$F_6$ &  &  & $G_{5,2}$& & $G_{6,14}$
	& $G_{6,19}$ & $G_{6,22}$& $G_{6,26}$& $G_{6,28}$& & $G_{7,3}$&  \\
	\hhline{-|-|~|-||----------|}
	(f)&$F_7$ &  & $2$ & & & $G_{6,1}$
	& $G_{6,6}$ & $G_{6,16}$& & & & &  \\
	\hhline{-|-|~|-||----------|}
	(g)&$F_9$ &  & $4$ & & & $G_{6,2}$
	& $G_{6,3}$ & & & & &$G_{7,1}$ &$G_{7,2}$  \\
	\hhline{|====:b:==========|}
\end{tabular}
\end{center}
\end{table}



\begin{lm}\label{lm:new2_F}
Let $G=F\uplus H$ be a Hoffman graph
   satisfying
\begin{eqnarray}
	H=\biguplus_{i=0}^nH^i,\label{eq:new2_path}\\
	V_f(F)\subset V_f(H),\label{eq:new2_path2}\\
	\ H^j\cong H_2
	\mbox{ for }j=0,1,\ldots,n,\label{eq:new2_path3}\\
	\ H\text{ is connected}.\label{eq:new2_path4}
\end{eqnarray}
Suppose $F\cong F_i$ for some $i\in \{2,3,5,8\}$,
	and let $F'$ be a subgraph of $F$
   such that $F'\cong F_3$.
Let $V_f(F')=\{f_0,f_1\}$.
If there is no edge between $N^s_{H}(f_0)$ and $N^s_{H}(f_1)$,
   then $G$ has a slim subgraph isomorphic to $G_{5,1}$, $G_{6,17}$ or $G_{6,27}$.
\end{lm}
\begin{proof}
%
First we note $N^s_H(f_0)\cap N^s_H(f_1)=\emptyset$
	by Definition~\ref{df:1}(iv).
In particular, we have $n>0$.
From Lemma~\ref{lem:4.4},
	there exists a path in $\subg{V_s(H)}{H}$
   connecting a vertex in $N^s_H(f_0)$ and a vertex in $N^s_H(f_1)$.
Let $P$ be such a path with shortest length.
The length of $P$ is at least $2$
   by the assumption.
Since $G$ contains $F'\uplus H$ as a subgraph by Lemma~\ref{lm:subgraph},
it suffices to show that $F'\uplus H$ contains a desired slim subgraph.
If $P$ has length $2$ or $3$, 
   then $F'\uplus H$ has a subgraph isomorphic to $G_{5,1}$ or $G_{6,17}$,
   respectively.
If the length of $P$ is at least $4$,
   then $F'\uplus H$ has a subgraph isomorphic to $G_{6,27}$.
\end{proof}

\begin{lm}\label{lm:new2_F_4}
Let $G=F\uplus H$ be a Hoffman graph
	satisfying (\ref{eq:new2_path})--(\ref{eq:new2_path4}).
Suppose $F\cong F_4$, $V_f(F)=\{f_0,f_1,f_2\}$ with $|N_F^s(f_0)|=2$.
If $\subg{N^s_{H}(f_0)\cup N^s_{H}(f_1)\cup N^s_{H}(f_2)}{H}$
	is not connected,
	then $G$ has a slim subgraph isomorphic to $G_{5,1}$,
	$G_{6,17}$, $G_{6,23}$ or $G_{6,27}$.
\end{lm}
\begin{proof}
By (\ref{eq:new2_path2}),
	$|V_f(H)| \ge |V_f(F)| = 3$,
	and therefore $n>0$.
From Lemma~\ref{lem:4.4},
       there exists a path in $\subg{V_s(H)}{H}$
  connecting a vertex in $N^s_H(f_0)$
  and a vertex in $N^s_H(f_1)\cup N^s_H(f_2)$
  such that
  the two vertices are not adjacent in $H$,
  by the assumption.
Let $P=u\sim v\sim\cdots\sim w$ be such a path with shortest length,
where
$u\in N^s_H(f_1)\cup N^s_H(f_2)$ and $w\in N^s_H(f_0)$.
Then $v\notin N^s_H(f_1)\cup N^s_H(f_2)$, and
we may assume $u\in N^s_H(f_1)$ without loss of generality.
Then $V(P)\cap N^s_H(f_1)=\{u\}$.
If $u\sim f_2$,
	then $N^f_H(u)=\{f_1,f_2\}$,
	which implies $N^f_H(u)\cap N^f_H(v)=\emptyset$,
	contradicting $u\sim v$.
Thus $u\notin N^s_H(f_2)$.

Put $S=V(P)\cap N^s_H(f_2)$.
Suppose $S=\emptyset$.
By Lemma~\ref{lm:subgraph},
	$F\uplus\subgg{V(P)}{H}\subset G$,
	while $f_2$ has no slim neighbour
	in $\subgg{V(P)}{H}$.
This implies $(F-f_2)\uplus\subgg{V(P)}{H}\subset G$.
Since $F-f_2\cong F_3$,
	the lemma follows from Lemma~\ref{lm:new2_F}.
Suppose $S\neq\emptyset$.
Since $P$ is the shortest path,
	$w$ is adjacent to exactly one vertex $s_1$ in $S$,
	and $|S|=2$.
Put $S\setminus \{s_1\}=\{s_2\}$,
	and let $w'$ be the neighbour of $s_2$
	different from $s_1$ in $P$.
Then $\subg{V_s(F)\cup S\cup\{w,w'\}}{G}\cong G_{6,23}$,
and hence $G$ contains a subgraph isomorphic to $G_{6,23}$.
\end{proof}

\begin{lm}\label{lm:Fpre}
Let $G=F\uplus H$ be a Hoffman graph
   satisfying (\ref{eq:new2_path}),
   (\ref{eq:new2_path2}) and
   the following conditions:
\begin{eqnarray}
	F\text{ is connected},\label{eq:Fpre1}\\
	\ [H^j]\in\h
	\mbox{ for }j=0,1,\ldots,n.\label{eq:Fpre2}
\end{eqnarray}
Let $V$ is a subset of $V_s(H)$,
	and let $K=\subgg{V}{H}$.
If $V_f(F)\subset V_f(K)$,
	and every vertex of $V$
	can be joined by
	a path in $K$ to a fat vertex of $F$,
	then $G$ contains a connected subgraph $F\uplus K$
	satisfying (\ref{eq:H}).
\end{lm}
\begin{proof}
From Lemma 12 of \cite{paperI},
	$\subgg{V_s(F)\cup V}{G}=F\uplus K$.
Since $F$ is connected
	and every vertex of $V$
	can be joined by
	a path in $K$ to a fat vertex of $F$,
	$F\uplus K$ is connected.
From Lemma~\ref{lm:new2_1},
	$K$ satisfies (\ref{eq:H}).
\end{proof}

\setcounter{cl}{0}
\begin{lm}\label{lm:F}
Let $G=F\uplus H$ be a Hoffman graph
   satisfying
   (\ref{eq:new2_path}),
   (\ref{eq:new2_path2}),
   (\ref{eq:Fpre2}),
   and $F\cong F_i$ for some $i\in \{1,2,\ldots,9\}$.
Let
\[
m(F)=\begin{cases}
	2 & \text{if }F\cong F_7,\\
	4 & \text{if }F\cong F_4,F_6\text{ or }F_9,\\
	5 & \text{otherwise.}
	\end{cases}
\]
If $H$ is connected and $|V_s(H)|\ge m(F)$,
   then $G$ has a slim subgraph isomorphic to one of the graphs
	in Figure~\ref{MFS}.
\end{lm}
\begin{proof}

Let $I=\{i\mid H^i\cong H_2,\ 0\le i\le n\}$.
First we suppose $I=\emptyset$.
Then,
   since $H^i\cong H_3$ or $H_5$,
   $|V_f(H^i)|=1$  for all $i\in \{0,1,\ldots,n\}$.
This implies $|V_f(H)|=1$ since $H$ is connected.
Hence $F\cong F_6,F_7$ or $F_9$
   by (\ref{eq:new2_path2}).
Suppose $F\cong F_7$.
Since $H_3$ is a subgraph of $H_5$,
   there exists a subgraph $K$ of $H$ such that $K\cong H_3$.
Then $G$ contains $F\uplus K$
	as a subgraph from Lemma~\ref{lm:subgraph}.
Since $F\uplus K$ satisfies the assumptions of Table~\ref{table:1},
	the conclusion holds.
Suppose $F\cong F_6$ or $F_9$.
Since $|V_s(H)|\ge 4$ and $H_3$ is a subgraph of $H_5$,
   there exists a subgraph $K$ of $H$
   isomorphic to the sum $H_3\uplus H_3$ sharing a fat vertex.
Then $G$ contains $F\uplus K$
	as a subgraph from Lemma~\ref{lm:subgraph}.
Since $F\uplus K$ satisfies the assumptions of Table~\ref{table:1},
	the conclusion holds.
In the remaining part of this proof,
   we suppose $I\neq\emptyset$.
For a subset $J$ of $\{0,1,\ldots,n\}$,
	we write $H(J)=\biguplus_{i\in J}H^i$.

\begin{cl}\label{cl:002}
The graph $\subg{V_s(H)}{H}$ is connected.
\end{cl}
Since $|V_s(H)|\ge m(F)\ge 2$ and $I\neq \emptyset$,
   $n>0$.
Hence,
   from the last part of Lemma~\ref{lem:4.4},
   $\subg{V_s(H)}{H}$ is connected.

\begin{cl}\label{cl:003}
$V_f(F)\subset V_f(H(I))$.
\end{cl}
From Lemma~\ref{lm:Fpre2}(ii),
	$V_f(H(I))=V_f(H)$.
By (\ref{eq:new2_path2}),
	$V_f(F)\subset V_f(H(I))$.

\begin{cl}\label{cl:004}
Suppose $F\cong F_1,F_3,F_4,F_6,F_7$ or $F_9$,
	and that there exists $I'\subset I$ such that $|I'|\le m(F)$,
	$V_f(F)\subset V_f(H(I'))$
	and $H(I')$ is connected.
Then the lemma holds.
\end{cl}

If $|I'|=1$,
   then obviously $\subg{V_s(H(I'))}{H}$ is connected.
If $|I'|>1$,
   then,
   from the last part of Lemma~\ref{lem:4.4},
   $\subg{V_s(H(I'))}{H}$ is connected.
The graph $\subg{V_s(H)}{H}$ is also connected from Claim~\ref{cl:002}.
Since $|V_s(H(I'))|=|I'|\le m(F)\le |V_s(H)|$,
	there exists a subset $V$
	such that
	$V_s(H(I'))\subset V\subset V_s(H)$,
	$|V|=m(F)$ and
	$\subg{V}{H}$ is connected.
Put $K=\subgg{V}{H}$.
Then $K$ is connected and $V_f(F)\subset V_f(K)$.
Hence $G$ contains a connected subgraph $F\uplus K$
	satisfying (\ref{eq:H})
	by Lemma~\ref{lm:Fpre}.
Therefore
	the assumptions of Table~\ref{table:1} are satisfied.
Hence the lemma holds.

\begin{cl}\label{cl:F6_7}
If $F\cong F_6,F_7$ or $F_9$,
   then the lemma holds.
\end{cl}
From Claim~\ref{cl:003},
   there exists $i\in I$ such that the unique fat vertex of $F$ is in $V_f(H^i)$.
Then $I'=\{i\}$ satisfies the hypotheses of Claim~\ref{cl:004},
   and hence the lemma holds.

\begin{cl}\label{cl:F1}
If $F\cong F_1$,
   then the lemma holds.
\end{cl}
Let $V_f(F)=\{f_0,f_1\}$.
From Claim~\ref{cl:003},
   there exist $i_0,i_1\in I$
   such that
   $f_k\in V_f(H^{i_k})$ for each $k=0,1$.
From Definition~\ref{df:1}(ii), $i_0\neq i_1$.
For each $k=0,1$,
	let $s_k$ be the unique slim vertex of $H^{i_k}$.
Since $H$ is connected and $5=m(F)\le |V_s(H)|$,
	there exist disjoint subsets $V_0$, $V_1$ of $V_s(H)$
	such that $|V_0\cup V_1|=5$,
	$\subgg{V_k}{H}$ is connected and
	$s_k\in V_k$ for each $k=0,1$.
Let $V=V_0\cup V_1$.
Then every vertex of $V$ can be joined by
	a path in $\subgg{V}{H}$ to $f_0$ or $f_1$.

Suppose $c(\subgg{V}{H})=1$,
	i.e., $\subgg{V}{H}$ is connected.
Let $I'=\{i\in I\mid V_s(H^i)\subset V\}$.
Then $|I'|\le |V|=m(F)$
	and $i_0,i_1\in I'$.
Since $I'\neq\emptyset$,
   $H(I')$ is connected
	from Lemma~\ref{lm:Fpre2}(i).
Since $i_0,i_1\in I'$,
   $V_f(F)\subset V_f(H(I'))$.
Hence $I'$ satisfies the hypotheses of Claim~\ref{cl:004},
   and the lemma holds.

Next suppose $c(\subgg{V}{H})>1$.
Since $\subgg{V_0}{H}$ and $\subgg{V_1}{H}$ are connected,
	$c(\subgg{V}{H})=2$.
Since $|V_0|+|V_1|=5$,
	we may assume $|V_0|\ge 3$	
	without loss of generality.
Let $s$ be a slim vertex of $\subgg{V_0}{H}$
	which has the largest distance from $s_0$.
Then $\subgg{V_0\setminus\{s\}}{H}$ is connected.
Put $K=\subgg{V\setminus\{s\}}{H}$.
Then $c(K)=2$.
Moreover $V_f(F)\subset V_f(K)$,
	and every vertex of $V\setminus \{s\}$ can be
	joined by a path in $K$ 
	to $f_0$ or $f_1$.
Hence $G$ contains a connected subgraph $F\uplus K$
	satisfying (\ref{eq:H})
	by Lemma~\ref{lm:Fpre}.
Since $|V_s(K)|=|V\setminus \{s\}|=4$,
	the assumptions of Table~\ref{table:1} are satisfied.
Hence the lemma holds.

Now we consider the remaining cases.
Let $F'$ be a subgraph of $F$ such that
\[
	\begin{cases}
	F'\cong F_3  &  \text{if }F\cong F_2,F_3,F_5\text{ or }F_8,\\
	F'= F  &  \text{if }F\cong F_4.
	\end{cases}
\]
Obviously $V_f(F')=V_f(F)$.
Hence $F'=\subgg{V_s(F')}{F}$. 
Thus $\subg{V(F')\cup V(H)}{G}=F'\uplus H$ from Lemma~\ref{lm:subgraph},
	i.e., $F'\uplus H\subset G$.
Let $f_0$ be
	the unique
	fat vertex of $F'$
   satisfying $|N^s_{F'}(f_0)|=2$,
   and let $f_1$ be a fat vertex of $F'$
   different from $f_0$.
Then $f_0,f_1\in V_f(H(I))$
from Claim~\ref{cl:003}.

\begin{cl}\label{cl:F2_3_5_8}
If $F\cong F_2,F_3,F_5\text{ or }F_8$,
   then the lemma holds.
\end{cl}

Then $F'\cong F_3$.
From Lemma~\ref{lm:Fpre2}(i),
   $H(I)$ is connected.
If there is no edge between $N^s_{H(I)}(f_0)$ and $N^s_{H(I)}(f_1)$,
   then the result follows from Lemma~\ref{lm:new2_F}.
Suppose that there exist $s_0\in N_{H(I)}^s(f_0)$ and $s_1\in N_{H(I)}^s(f_1)$
   such that $s_0\sim s_1$.
For each $k=0,1$,
   there exists $i_k\in I$
   such that $V_s(H^{i_k})=\{s_k\}$.
Put $I'=\{i_0,i_1\}$.
By Lemma~\ref{lm:Fpre2}(i),
	$H(I')$ is connected.
Then $I'$ satisfies the hypotheses of Claim~\ref{cl:004},
   and the lemma holds.

\begin{cl}\label{cl:F4}
If $F\cong F_4$,
   then the lemma holds.
\end{cl}
Let $f_2$ be a fat vertex of $F$
   different from $f_0,f_1$.
From Lemma~\ref{lm:Fpre2}(i),
   $H(I)$ is connected,
   and from Claim~\ref{cl:003},
   $V_f(F)\subset V_f(H(I))$.
Put $N_i=N_{H(I)}^s(f_i)$ for $i=0,1,2$.
If $\subg{N_0\cup N_1\cup N_2}{H(I)}$
	is not connected,
	then the result follows from Lemma~\ref{lm:new2_F_4}.
Suppose that $\subg{N_0\cup N_1\cup N_2}{H(I)}$
	is connected.
Then,
	for each $i=1,2$,
	there exists an edge $s_is_0^{(i)}$ between $N_i$ and $N_0$
	such that $s_i\in N_i$ and $s_0^{(i)}\in N_0$.
Put $I'=\{i\in I\mid
   V_s(H^i)\subset \{s_0^{(1)},s_0^{(2)},s_1,s_2\}\}$.
Since $s_0^{(1)},s_0^{(2)},s_1,s_2\in V_s(H(I))$,
	$\subgg{\{s_0^{(1)},s_0^{(2)},s_1,s_2\}}{H}=H(I')$.
Since $f_0$ is a common fat neighbour of $s_0^{(1)}$ and $s_0^{(2)}$,
	$s_0^{(1)}$ and $s_0^{(2)}$ are adjacent,
	or equivalently in $H(I')$.
Thus $H(I')$ is connected.
Then $I'$ satisfies the hypotheses of Claim~\ref{cl:004},
	and hence the lemma holds.
\end{proof}

The next three lemmas are verified by computer.
\begin{lm}\label{lm:100}
Let $F$ be a fat connected Hoffman graph
	satisfying the following conditions:
\begin{enumerate}[(i)]
\item $|V_s(F)|=2$,
\item the two slim vertices of $F$ are not adjacent,
\item $|V_f(F)|\le 4$,
\item every slim vertex has at most $2$ fat neighbours,
\item $F$ is a non $\h$-line graph.
\end{enumerate}
Then $F$ is isomorphic to $F_1$, $F_3$ or $F_4$.
\end{lm}
\begin{lm}\label{lm:101}
Let $F$ be a fat connected Hoffman graph
	satisfying the following conditions:
\begin{enumerate}[(i)]
\item $3\le |V_s(F)|\le 4$,
\item $|V_f(F)|\le 2$,
\item some slim vertex $s$ of $F$ has $2$ fat neighbours,
\item some slim vertex $s'$ of $F$ is not adjacent to $s$,
   and the others are adjacent to $s$,
\item $\subgg{V_s(F)\setminus\{s\}}{F}\cong H_3$ or $H_5$,
\item $F$ is a non $\h$-line graph.
\end{enumerate}
Then $F$ is isomorphic to $F_2$, $F_5$ or $F_8$.
\end{lm}
\begin{lm}\label{lm:102}
Let $F$ be a fat connected Hoffman graph
	satisfying the following conditions:
\begin{enumerate}[(i)]
\item $3\le |V_s(F)|\le 6$,
\item $|V_f(F)|=1$,
\item every slim vertex of $F$ has $1$ fat neighbour,
\item there exist different subsets $V_1$ and $V_2$ of $V_s(F)$
   such that $V_1\cup V_2=V_s(F)$,
   $\subgg{V_1}{F}$ and $\subgg{V_2}{F}$
   are isomorphic to $H_3$ or $H_5$,
   the vertex of $V_s(F)\setminus V_2$ and 
   the vertex of $V_s(F)\setminus V_1$ are adjacent to each other
      except some pair
      $\{s_1,s_2\}$
      $(s_1\in V_s(F)\setminus V_2$,
      $s_2\in V_s(F)\setminus V_1)$,
\item $F$ is a non $\h$-line graph.
\end{enumerate}
Then $F$ contains a subgraph isomorphic to $F_6,F_7$ or $F_9$.
\end{lm}

\setcounter{cl}{0}
We shall now prove our main result.
\begin{proof}[Proof of Theorem~\ref{mainthm}.]\label{proof:mainthm}
From Proposition \ref{prop1},
	it is enough to prove $|V(\Gamma)|<10$,
	so suppose $|V(\Gamma)|\ge 10$.
Since a complete graph and a cycle are $\h$-line graphs,
$\Gamma$ is neither a complete graph nor a cycle.
Hence,
	from Lemma~\ref{connected},
	there exists a non-adjacent pair $\{x,y\}$ in $V(\Gamma)$
		such that $\Gamma-\{x,y\}$ is connected.
Then $\Gamma-x$ and $\Gamma-y$ are connected as well.
The graphs $\Gamma-x$, $\Gamma-y$ and $\Gamma-\{x,y\}$
		are $\h$-line graphs by the minimality of $\Gamma$
		and $|V(\Gamma-\{x,y\})|\ge 8$.

Let $X=\biguplus_{i=0}^{m_1}X^i$ (resp. $Y=\biguplus_{i=0}^{m_2}Y^i$)
	be a strict $\h$-cover graph of $\Gamma-y$ (resp. $\Gamma-x$).
Without loss of generality,
	we may suppose $x\in V_s(X^0)$ and $y\in V_s(Y^0)$.
From Lemma~\ref{lm:newLemma001},
   there exists a strict $\h$-cover graph
	$\X=\X^0\uplus(\biguplus_{i=1}^{m_1}X^i)$ of $X-x$.
Similarly,
	there exists a strict $\h$-cover graph
	$\Y=\Y^0\uplus(\biguplus_{i=1}^{m_2}Y^i)$ of $Y-y$.
Obviously $\X$ and $\Y$ are strict $\h$-cover graph of $\Gamma-\{x,y\}$.
From Theorem 31 of \cite{paperI},
	there exists an isomorphism
	$\varphi:~\Y\to\X$
	such that $\varphi|_{(\Gamma-\{x,y\})}$ is 
	the identity automorphism of $\Gamma-\{x,y\}$.

From Lemma~\ref{lm:newLemma001},
	we can put 
   $\X^0=\X^0_1\uplus \X^0_2$
   ($[\X^0_1], [\X^0_2]
   \in \{[\phi],[H_2],[H_3]\}$)
   and
   $\Y^0=\Y^0_1\uplus \Y^0_2$
   ($[\Y^0_1], [\Y^0_2]
   \in \{[\phi],[H_2],[H_3]\}$),
   and put $\mathcal{X}=\{\phi,\X^0_1,\X^0_2\}$
   and $\mathcal{Y}=\{\phi,\Y^0_1,\Y^0_2\}$.   
Then $\tilde{X}=(\biguplus_{K\in \mathcal{X}}K)\uplus
   (\biguplus_{i=1}^{m_1}X^i)
   =(\biguplus_{L\in \mathcal{Y}}\varphi(L))\uplus
   (\biguplus_{j=1}^{m_2}\varphi(Y^j))$.
From Lemma~\ref{newlemma},
$\{\varphi(L)\mid L\in\mathcal{Y}\}
   \cup\{\varphi(Y^i)\mid 1\le i\le m_2\}
   =\mathcal{X}\cup\{X^i\mid 1\le i\le m_1\}$.
Put $\mathcal{Z}
      =\mathcal{X}\cup\{\varphi(L)\mid L\in\mathcal{Y}\}$.
Then 
\begin{equation}\label{eq:005}
\tilde{X}=(\biguplus_{Z\in\mathcal{Z}}Z)\uplus H,
\end{equation}
   where 
\begin{equation}\label{eq:01}
H=\biguplus_{i\in I}X^i=\biguplus_{j\in J}\varphi(Y^j)
\end{equation}
   for some $I\subset\{1,2,\ldots,m_1\}$
      and $J\subset\{1,2,\ldots,m_2\}$.
Obviously
\begin{equation}\label{eq:007}
X=X^0\uplus (\biguplus_{Z\in\mathcal{Z}\setminus\mathcal{X}}Z)\uplus H,\quad
Y=Y^0\uplus (\biguplus_{Z\in\mathcal{Z}\setminus\{\varphi (L)\mid L\in\mathcal{Y}\}}\varphi^{-1}(Z))\uplus \varphi^{-1}(H),
\end{equation}

\begin{cl}\label{cl:017}
The graph $H$ is connected.
\end{cl}
Since $\Gamma-\{x,y\}$ is connected,
   so is $\tilde{X}$.
The Hoffman graph $H'=\biguplus_{Z\in\mathcal{Z}}Z$
	has the unique fat vertex $\alpha$
	satisfying $V_f(H')\cap V_f(H)=\{\alpha\}$
	and $N_{H'}^s(\alpha)=V_s(H')$.
Using Lemma~\ref{lem:4.2} on the decomposition
   (\ref{eq:005}),
	We conclude that $H$ is connected.

We define the edge set
\[
E_0=\left(\bigcup_{z\in V_s(X^0)}
	\{zf \mid
		f\in V_f(H)\cap N_{X^0}^f(z)
	\}
	\right)
	\cup
	\left(\bigcup_{z\in V_s(Y^0)}
	\{z\varphi(g) \mid
		g\in V_f(\varphi^{-1}(H))\cap N_{Y^0}^f(z)
	\}
	\right)
\]

and the Hoffman graph
\[
G=(V(\Gamma)\cup V_f(H),~E(\Gamma)\cup E(H)\cup E_0).
\]
Let
\[
	F=\subg{\subg{V_s(X^0)\cup V_s(Y^0)}{}}{G}.
\]
Obviously the following holds:
\begin{equation}\label{eq:006}
s\in V_s(F),\ f\in V_f(G),\ sf\in E(G)
\Longrightarrow sf\in E_0,
\end{equation}
and
\begin{equation}\label{eq:02}
\begin{array}{cccccc}
(\text{a}) & V_f(F)\subset V_f(H), & (\text{b}) & E_0\subset E(F),&
 (\text{c}) &
\Gamma\subset G\mbox{ and }V_s(\Gamma)=V_s(G). 
\end{array}
\end{equation}
Also,
	from (\ref{eq:005}),
\begin{equation}\label{eq:03}
\begin{array}{cccc}
(\text{a}) & V_s(\Gamma)=V_s(F)\cup V_s(H), & (\text{b}) &
V_s(F)\cap V_s(H)=\emptyset.
\end{array}
\end{equation}
From (\ref{eq:03}),
\begin{equation}\label{eq:04}
	|V_s(H)|\ge10-|V_s(F)|.
\end{equation}
By the definition of $G$,
\begin{equation}\label{eq:008}
	V_f(G)=V_f(H).
\end{equation}

\begin{cl}\label{cl:007}
$G=F\uplus H$.
\end{cl}
Let us check the conditions (i)--(iv) of  Definition \ref{df:1}.

From (\ref{eq:02})-(c) and (\ref{eq:03})-(a),
	$V_s(G)=V_s(F)\cup V_s(H)$.
Moreover,
	$V_f(G)=V_f(H)=V_f(F)\cup V_f(H)$
	by (\ref{eq:02})-(a) and (\ref{eq:008}).
Hence the condition (i) is satisfied.
Also,
	by  (\ref{eq:03})-(b),
	the condition (ii) is satisfied.
By the definitions of $F$ and $G$,
   the condition (iii) is satisfied.

Let $s_1\in V_s(F)$,
	and let $s_2\in V_s(H)$.
Then $s_1\in V_s(X^0)$ or $s_1\in V_s(Y^0)$,
   $s_2\in V_s(H)\subset V_s(\biguplus_{i=1}^{m_1}X^i)$.
By (\ref{eq:008}),
	$N^f_{G}(s_2)=N^f_{H}(s_2)$.
First suppose $s_1\in V_s(X^0)$.
Since $N_H^f(s_2)\subset V_f(H)$,
	$N_G^f(s_1)\cap N_G^f(s_2)
   =(N_{X^0}^f(s_1)\cap V_f(H))\cap N_H^f(s_2)
   =N_{X^0}^f(s_1)\cap N_H^f(s_2).$
Since  $N^f_{X^0}(s_1)=N^f_{X}(s_1)$
	and $N^f_{H}(s_2)=N^f_{X}(s_2)$,
	$N_G^f(s_1)\cap N_G^f(s_2)
   =N_X^f(s_1)\cap N_X^f(s_2)$.
Thus,
   $s_1$ and $s_2$ have at most one common fat neighbour in $G$,
   and they have one
   if and only if
   they are adjacent in $X$,
   or equivalently in $G$.
Hence (iv) holds in this case.
Next suppose $s_1\in V_s(Y^0)$.
Since $s_2\in V_s(H)$,
   $s_2\in \varphi (V_s(Y^j))$ for some $j\in\{1,2,\ldots,m_2\}$.
Hence $N_G^f(s_2)=N_H^f(s_2)=N_{\varphi (Y^j)}^f(s_2)
   =N_{\varphi (Y^j)}^f(\varphi (s_2))
   =\varphi(N_{Y^j}^f(s_2))
   =\varphi(N_{Y}^f(s_2))$.
Thus
\begin{align*}
N_G^f(s_1)\cap N_G^f(s_2)
   &=\varphi(V_f(\varphi^{-1}(H))\cap N_{Y^0}^f(s_1))
      \cap \varphi(N_Y^f(s_2))
   \\&=\varphi(N_{Y}^f(s_1)\cap N_Y^f(s_2)\cap V_f(\varphi^{-1}(H)))
   \\&=\varphi(N_{Y}^f(s_1)\cap N_Y^f(s_2))
\end{align*}
   since $\varphi^{-1}(H)\subset Y$.
A similar argument shows that (iv) holds
   in this case as well.

\begin{cl}\label{cl:012}
For any $s\in V_s(F)$,
\[
   |N_G^f(s)|\le
   \begin{cases}
   |V_f(X^0)| & \mbox{ if }s\in V_s(X^0),\\
   |V_f(Y^0)| & \mbox{otherwise.}
   \end{cases}
\]
\end{cl}
By (\ref{eq:006}),
   $sf\in E_0$ for each $f\in N_G^f(s)$.
Suppose $s\in V_s(X^0)$.
Then $sf\in E(X^0)$.
Hence $|N_G^f(s)|\le |N_{X^0}^f(s)|\le |V_f(X^0)|$.
Suppose $s\in V_s(Y^0)$.
Then $s\varphi^{-1}(f)\in E(Y^0)$.
Hence $|N_G^f(s)|\le |N_{Y^0}^f(s)| \le |V_f(Y^0)|$.

\begin{cl}\label{cl:011}
$|V_f(F)|\le |V_f(X^0)|+|V_f(Y^0)|$.
\end{cl}

From Claim~\ref{cl:007},
	$V_f(F)=V_f(\subgg{V_s(X^0)\cup V_s(Y^0)}{F\uplus H})$.
By (\ref{eq:02})-(a),
	$V_f(F)=V_f(\subgg{V_s(X^0)\cup V_s(Y^0)}{F\uplus H})
	\cap V_f(H)$,
	i.e.,
\[
	V_f(F)=\left(V_f(X^0)\cap V_f(H)\right)
	\cup
	\left(\varphi(
		V_f(Y^0)\cap \varphi^{-1}(V_f(H))
	)\right).
\]
Hence
\begin{align*}
|V_f(F)|& \le |V_f(X^0)\cap V_f(H)|
	+|
		V_f(Y^0)\cap \varphi^{-1}(V_f(H))|\\
	&\le |V_f(X^0)|+|V_f(Y^0)|.
\end{align*}

\begin{cl}\label{cl:008}
The Hoffman graph $F$ is a non $\h$-line graph.
\end{cl}
The Hoffman graph $H$ is a strict $\h$-cover graph of itself.
Suppose that $F$ is an $\h$-line graph.
Then there exists a strict $\h$-cover graph of $F$
   (cf. Example 22 of \cite{paperI}).
Hence
   $G$ has a strict $\h$-cover graph
   from Lemma 20 of \cite{paperI}.
Since $\Gamma\subset G$,
   $\Gamma$ is an $\h$-line graph,
   a contradiction.

\begin{cl}\label{cl:013}
If $X^0$ or $Y^0$ is isomorphic to $H_2$,
   the theorem holds.
\end{cl}
If $X^0$ or $Y^0$ is isomorphic to $H_2$,
   then $V_s(X^0)\cap V_s(Y^0)=\emptyset$,
   and each slim vertex of $F$ has at most $2$ fat neighbours
   by Claim~\ref{cl:012}.
First suppose that $X^0$ and $Y^0$ are isomorphic to $H_2$.
Then $|V_s(F)|=|\{x,y\}|=2$ and
   $|V_f(F)|\le 4$ by Claim~\ref{cl:011}.
Hence the hypotheses of Lemma~\ref{lm:100} hold
   by Claim~\ref{cl:008}.
Thus $F\cong F_1,F_3$ or $F_4$,
	and $|V_s(H)|\ge8$ by (\ref{eq:04}).
Next suppose otherwise.
Then 
   $3\le |V_s(F)|=|V_s(X^0)\cup V_s(Y^0)|\le 4$,
   and $|V_f(F)|\le 3$ by Claim~\ref{cl:011}.
If $|V_f(F)|= 3$,
   then $V_f(X^0)\cap V_f(Y^0)=\emptyset$,
   and therefore $F$ is an $\h$-line graph
   since $V_s(X^0)\cap V_s(Y^0)=\emptyset$,
   a contradiction to Claim~\ref{cl:008}.
Obviously the hypotheses (v) and (iv) of Lemma~\ref{lm:101} hold.
Hence the hypotheses of Lemma~\ref{lm:101} hold
   by Claim~\ref{cl:008}.
Thus $F\cong F_2$, $F_5$ or $F_8$,
	and $|V_s(H)|\ge6$ by (\ref{eq:04}).
Hence the theorem holds from Lemma~\ref{lm:F}
   if $X^0$ or $Y^0$ is isomorphic to $H_2$.

For the remainder of this proof,
   we assume that $X^0$ and $Y^0$ are isomorphic
      to $H_3$ or $H_5$.
Then $3\le |V_s(X^0)\cup V_s(Y^0)|(=|V_s(F)|)\le 6$.
Hence the condition (i) of Lemma~\ref{lm:102} holds.
Suppose $V_f(X^0)\cap \varphi(V_f(Y^0))=\emptyset$.
Then $V_f(\X^0)\cap V_f(\varphi(\Y^0))=\emptyset$.
Hence $V(\X^0)\cap  V(\varphi(\Y^0))=\emptyset$
	by (\ref{eq:005})
	since $\X=\varphi(\Y)$.
Thus $V_s(X^0)\cap V_s(Y^0)=\emptyset$,
   and therefore
   $F=\subgg{V_s(X^0)\cup V_s(Y^0)}{G}
      =\subgg{V_s(X^0)}{G}\uplus \subgg{V_s(Y^0)}{G}$.
Obviously $\subgg{V_s(X^0)}{G}$ and $\subgg{V_s(Y^0)}{G}$
   are isomorphic to $H_3$ or $H_5$.
Hence	$F$ is an $\h$-line graph.
But this contradicts Claim~\ref{cl:008}.
Thus $V_f(X^0)\cap \varphi(V_f(Y^0))\neq\emptyset$,
   i.e., $\varphi$ maps the unique fat vertex of $Y^0$
	to the unique fat vertex of $X^0$,
	and $|V_f(F)|=1$.
Hence the conditions (ii) and (iii) of Lemma~\ref{lm:102} hold.
Moreover the condition (v) of Lemma~\ref{lm:102} holds
   by Claim~\ref{cl:008}.

Put $V_1=V_s(X^0)$ and $V_2=V_s(Y^0)$,
	and put $s_1=x$ and $s_2=y$.
Then
\begin{itemize}
\item $(V_s(F)\setminus V_2)\setminus\{s_1\}=V_s(\tilde{X}^0)\setminus V_s(Y^0)
	\subset
	V_s(\biguplus_{Z\in
	\mathcal{Z}\setminus\{\varphi (L)\mid L\in\mathcal{Y}\}}
	\varphi^{-1} (Z))$ by (\ref{eq:007}),
\item $V_s(F)\setminus V_1=V_s(Y^0)\setminus V_s(X^0)
	\subset
	V_s(Y^0)$,
\item $(V_s(F)\setminus V_1)\setminus\{s_2\}=V_s(\tilde{Y}^0)\setminus V_s(X^0)
	\subset
	V_s(\biguplus_{Z\in
	\mathcal{Z}\setminus\mathcal{X}}
	Z)$ by (\ref{eq:007}),
\item $V_s(F)\setminus V_2=V_s(X^0)\setminus V_s(Y^0)
	\subset
	V_s(X^0).$
\end{itemize}
Hence the vertex of $V_s(F)\setminus V_2$ and 
	the vertex of $V_s(F)\setminus V_1$ are adjacent to each other
      except the pair
      $\{s_1,s_2\}$
      $(s_1\in V_s(F)\setminus V_2$,
      $s_2\in V_s(F)\setminus V_1)$.
Thus the conditions (iv) of Lemma~\ref{lm:102} holds.
Therefore
	$F$ has a subgraph isomorphic to $F_6$, $F_7$ or $F_9$.
Let $F'$ be a subgraph isomorphic to $F_6$, $F_7$ or $F_9$,
	of $F$.
Then $F'\uplus H\subset G$ from Lemma~\ref{lm:subgraph}.
Now $|V_f(F')|=1$,
	$|V_s(H)|\ge4$ by
	(\ref{eq:04}).
Moreover $V_f(F')=V_f(F)\subset V_f(H)$ from (\ref{eq:02})-(a).
Hence the hypothesis of Lemma~\ref{lm:F} is satisfied.
Thus $F'\uplus H$ has a slim subgraph isomorphic to
	one of the graphs in Figure~\ref{MFS},
	and so does $G$.
\end{proof}


\paragraph{Acknowledgements}
The author would like to thank his advisor,
	Professor Akihiro Munemasa,
	for his valuable suggestions and fruitful discussions on the topic.
The author is also grateful to
	Professors Drago\'{s} Cvetkovi\'{c} and Arnold Neumaier
	for their valuable comments on the presentation of the results contained
	in this paper.



\begin{thebibliography}{9}
\bibitem{root} P.~J.~Cameron, J.~M.~Goethals, J.~J.~Seidel, and E.~E.~Shult,
		Line graphs, root systems and elliptic geometry, 
		J. Algebra 43:305-327(1976).
\bibitem{spectra}D.~Cvetkovi\'{c},
		M.~Doob,
		H.~Sachs,
		Spectra of Graphs -- Theory and Applications,
		III revised and enlarged edition, Johan Ambrosius Bart.
		Verlag, Heidelberg - Leipzig, 1995.
\bibitem{glg} D.~Cvetkovi\'{c},
		M.~Doob and S.~Simi\'{c}, Generalized line graphs,
		J. Graph Theory.
		5:385-399(1981).
\bibitem{tech2} D.~Cvetkovi\'{c}, M.~Lepovi\'{c},
		P.~Rowlinson and S.~K.~Simi\'{c}, The maximal exceptional
		graphs, J.~Combin. Theory Ser. B 86 (2002), no.2, 347--363.
\bibitem{new} D.~Cvetkovi\'{c}, P.~Rowlinson, S.~K.~Simi\'{c},
		Spectral Generalizations of Line Graphs --- On Graphs
		with Least Eigenvalue $-2$,
		Cambridge Univ Press, 2004.
\bibitem{tech1} D.~Cvetkovi\'{c}, P.~Rowlinson, S.~K.~Simi\'{c},
		Constructions of the maximal exceptional graphs
			with largest degree 28,
		Department of Computing Science and Mathematics,
		University of Stirling,
		Scotland, Technical Report CSM-156, Stirling, 2000.
\bibitem{hoffman0} A.~J.~Hoffman, On graphs whose least eigenvalue
		exceeds $-1-\sqrt{2}$,
		Linear Algebra Appl. 16:153-165(1977).
\bibitem{exercises} L. Lov\'{a}sz, Combinatorial Problems and Exercises,
		2nd edition,
		North-Holland, 1993.
\bibitem{paperI} T.~Taniguchi, On graphs with the smallest eigenvalue
		at least $-1-\sqrt{2}$,
		part I,
		Ars Math,
		Contemp.1 (2008),
		no.1,
		81--98.
\bibitem{hlg} R.~Woo and A.~Neumaier,
		On graphs whose smallest eigenvalue is	at least $-1-\sqrt{2}$,
		Linear Algebra Appl. 226-228:577-591(1995).
\bibitem{MAGMA}  The Magma Computational Algebra System for Algebra,
		Number Theory and Geometry,
		URL: http://magma.maths.usyd.edu.au/magma/
\bibitem{bdm}B. D. McKay, Combinatorial Data,
		URL: http://cs.anu.edu.au/~bdm/data/graphs.html
\end{thebibliography}
\end{document}